\documentclass[11pt,reqno,letterpaper]{amsart}

\usepackage{amsmath}
\usepackage{amssymb}
\usepackage{amsthm}
\usepackage{empheq}
\usepackage{enumitem}
\usepackage[T1]{fontenc}
\usepackage{mathtools}
\usepackage{microtype}
\usepackage[dvipsnames]{xcolor}

\usepackage{tikz}
\usepackage[all,cmtip]{xy}

\usepackage[
	backend=biber,
	style=alphabetic,
	sorting=anyt,
	backref=true,
	date=year,
	doi=false,
	isbn=false,
	url=false
]{biblatex}

\renewbibmacro*{volume+number+eid}{%
  \printfield{volume}%
  \setunit*{\addnbthinspace}%
  \printfield{number}%
  \setunit{\addcomma\space}%
  \printfield{eid}}
\DeclareFieldFormat[article]{number}{\mkbibparens{#1}}
\DeclareFieldFormat[article]{volume}{\mkbibbold{#1}}
\usepackage{hyperref}
\usepackage[capitalise]{cleveref}

\hypersetup{
	colorlinks=true,
	linkcolor=Aquamarine,
	citecolor=purple,
	filecolor=cyan,
	urlcolor=teal
}

\theoremstyle{plain}
\newtheorem{theorem}{Theorem}[section]

\newtheorem{corollary}[theorem]{Corollary}

\newtheorem{lemma}[theorem]{Lemma}
\newtheorem{proposition}[theorem]{Proposition}

\theoremstyle{definition}
\newtheorem{definition}[theorem]{Definition}

\newtheorem{example}[theorem]{Example}

\newtheorem{construction}[theorem]{Construction}
\newtheorem*{ack}{Acknowledgement}

\theoremstyle{remark}
\newtheorem{claim}[theorem]{Claim}

\newtheorem{remark}[theorem]{Remark}
\newtheorem{setup}[theorem]{}
\newtheorem*{remark*}{Remark}
\newtheorem*{notation*}{Notation and Terminology}

\numberwithin{equation}{section}

\DeclareMathOperator{\Aut}{Aut}

\DeclareMathOperator{\Span}{Span}
\DeclareMathOperator{\Sym}{Sym}

\DeclareMathOperator{\length}{length}

\DeclareMathOperator{\rank}{rank}
\DeclareMathOperator{\rk}{rk}

\DeclareMathOperator{\GL}{GL}

\DeclareMathOperator{\coker}{coker}
\DeclareMathOperator{\Ext}{Ext}
\DeclareMathOperator{\Hom}{Hom}

\DeclareMathOperator{\id}{id}
\DeclareMathOperator{\im}{Im} 

\DeclareMathOperator{\Bs}{Bs}

\DeclareMathOperator{\Exc}{Exc}

\DeclareMathOperator{\Gr}{Gr}
\DeclareMathOperator{\Grass}{Grass}

\DeclareMathOperator{\Proj}{Proj}

\DeclareMathOperator{\Sing}{Sing}
\DeclareMathOperator{\Spec}{Spec} 
\DeclareMathOperator{\Supp}{Supp} 

\DeclareMathOperator{\ch}{ch}

\DeclareMathOperator{\cNE}{\overline{NE}} 

\DeclareMathOperator{\Hilb}{Hilb}
\DeclareMathOperator{\Quot}{Quot}

\DeclareMathOperator{\pt}{pt}
\newcommand{\Romannum}[1]{\uppercase\expandafter{\romannumeral #1}}

\newcommand{\hooklongrightarrow}{\lhook\joinrel\longrightarrow}
\renewcommand{\div}{\operatorname{div}}

\DeclarePairedDelimiter{\rdown}{\lfloor}{\rfloor}
\DeclarePairedDelimiter{\abs}{\lvert}{\rvert}

\DeclarePairedDelimiterX{\pair}[2]{\langle}{\rangle}{#1,#2}

\title{Sheaf stable pairs, Quot-schemes, and birational geometry}

\author{Caucher Birkar}
\address{Yau Mathematical Sciences Center,
Jingzhai, Tsinghua University, Haidian District,
Beijing, China, 100084}
\email{birkar@tsinghua.edu.cn}

\author{Jia Jia}
\address{Yau Mathematical Sciences Center,
Jingzhai, Tsinghua University, Haidian District,
Beijing, China, 100084}
\email{jia\_jia@u.nus.edu,mathjiajia@tsinghua.edu.cn}

\author{Artan Sheshmani}
\address{Beijing Institute of Mathematical Sciences and Applications,
No. 544, Hefangkou Village, Huaibei Town, Huairou District,
Beijing, China, 101408}
\email{artan@bimsa.cn}

\address{Massachusetts Institute of Technology (MIT), IAiFi Institute, 182 Memorial Drive, Cambridge, MA 02139, USA}
\email{artan@mit.edu}

\date{\today}


\begin{document}

\begin{abstract}
	In this paper we build bridges between moduli theory of sheaf stable pairs on one hand
	and birational geometry on the other hand.
	We will in particular treat moduli of sheaf stable pairs on smooth projective curves in detail
	and present some calculations in low degrees.
	We will also outline problems in various directions.
\end{abstract}

\maketitle

\setcounter{tocdepth}{1}
\tableofcontents

We work over an algebraically closed field \(k\) of characteristic \(0\).

\section{Introduction}

Given an algebraic variety, the Quot-schemes parameterise flat families of quotient sheaves
with fixed numerical characteristics,
for instance Hilbert polynomial on that variety.
Hilbert and Quot schemes owe their construction to Grothendieck in \cite{GrothIV},
and later they got further developed following results of Mumford, Altman and Kleiman.
Here is a short review:

Let \(X\) be a reduced connected projective scheme over \(k\),
equipped with a very ample line bundle \(\mathcal{O}_{X}(1)\),
and let \(P\in\mathbb Q[t]\) be a numerical polynomial with rational coefficients.
As the family of semi-stable coherent sheaves on \(X\) with Hilbert polynomial \(P\) is bounded
(e.g.\@ \cite[Theorem~3.3.7, p.~78]{huybrechts2010geometry}),
there is \(m\in\mathbb{N}\), such that any such sheaf \(\mathcal{F}\) is \(m\)-regular
(\cite[Theorem~1.13, p.~623]{Kleiman}).
From \(m\)-regularity of \(\mathcal{F}\) it follows that for all \(i\geq 0\),
the sheaf \(\mathcal{F}(i+m)\) is globally generated and
\begin{equation}\label{eq:surjectivity}
	H^0(X,\mathcal{O}_{X}(i))\otimes H^0(X,\mathcal{F}(m))\longrightarrow H^0(X,\mathcal{F}(i+m)))
\end{equation}
is surjective (\cite[p.~100]{Mumford}).
In particular we can find any such \(\mathcal{F}\) among quotients of \(\mathcal{O}_{X}^{\oplus n}(-m)\),
where \(n\coloneqq P(m)\).
So we need to consider the Quot-scheme \(\Quot(\mathcal{O}_{X}^{\oplus n}(-m), P)\)
(\cite[Th\'eor\'em~3.2, p.~260]{GrothIV}).
The kernels of the quotients \(\mathcal{O}_{X}^{\oplus n}(-m)\twoheadrightarrow\mathcal{F}\)
do not have to be globally generated,
however, the family of all these kernels is bounded
(\cite[Proposition~1.2, p.~252]{GrothIV}),
and hence \(m'\)-regular for some \(m'\in\mathbb{N}\).
Therefore there is \(p\in\mathbb{N}\) such that for all \(l\geq p\),
\(\mathcal{O}_{X}^{\oplus n}(-m)\) is \(l\)-regular as well as any \(\mathcal{F}\) as above,
and the kernel of \(\mathcal{O}_{X}^{\oplus n}(-m)\twoheadrightarrow\mathcal{F}\).
Then surjectivity of maps as in \eqref{eq:surjectivity} gives us for each \(l\geq p\) a realisation of
\(\Quot(\mathcal{O}_{X}^{\oplus n}(-m), P)\) as a closed subscheme
\begin{equation}
	\Quot(\mathcal{O}_{X}^{\oplus n}(-m), P)\hooklongrightarrow \Gr(N_{l}-P(l), N_{l}),
\end{equation}
where \(N_{l}\coloneqq nP_{\mathcal{O}_{X}}(l-m)\) and \(\Gr(N_{l}-P(l), N_{l})\) is the Grassmannian of
\(N_{l}-P(l)\)-dimensional subspaces in an \(N_{l}\)-dimensional space (\cite[Lemmes~3.3 and 3.7]{GrothIV}).

\medskip

The action of \(\GL_{n}(k)\) on \(\mathcal{O}_{X}^{\oplus n}\) induces an action of
\(\GL_{n}(k)\) on
\[
	H^0(X, \mathcal{O}_{X}^{\oplus n}(l-m))
\]
where each \(n\times n\) matrix becomes a matrix of \(P_{\mathcal{O}_{X}}(l-m)\times P_{\mathcal{O}_{X}}(l-m)\)
scalar matrices (of matrices entries).\footnote{In particular, scalar \(n\times n\)-matrices are mapped to scalar matrices.}
Thus we have a right\footnote{We regard global sections of \(\mathcal{O}_{X}^{\oplus n}\) as row vectors,
	i.e.\@ the action of \(\GL_{n}(k)\) is from the right.}
action of \(\GL_{n}(k)\) on \(\Gr(N_{l}-P(l), N_{l})\).
The Pl\"ucker embedding
\begin{equation}
	\Gr(N_{l}-P(l), N_{l})\hooklongrightarrow\mathbb{P}^{M_{l}},
	\quad M_{l}=\binom{N_l}{P(l)}-1
\end{equation}
comes with a \(\GL_{n}(k)\)-linearisation of the very ample line bundle,
that is induced from the canonical \(\GL_{N_{l}}(k)\)-linearisation.
As \(\Quot(\mathcal{O}_{X}^{\oplus n}(-m), P)\subseteq\Gr(N_{l}-P(l), N_{l})\) is \(\GL_{n}(k)\)-invariant,
the induced very ample line bundle on \(\Quot(\mathcal{O}_{X}^{\oplus n}(-m), P)\) is \(\GL_{n}(k)\)-linearised
(e.g.\@ \cite[p.~101]{huybrechts2010geometry}).
This linearised ample line bundle allows one to describe the moduli space of (semi) stable coherent sheaves,
as the GIT quotient of a locus of ``\emph{(semi) stable}'' quotient sheaves cut out in the Quot-scheme.
The moduli space of coherent sheaves became an instrumental tool to study many fundamental problems
in modern algebraic geometry.
In the 1990s Le Potier studied the moduli space of coherent systems \cite{Lepotier}.
These parameterise further, the information of pairs \((V,\mathcal{F})\) composed of coherent sheaf \(\mathcal{F}\)
with fixed numerical characteristics, together with a subspace \(V\) of its space of global sections,
that is, morphisms
\[
	V\otimes \mathcal{O}_{X}\to \mathcal{F}, \quad V\subset H^0(X, \mathcal{F}),
\]
equipped with a suitable notion of stability condition associated to \((V, \mathcal{F})\).

\medskip

The current article aims at studying a particular instance of coherent systems, known as stable pairs,
with support over a fixed algebraic variety and
explores the connections between the birational geometry of the underlying variety
and the associated moduli space of stable pairs.

\begin{definition}
	Let \(Z\) be an algebraic variety.
	In this paper, a \emph{sheaf stable pair} \(\mathcal{E},s\) on \(Z\) consists of
	a torsion-free coherent sheaf \(\mathcal{E}\) and a morphism
	\[
		\mathcal{O}_Z^r\xlongrightarrow{s}\mathcal{E}
	\]
	of sheaves (of \(\mathcal{O}_{Z}\)-modules) such that
	\[
		\dim\Supp\coker(s)<\dim Z,
	\]
    where \(r=\rank(\mathcal{E})\).
\end{definition}

For simplicity, we will usually drop ``sheaf'' and
just refer to \(\mathcal{E},s\) as a stable pair \cite{pandharipande2009curve,sheshmani2016higher,lin2018pair}.

\medskip

We say two stable pairs \(\mathcal{E},s\) and \(\mathcal{G},t\) are equivalent
if there is a commutative diagram
\[
	\xymatrix{
	\mathcal{O}_Z^r \ar[r]^s \ar@{=}[d]_{\id} & \mathcal{E} \ar[d]_{\rotatebox{90}{\(\sim\)}}^{\text{isomorphism}} \\
	\mathcal{O}_{Z}^r \ar[r]_t & \mathcal{G}.
	}
\]
The equivalence class of \(\mathcal{E},s\) is denoted by \([\mathcal{E},s]\).

However, above stable pairs should not be confused with
stable pairs studied in birational geometry.
In fact, we relate above stable pairs with stable minimal models
in birational geometry \cite{birkar2022moduli}.
It is this connection that inspired this work.

Much of this paper is devoted to understanding the moduli spaces \(M_Z(r,n)\)
of stable pairs classes \([\mathcal{E},s]\) on a smooth curve \(Z\)
where \(\mathcal{E}\) is of rank \(r\) (mostly \(r=2\))
and \(\deg\mathcal{E}=n\).
Here \(\deg\mathcal{E}\) is defined in terms of \(\coker(s)\).
This case already exhibits a rich geometry.
To such \([\mathcal{E},s]\) one can associate \(X=\mathbb{P}(\mathcal{E})\to Z\)
together with divisors \(D_1,\cdots,D_r,A\).
The structure
\[
	(X,D_1+\cdots+D_r),A \longrightarrow Z
\]
is a stable minimal model over the generic point of \(Z\) but often not over the whole \(Z\).
But a birational procedure produces a stable minimal model
\[
	(X',D_1'+\cdots+D_r'),A' \longrightarrow Z.
\]

It turns out that the moduli space of the initial stable pairs \([\mathcal{E},s]\)
parametrises the \textbf{procedure} of going
from \((X,D_1+D_2),A\) to \((X',D_1'+D_2'),A'\).
This geometric picture provides a crucial tool to study the above moduli spaces.

We can now state the first result of this paper for stable sheaves on curves.
\begin{theorem}\label{thm:main_rank_2_on_curves}
	Let \(Z\) be a smooth projective curve and \(n\) be a non-negative integer.
	Then
	\begin{enumerate}[label=(\arabic*),ref=(\arabic*)]
		\item \(M_Z(r,n)\) is a smooth projective variety.
		\item \label{enu:1.2_2}
		      Consider the natural morphism \(M_Z(r,n)\xlongrightarrow{\pi}\Hilb_Z^n\)
		      sending \([\mathcal{E},s]\) to the divisor of \(\coker(s)\).
		      The fibre of \(\pi\) over \(\sum_{1}^{\ell}n_{j}q_{j}\) is isomorphic to
		      \[
			      F_1 \times \dots \times F_{\ell}
		      \]
		      where \(F_{j}\) depends only on \(r\) and \(n_{j}\)
		      (so it is independent of \(Z\) and the choice of \(\sum_{1}^{\ell}n_{j}q_{j}\)).
		\item \(F_{j}\) in \ref{enu:1.2_2} is a normal variety of dimension \(n_j(r-1)\) with Cartier canonical divisor.\footnote{there are related works by Biswas, Gangopadhyay and Sebastian in \cite{gangopadhyay2020fundamental,biswas2024infinitesimal}.}
	\end{enumerate}
\end{theorem}

We have a more precise description of \(M_Z(2,n)\) and
the fibres of \(\pi\) in low degrees \(n\leq 3\).

\begin{theorem}\label{thm:main_m21}
	Let \(Z\) be a smooth projective curve.
	Then \(M_Z(2,1)\) is isomorphic to \(Z\times\mathbb{P}^1\).
\end{theorem}

\begin{theorem}\label{thm:main_m22}
	Let \(Z\) be a smooth projective curve.
	Then the fibre of
	\[
		\pi\colon M_Z(2,2)\longrightarrow\Hilb_Z^2
	\]
	over a point \(\sum_{1}^{\ell}n_{j}q_{j}\) is isomorphic to
	\[
		\begin{cases}
			\text{smooth quadric in } \mathbb{P}^3,   & \text{when } \ell=2, \\
			\text{singular quadric in } \mathbb{P}^3, & \text{when } \ell=1.
		\end{cases}
	\]
\end{theorem}

\begin{theorem}\label{thm:main_m23}
	Let \(Z\) be a smooth projective curve.
	Then the fibre of
	\[
		\pi\colon M_Z(2,3) \longrightarrow \Hilb_Z^3
	\]
	over a point \(\sum_{1}^{\ell}n_{j}q_{j}\) is isomorphic to
	\[
		\begin{cases}
			\mathbb{P}^1\times\mathbb{P}^1\times\mathbb{P}^1, & \text{when } \ell=3, \\
			\mathbb{P}^1\times \text{singular quadric},       & \text{when } \ell=2, \\
			F_3,                                              & \text{when } \ell=1
		\end{cases}
	\]
	where \(F_3\) is a \(\mathbb{Q}\)-factorial Fano \(3\)-fold of Picard number one
	with canonical singularities
	along a copy of \(\mathbb{P}^1\).
	Moreover, \(F_3\) is birational to \(\mathbb{P}^3\).
\end{theorem}

We will give an explicit construction of \(F_3\) from \(\mathbb{P}^3\).
\medskip

The present article is only the beginning of a long term project on sheaf stable pairs and birational geometry.
There are various directions to explore. We outline some of these:
\begin{itemize}
	\item Study \(M_Z(2,n)\) over \(Z=\mathbb{P}^1\) for degrees \(n\geq 4\);
	\item Study \(M_Z(r,n)\) over \(Z=\mathbb{P}^1\) for higher ranks \(r\geq 3\);
	\item Study \(M_Z(r,n)\) over curves \(Z\) of genus \(g(Z)\geq 1\);
	\item Study \(M_Z(\ch)\) over higher dimensional bases \(Z\) with fixed Chern character \(\ch\);
	\item Use techniques of enumerative geometry to get results in birational geometry.
\end{itemize}

Each direction exhibits its own challenges.
Overall, the above program will enrich both birational geometry and enumerative geometry.
We believe that deeper connections between the two fields become more apparent
when one studies the above moduli spaces over higher dimensional bases.
In this article, we mainly apply birational geometry to understand these moduli spaces over curves.
But in the higher dimensional case one might be able to go in the opposite direction as well
and relate invariants in the two fields.

Finally, for completeness of this discussion
we say a few words about our investigation of \(M_Z(2,2)\) over a smooth curve \(Z\).
Hope this helps to see how birational geometry comes into the picture.
Assume  \([\mathcal{E},s]\in M_Z(2,2)\).
As stated above, we can associate a model
\[
	X, D_1,D_2, A \longrightarrow Z
\]
and from this we can get a stable minimal model
\[
	(X',B'), A'\longrightarrow Z.
\]

There are two main cases to consider:
one is when the cokernel divisor is reduced and the other is when the cokernel divisor is non-reduced.

Assume first that \(\mathcal{E},s\) has reduced cokernel divisor \(Q=q_1+q_2\).
The two divisors \(D_1,D_2\) intersect at two distinct points, one over each \(q_i\).
To go from \(X\) to \(X'\) it is enough to blowup these points and then blow down two curves over the \(q_i\).
Already from this picture one can guess that the classes \(\mathcal{E},s\) with fixed \(q_1,q_2\)
are parametrised by \(\mathbb{P}^1\times \mathbb{P}^1\).

Now assume the cokernel divisor is non-reduced, say \(Q=2q\).
This is the more complicated case.
In this case, \(D_1\cdot D_2=2\) and the intersection points are over \(q\).
There are three subcases to be considered.

\textbf{Case \Romannum{1}:}
The fibre \(F\) over \(q\) is a component of both \(D_1,D_2\).
In this case, \(X=X'\) and \(D_{i}'\) is the horizontal part of \(D_{i}\) (similarly for \(A\)).

\textbf{Case \Romannum{2}:}
\(F\) is not a component of \(D_1,D_2,A\).
Then \(D_1\) and \(D_2\) are tangent to each other at some point over \(q\).
Then the stable minimal model is obtained as in the following picture:

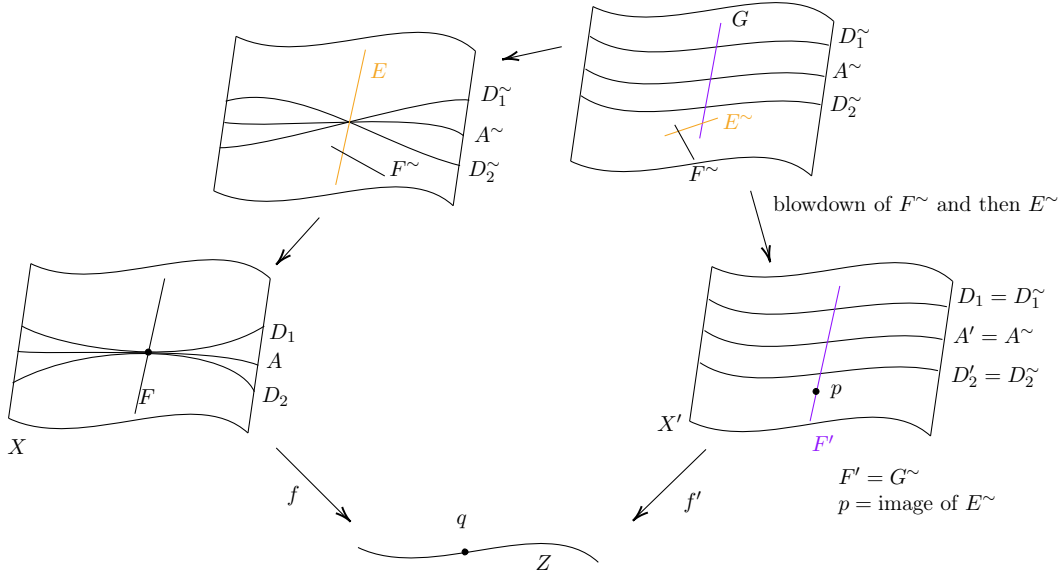
\begin{center}
	\begin{tikzpicture}[x=0.75pt,y=0.75pt,yscale=-0.75,xscale=0.75]
		\draw (396,3) .. controls (446,27) and (520.6,-21) .. (558,13);
		\draw (381,107) .. controls (431,131) and (505.6,83) .. (543,117);
		\draw (396,3) -- (381,107);
		\draw (558,13) -- (543,117);
		\draw (394,23) .. controls (433,49) and (503.6,17) .. (554.6,29);
		\draw (559.6,16) node [anchor=north west][inner sep=0.75pt] [xscale=0.75,yscale=0.75] {\(D_{1}^{\sim}\)};
		\draw (391,45) .. controls (429.6,70) and (500.6,39) .. (551.6,50);
		\draw (555.54,40) node [anchor=north west][inner sep=0.75pt] [xscale=0.75,yscale=0.75] {\(A^{\sim}\)};
		\draw (388,63) .. controls (426,89) and (498,58) .. (549,69);
		\draw (554,62) node [anchor=north west][inner sep=0.75pt] [xscale=0.75,yscale=0.75] {\(D_{2}^{\sim}\)};
		\draw [color={rgb, 255:red, 144; green, 19; blue, 254}, draw opacity=1]   (482,14) -- (467.76,91.7);
		\draw (488,5) node [anchor=north west][inner sep=0.75pt] [xscale=0.75,yscale=0.75] {\(G\)};
		\draw [color={rgb, 255:red, 245; green, 166; blue, 35}, draw opacity=1]   (480,78) -- (444.2,90);
		\draw (482,74) node [anchor=north west][inner sep=0.75pt] [color={rgb, 255:red, 245; green, 166; blue, 35 }, opacity=1, xscale=0.75, yscale=0.75] {\(E^{\sim}\)};
		\draw    (464,106) -- (451,82.58);
		\draw (459,109) node [anchor=north west][inner sep=0.75pt] [xscale=0.75,yscale=0.75] {\(F^{\sim}\)};

		\draw (1,293) node [anchor=north west][inner sep=0.75pt] [xscale=0.75,yscale=0.75] {\(X\)};
		\draw (18,176) .. controls (68,200) and (142,152) .. (179,186);
		\draw (3,280) .. controls (53,304) and (127,256) .. (164,290);
		\draw (18,176) -- (3,280);
		\draw (179,186) -- (164,290);
		\draw (5.71,256.29) .. controls (61.59,225.82) and (155,232.53) .. (168.31,262.2);
		\draw (177,215.5) node [anchor=north west][inner sep=0.75pt] [xscale=0.75,yscale=0.75] {\(D_{1}\)};
		\draw (12,218) .. controls (36.5,232) and (126.5,249) .. (175,218);
		\draw (172,258) node [anchor=north west][inner sep=0.75pt] [xscale=0.75,yscale=0.75] {\(D_{2}\)};
		\draw (9,235) .. controls (30,238) and (138,231) .. (171,244.4);
		\draw (175,235) node [anchor=north west][inner sep=0.75pt] [xscale=0.75,yscale=0.75] {\(A\)};
		\draw (108,186) -- (88,277);
		\draw (89,259) node [anchor=north west][inner sep=0.75pt] [xscale=0.75,yscale=0.75] {\(F\)};
		\draw  [fill={rgb, 255:red, 0; green, 0; blue, 0 }, fill opacity=1] (95,235.5) .. controls (95,234.4) and (96,233.5) .. (97,233.5) .. controls (98,233.5) and (99,234.4) .. (99,235.5) .. controls (99,237) and (98,237.5) .. (97,237.5) .. controls (96,237.5) and (95,237) .. (95,235.5) -- cycle;

		\draw (438,276) node [anchor=north west][inner sep=0.75pt] [xscale=0.75,yscale=0.75] {\(X'\)};
		\draw (476,178) .. controls (526,202) and (600.6,154) .. (638,188);
		\draw (461,282) .. controls (511,306) and (585.6,258) .. (623,292);
		\draw (476,178) -- (461,282);
		\draw (638,188) -- (623,292);
		\draw (474,200) .. controls (512,225) and (583,194) .. (634,205);
		\draw (639.6,191) node [anchor=north west][inner sep=0.75pt] [xscale=0.75,yscale=0.75] {\(D_{1} =D_{1}^{\sim}\)};
		\draw (471,221) .. controls (509,247) and (580,216) .. (631,227);
		\draw (637,217) node [anchor=north west][inner sep=0.75pt] [xscale=0.75,yscale=0.75] {\(A'=A^{\sim}\)};
		\draw (468,242.64) .. controls (506,268) and (577,236.89) .. (628,248);
		\draw (635.5,242.84) node [anchor=north west][inner sep=0.75pt] [xscale=0.75,yscale=0.75]	{\(D_{2}'=D_{2}^{\sim}\)};
		\draw [color={rgb, 255:red, 144; green, 19; blue, 254}, draw opacity=1]   (562,191) -- (542,282);
		\draw (541.87,289.84) node [anchor=north west][inner sep=0.75pt] [color={rgb, 255:red, 144; green, 19; blue, 254 }, opacity=1, xscale=0.75, yscale=0.75] {\(F'\)};
		\draw  [fill={rgb, 255:red, 0; green, 0; blue, 0}, fill opacity=1] (544,262) .. controls (544,261) and (545,260) .. (546,260) .. controls (547,260) and (548,261) .. (548,262) .. controls (548,263) and (547,263.97) .. (546,263.97) .. controls (545,263.97) and (544,263) .. (544,262) -- cycle;
		\draw (554.6,256) node [anchor=north west][inner sep=0.75pt] [xscale=0.75,yscale=0.75] {\(p\)};

		\draw (238,367) .. controls (288,391) and (363,343) .. (399.6,377);
		\draw (356,371) node [anchor=north west][inner sep=0.75pt] [xscale=0.75,yscale=0.75] {\(Z\)};
		\draw  [fill={rgb, 255:red, 0; green, 0; blue, 0}, fill opacity=1] (308,370) .. controls (308,369) and (309,368) .. (309.69,368) .. controls (310.8,368) and (312,369) .. (312,370) .. controls (312,371) and (310.8,372) .. (309.69,372) .. controls (309,372) and (308,371) .. (308,370) -- cycle;
		\draw (303,343.54) node [anchor=north west][inner sep=0.75pt] [xscale=0.75,yscale=0.75] {\(q\)};

		\draw (156,23) .. controls (206,47) and (280,-1) .. (317,33);
		\draw (141,127) .. controls (191,151) and (265,103) .. (302,137);
		\draw (156,23) -- (141,127);
		\draw (317,33) -- (302,137);
		\draw (145,98) .. controls (181,98) and (286.87,62) .. (312.87,66);
		\draw (319,55) node [anchor=north west][inner sep=0.75pt] [xscale=0.75,yscale=0.75] {\(D_{1}^{\sim}\)};
		\draw (148.2,81) .. controls (185,86) and (289,72) .. (309,90);
		\draw (314.8,83) node [anchor=north west][inner sep=0.75pt] [xscale=0.75,yscale=0.75] {\(A^{\sim}\)};
		\draw (149,66.7) .. controls (197,53) and (254.64,97) .. (306.64,110);
		\draw (310.8,106) node [anchor=north west][inner sep=0.75pt] [xscale=0.75,yscale=0.75] {\(D_{2}^{\sim}\)};
		\draw [color={rgb, 255:red, 245; green, 166; blue, 35}, draw opacity=1]   (243,32) -- (223,123);
		\draw (245,40) node [anchor=north west][inner sep=0.75pt] [color={rgb, 255:red, 245; green, 166; blue, 35 }, opacity=1, xscale=0.75, yscale=0.75] {\(E\)};
		\draw (220,97) -- (256,117);
		\draw (258,106) node [anchor=north west][inner sep=0.75pt] [xscale=0.75,yscale=0.75] {\(F^{\sim}\)};

		\draw (183,300) -- (231,348);
		\draw [shift={(233,349.64)}, rotate = 225.1] [color={rgb, 255:red, 0; green, 0; blue, 0}][line width=0.75] (11,-3) .. controls (7,-1) and (3,0) .. (0,0) .. controls (3,0) and (7,1) .. (11,3);
		\draw (188.93,324.54) node [anchor=north west][inner sep=0.75pt] [xscale=0.75,yscale=0.75] {\(f\)};
		\draw (502,126.97) -- (515,173);
		\draw [shift={(515.2,175)}, rotate = 254] [color={rgb, 255:red, 0; green, 0; blue, 0}][line width=0.75] (11,-3) .. controls (7,-1) and (3,0) .. (0,0) .. controls (3,0) and (7,1) .. (11,3);
		\draw (516,129) node [anchor=north west][inner sep=0.75pt] [xscale=0.75,yscale=0.75] {blowdown of \(F^{\sim}\) and then \(E^{\sim}\)};
		\draw (212,145) -- (185,175);
		\draw [shift={(183,176.64)}, rotate = 312] [color={rgb, 255:red, 0; green, 0; blue, 0}][line width=0.75] (11,-3) .. controls (7,-1) and (3,0) .. (0,0) .. controls (3,0) and (7,1) .. (11,3);
		\draw (472,301) -- (424,347.57);
		\draw [shift={(422.93,349)}, rotate = 316] [color={rgb, 255:red, 0; green, 0; blue, 0}][line width=0.75] (11,-3) .. controls (7,-1) and (3,0) .. (0,0) .. controls (3,0) and (7,1) .. (11,3);
		\draw (455.6,327.87) node [anchor=north west][inner sep=0.75pt] [xscale=0.75,yscale=0.75] {\(f'\)};
		\draw (375,30) -- (337,38.2);
		\draw [shift={(335,38.64)}, rotate = 347] [color={rgb, 255:red, 0; green, 0; blue, 0}][line width=0.75] (11,-3) .. controls (7,-1) and (3,0) .. (0,0) .. controls (3,0) and (7,1) .. (11,3);

		\draw (553,311) node [anchor=north west][inner sep=0.75pt] [xscale=0.75,yscale=0.75] {$ \begin{array}{l}
					F'=G^{\sim} \\
					p=\text{image of } E^{\sim}
				\end{array}$};
	\end{tikzpicture}
\end{center}

\textbf{Case \Romannum{3}:}
\(F\) is a component of one of \(D_1,D_2,A\).
Say \(F\) is a component of \(D_1\).
Then \(D_2,A\) are tangent, and
the stable minimal model is obtained by a similar but slightly different process.
One then considers the case when \(F\) is a component of \(D_2\) (resp.\ \(A\)), etc.

\medskip
The above arguments make it clear that the fibres of \(M_Z(2,2)\to \Hilb_Z^2\) are independent of the genus of \(Z\),
that is, the fibre only depends on whether the cokernel divisor is reduced or not
(and the same arguments apply even if \(Z\) is not projective).
In fact, it is enough to work in a formal neighbourhood of the cokernel divisor.
From this one can reduce the calculation of the fibres to a local problem.

To make the story short, in the non-reduced cokernel case one is reduced to classifying all quotients
\[
	k[t]/\langle t^2 \rangle\oplus k[t]/\langle t^2 \rangle \longrightarrow L
\]
where \(L\) is a \(k[t]/\langle t^2 \rangle\)-module of length \(2\).
Such \(L\) is either \(k\oplus k\) or \(k[t]/\langle t^2\rangle\).
Some careful calculations show that such quotients are parametrised by a singular quadric in \(\mathbb{P}^3\).

But to investigate fibres in the degree 3 case we need to borrow more sophisticated tools from birational geometry.

\begin{ack}
	The first author is supported by a grant from Tsinghua University and
	a grant of the National Program of Overseas High Level Talent.
	The second author is supported by Shuimu Tsinghua Scholar Program and
	China Postdoctoral Science Foundation (2023TQ0172).
	The third author would like to thank Vladimir Baranovsky for helpful discussions. He is supported by a grant from Beijing Institute of Mathematical Sciences and Applications (BIMSA),
	and by the National Program of Overseas High Level Talent.
\end{ack}

\section{Preliminaries}

We work over an algebraically closed field \(k\) of characteristic zero.
All varieties and schemes are defined over \(k\) unless stated otherwise.
Varieties are assumed to be irreducible.

\subsection{Contractions}

By a \emph{contraction} we mean a projective morphism \(f\colon X\to Y\) of schemes such that
\(f_{*}\mathcal{O}_{X}=\mathcal{O}_{Y}\)
In particular, \(f\) is surjective and has connected fibres.

We say that a birational map \(f\colon X\dashrightarrow Y\) \emph{contracts} a divisor \(D\subset X\)
if \(f\) is defined at the generic point of \(D\) and \(f(D)\subset Y\) has codimension at least \(2\).
The map \(f\) is called a \emph{birational contraction} if \(f^{-1}\) does not contract any divisor.

\subsection{Pairs and Singularities}

Let \(X\) be a pure dimensional scheme of finite type over \(k\) and
let \(M\) be a \(\mathbb{Q}\)-divisor on \(X\).
We denote the coefficient of a prime divisor \(D\) in \(M\) by \(\mu_DM\).

A \emph{pair} \((X,B)\) consists of a normal quasi-projective variety \(X\)
and a \(\mathbb{Q}\)-divisor \(B\geq 0\) such that \(K_X+B\) is \(\mathbb{Q}\)-Cartier.
We call \(B\) the \emph{boundary divisor}.

Let \(\phi\colon W\to X\) be a log resolution of a pair \((X,B)\).
Let \(K_W+B_W\) be the pullback of \(K_X+B\).
The \emph{log discrepancy} of a prime divisor \(D\) on \(W\)
with respect to \((X,B)\) is defined as
\[
	a(D,X,B)\coloneqq 1-\mu_D B_W.
\]
A \emph{non-klt} place of \((X,B)\) is a prime divisor \(D\) over \(X\),
that is, on birational models of \(X\), such that \(a(D,X,B)\leq 0\),
and a \emph{non-klt centre} is the image of such a \(D\) on \(X\).
We say \((X,B)\) is lc (resp.\ klt) if \(a(D,X,B)\) is \(\geq 0\) (resp.\ \(>0\)) for every \(D\).
This means that every coefficient of \(B_W\) is  \(\leq 1\) (resp.\ \(<1\)).

A \emph{log smooth} pair is a pair \((X,B)\) where \(X\) is smooth and
\(\Supp B\) has simple normal crossing singularities.

\subsection{Base locus}

Let \(\mathcal{L}\) be an invertible sheaf on a scheme \(X\).
The \emph{base locus} of \(\mathcal{L}\) is
\[
	\Bs(\mathcal{L})=\{x\in X \mid t(x)=0, \forall t\in H^{0}(X,\mathcal{L})\}.
\]
If \(H^{0}(X,\mathcal{L})=0\), by convention, \(\Bs(\mathcal{L})=X\).

\subsection{Types of Models}

Let \((X,B)\) be an lc pair.

Let \(X\to Z\) be a contraction to a normal variety and
assume that \(K_X+B\) is big over \(Z\).
We say that the \emph{log canonical model} of \((X,B)\) over \(Z\) exists
if there is a birational contraction \(\varphi\colon X\dashrightarrow Y\) where
\begin{itemize}
	\item \(Y\) is normal and projective over \(Z\);
	\item \(K_Y+B_Y\coloneqq \varphi_*(K_X+B)\) is ample over \(Z\); and
	\item \(\alpha^*(K_X+B)\geq \beta^*(K_Y+B_Y)\) for any common resolution
	      \[
		      \xymatrix{
			      & W \ar[ld]_{\alpha} \ar[rd]^{\beta} \\
			      X \ar@{-->}[rr]_{\varphi} & & Y.
		      }
	      \]
\end{itemize}
We call \((Y,B_Y)\) the log canonical model of \((X,B)\) over \(Z\).

\medskip

A \emph{dlt model} of an lc pair \((X,B)\) is a pair \((X',B')\)
with a projective birational morphism \(\psi\colon X'\to X\)
such that
\begin{itemize}
	\item \((X',B')\) is dlt;
	\item every exceptional prime divisor of \(\psi\) appears in \(B'\) with coefficients one;
	\item \(K_{X'}+B'=\psi^{*}(K_X+B)\).
\end{itemize}

\medskip

We say that a pair \((Y,B_Y)\) is a \emph{log minimal model} of an lc pair \((X,B)\) if
there exists a birational contraction \(\phi\colon X\dashrightarrow Y\) such that
\begin{itemize}
	\item \((Y,B_Y)\) is \(\mathbb{Q}\)-factorial dlt;
	\item \(K_Y+B_Y\) is nef; and
	\item \(\phi\) is \(K_X+B\)-negative,
	      i.e., for any prime divisor \(D\) on \(X\) which is exceptional over \(Y\),
	      we have \(a(D,X,B)<a(D,Y,B_Y)\).
\end{itemize}
A log minimal model \((Y,B_Y)\) is \emph{good} if \(K_Y+B_Y\) is semi-ample.

\subsection{Stratification}

Let \(X\) be a scheme.
A \emph{stratification} of \(X\) consists of a set of finitely many locally closed subschemes
\(X_{1}, \dots, X_{n}\) of \(X\), called strata, pairwise disjoint and
such that \(X=\bigcup_1^n X_i\),
i.e., such that we have a surjective morphism \(\coprod_1^n X_i\to X\).

\section{Higher rank sheaf stable pairs}

\subsection{Stability}

Let \(X\) be a projective variety of dimension \(d\),
let \(\mathcal{O}_{X}(1)\) be a very ample invertible sheaf on \(X\).
Denote by \(P_{\mathcal{F}}\) the \emph{Hilbert polynomial} of a coherent sheaf \(\mathcal{F}\) on \(X\).

\begin{definition}[{\cite[Definition~2.6]{sheshmani2016higher}}]\label{def:tau_stable}
	Let \(q(m)\) be given by a polynomial with rational coefficients
	such that its leading coefficient is positive.
	A pair \(\mathcal{O}_{X}^r\xrightarrow{\phi}\mathcal{F}\), where \(\mathcal{F}\) is a pure sheaf,
	is \emph{\(\tau'\)-stable} (resp.\ \emph{\(\tau'\)-semi-stable})
	with respect to (stability parameter) \(q(m)\) if
	\begin{enumerate}[leftmargin=*]
		\item for all proper non-zero subsheaves \(\mathcal{G}\subseteq\mathcal{F}\) for which
		      \(\phi\) does not factor through \(\mathcal{G}\) we have
		      \[
			      \frac{P_{\mathcal{G}}}{\rk(\mathcal{G})} < \frac{P_{\mathcal{F}}+q(m)}{\rk(\mathcal{F})}, \quad \text{resp.\ } (\leq)
		      \]
		\item for all proper subsheaves \(\mathcal{G}\subseteq\mathcal{F}\) for which
		      \(\phi\) factors through
		      \[
			      \frac{P_{\mathcal{G}}+q(m)}{\rk(\mathcal{G})}<\frac{P_{\mathcal{F}}+q(m)}{\rk(\mathcal{F})}, \quad \text{resp.\ } (\leq).
		      \]
	\end{enumerate}
\end{definition}

We consider the stability condition when \(q(m)\to \infty\).

\begin{definition}
	Fix \(q(m)\) to be given as a polynomial of degree at least \(d+1\) with rational coefficients
	such that its leading coefficient is positive.
	A pair \(\mathcal{O}_X^r\xrightarrow{\phi}\mathcal{F}\) is called to be \emph{\(\tau'\)-limit-stable}
	(resp.\ \(\tau'\)-limit-semi-stable)
	if it is stable (resp.\ semi-stable) in the sense of \cref{def:tau_stable}
	with respect to this fixed choice of \(q(m)\).
\end{definition}

\begin{lemma}[{\cite[Lemma~2.7]{sheshmani2016higher}}]\label{lem:stable_equivalent_condition}
	Fix \(q(m)\) to be given as a polynomial of at least degree \(d+1\) with rational coefficients
	such that its leading coefficient is positive.
	Then stability and semi-stability coincide.
	A pair \(\mathcal{O}_X^r\xrightarrow{\phi}\mathcal{F}\) is \(\tau'\)-limit-stable
	if and only if \(\coker(\phi)\) is a sheaf with at most \(d-1\)-dimensional support,
	i.e., \(\coker(\phi)\) is a torsion sheaf.
\end{lemma}

\begin{proof}
	The exact sequence
	\[
		0\longrightarrow \mathcal{K}\coloneqq \ker(\phi)\longrightarrow \mathcal{O}_{X}^{r}\xlongrightarrow{\phi}
		\mathcal{F}\longrightarrow \mathcal{Q}\coloneqq \coker(\phi)\longrightarrow 0
	\]
	induces a short exact sequence
	\[
		0\longrightarrow \im(\phi)\longrightarrow \mathcal{F}\longrightarrow \mathcal{Q}\longrightarrow 0.
	\]
	Hence we obtain a commutative diagram
	\[
		\xymatrix{
			\mathcal{O}_{X}^{r} \ar[r]^{\phi} \ar@{=}[d] & \im(\phi) \ar@{^{(}->}[d] \\
			\mathcal{O}_{X}^{r} \ar[r] & \mathcal{F}.
		}
	\]
	Now we assume that \(\mathcal{O}_{X}^{r}\to\mathcal{F}\) is \(\tau'\)-limit-stable:
	\[
		\frac{P_{\im(\phi)}+q(m)}{\rk(\im(\phi))} < \frac{P_{\mathcal{F}}+q(m)}{\rk(\mathcal{F})}.
	\]
	By rearrangement we get
	\begin{equation}\label{eq:rearrangement}
		q(m)\cdot (\rk(\mathcal{F})-\rk(\im(\phi))) < \rk(\im(\phi))P_{\mathcal{F}}-\rk(\mathcal{F})P_{\im(\phi)}.
	\end{equation}
	Note that the right-hand side of \eqref{eq:rearrangement} is a polynomial in \(m\)
	of degree at most \(d\).
	However by the choice of \(q(m)\), it is a polynomial of degree at least \(d+1\) with positive leading coefficients.
	Hence, as \(m\to \infty\),
	the only way for the inequality \eqref{eq:rearrangement} to hold is \(\rk(\im(\phi))=\rk(\mathcal{F})\) and
	therefore \(\mathcal{Q}\) has rank zero.

	For the other direction,
	we assume that \(\mathcal{Q}\) is a torsion sheaf,
	but \(\mathcal{O}_X^r\to \mathcal{F}\) is not \(\tau'\)-limit-stable.
	Then there exists a \emph{saturated} subsheaf \(\mathcal{G}\) satisfying the destabilising condition:
	\[
		\frac{P_{\mathcal{G}}}{\rk(\mathcal{G})} \geq \frac{P_{\mathcal{F}}+q(m)}{\rk(\mathcal{F})},
	\]
	noting that \(\im(\phi)\not\subseteq \mathcal{G}\).
	Since \(\mathcal{F}\) is pure one has \(\rk(\mathcal{G})>0\).
	But the degree of \(q(m)\) is chosen to be sufficiently large, a contradiction.
\end{proof}

Now we study automorphisms of stable pairs.

\begin{lemma}[{\cite[Lemma~3.6]{sheshmani2016higher}}]\label{lem:auto_stable_pair}
	Given a \(\tau'\)-limit-stable pair \(\mathcal{O}_{X}^{r}\xrightarrow{\phi}\mathcal{F}\)
	and a commutative diagram
	\begin{equation}\label{eq:stable_identity}
		\xymatrix{
			\mathcal{O}_{X}^{r} \ar[r]^{\phi} \ar@{=}[d]_{\id} & \mathcal{F} \ar[d]^{\rho} \\
			\mathcal{O}_{X}^{r} \ar[r]_{\phi} & \mathcal{F}.
		}
	\end{equation}
	The map \(\rho\) is given by \(\id_{\mathcal{F}}\).
\end{lemma}

\begin{proof}
	The diagram \eqref{eq:stable_identity} induces
	\[
		\xymatrix@C=3em{
		\mathcal{O}_{X}^{r} \ar@{->>}[r]^-{\phi} \ar@{=}[d]_{\id} & \im(\phi) \ar@{^{(}->}[r] \ar[d]^{\rho|_{\im(\phi)}} & \mathcal{F} \ar[d]^{\rho} \\
		\mathcal{O}_{X}^{r} \ar@{->>}[r]^-{\phi} & \im(\phi) \ar@{^{(}->}[r] & \mathcal{F}.
		}
	\]
	By commutativity of \eqref{eq:stable_identity}, \(\rho\circ\phi=\phi\circ\id=\phi\) then
	\(\rho(\im(\phi))=\im(\phi)\).
	Hence \(\rho(\im(\phi))\subseteq \im(\phi)\).
	It follows that \(\rho\mid _{\im(\phi)}=\id_{\im(\phi)}\).
	Indeed, if \(s\in \im(\phi)(U)\) where \(U\subseteq X\) is affine open
	with \(\widetilde{s}\in \mathcal{O}_{X}^{r}(U)\) satisfying \(\phi(\widetilde{s})=s\)
	then
	\[
		\rho(s)=\rho(\phi(\widetilde{s}))=\phi(\id(\widetilde{s}))=\phi(\widetilde{s})=s.
	\]
	Now apply \(\Hom(-,\mathcal{F})\) to the short exact sequence
	\[
		0\longrightarrow \im(\phi)\longrightarrow \mathcal{F}\longrightarrow \mathcal{Q}\longrightarrow 0
	\]
	where \(\mathcal{Q}\) is the corresponding cokernel.
	We obtain
	\[
		0\longrightarrow \Hom(\mathcal{Q},\mathcal{F})\longrightarrow \Hom(\mathcal{F},\mathcal{F})\longrightarrow \Hom(\im(\phi),\mathcal{F}).
	\]
	Since \(\mathcal{O}_{X}^{r}\to \mathcal{F}\) is \(\tau'\)-limit-stable,
	\(\mathcal{Q}\)	is a torsion sheaf.
	Hence by purity of \(\mathcal{F}\), \(\Hom(\mathcal{Q},\mathcal{F})=0\).
	We obtain an injection
	\[
		\Hom(\mathcal{F},\mathcal{F})\hooklongrightarrow \Hom(\im(\phi),\mathcal{F}).
	\]
	Now \(\rho|_{\im(\phi)}=\id_{\im(\phi)}=(\id_{\mathcal{F}})|_{\im(\phi)}\),
	so \(\rho=\id_{\mathcal{F}}\).
\end{proof}

\subsection{Sheaf stable pairs}

In this subsection,
we establish some general results regarding sheaf stable pairs on varieties.
Inspired by \cref{lem:stable_equivalent_condition},
we propose the following definition of such pairs.

\begin{definition}
	Assume \(Z\) is a variety.
	A \emph{sheaf stable pairs} is of the form \(\mathcal{E},s\) where
	\[
		\begin{cases}
			\mathcal{E}\text{ is a torsion-free coherent sheaf of rank } r>0, \\
			\mathcal{O}_{Z}^r\xlongrightarrow{s}\mathcal{E}\text{ is a morphism of }
			\mathcal{O}_{Z}\text{-modules},                                   \\
			\dim\Supp\coker(s)<\dim Z.
		\end{cases}
	\]
\end{definition}

To give a morphism \(\mathcal{O}_{Z}^r\xlongrightarrow{s}\mathcal{E}\) is the same as
giving \(r\) sections
\[
	s_{1}, \dots, s_{r}\in H^{0}(Z,\mathcal{E})
\]
where \(s_{i}\) corresponds to the morphism from the \(i\)-th summand
of \(\mathcal{O}_{Z}^r\) to \(\mathcal{E}\)
determined by \(s\).
We will then sometimes use the notation \(\mathcal{E},s_{1},\dots,s_{r}\)
instead of \(\mathcal{E},s\).

We also often denote \(\mathcal{Q}\coloneqq\coker(s)\).
So we get an exact sequence
\[
	\mathcal{O}_{Z}^r\xlongrightarrow{s}\mathcal{E}\longrightarrow
	\mathcal{Q}=\coker(s)\longrightarrow 0.
\]

\begin{lemma}
	Assume \(\mathcal{E},s\) is a stable pair.
	Then the morphism \(\mathcal{O}_{Z}^r\xlongrightarrow{s}\mathcal{E}\) is injective.
\end{lemma}

\begin{proof}
	Since \(\dim\Supp\coker(s)<\dim Z\),
	\(s\) is generically an isomorphism, so \(\ker(s)\) is torsion, hence zero.
	So we get a short exact sequence
	\[
		0\longrightarrow\mathcal{O}_{Z}^r\xlongrightarrow{s}\mathcal{E}\longrightarrow
		\mathcal{Q}\longrightarrow 0.
		\qedhere
	\]
\end{proof}

\begin{lemma}
	Assume \(\mathcal{E},s\) is a stable pair on a normal variety \(Z\).
	Then \(\Supp\coker(s)\) is empty or of pure codimension one.
\end{lemma}

\begin{proof}
	Let \(\mathcal{Q}=\coker(s)\).
	If \(\mathcal{Q}=0\), then \(\Supp\mathcal{Q}=\emptyset\).

	Suppose \(\mathcal{Q}\neq 0\).
	Assume \(\Supp\mathcal{Q}\) is not of pure codimension one.
	Then after shrinking we can assume \(0<\dim\Supp\mathcal{Q}\leq\dim Z-2\).
	Now we have a diagram
	\[
		\xymatrix{
		\mathcal{O}_{Z}^r \ar[r]^s \ar[d]_{\text{isomorphism}} & \mathcal{E} \ar[d] \\
		(\mathcal{O}_{Z}^r)^{\vee\vee} \ar[r] & \mathcal{E}^{\vee\vee}
		}
	\]
	where \(^{\vee\vee}\) denotes double dual.
	All the maps are isomorphisms on \(U=Z\setminus\Supp\mathcal{Q}\).
	But \(\mathcal{E}^{\vee\vee}\) is reflexive (\cite[Corollary 1.2]{hartshorne1980stable})
	and \(Z\) is normal,
	hence \(\mathcal{E}^{\vee\vee}\) is determined by
	\(\mathcal{E}^{\vee\vee}|_U\simeq\mathcal{O}_{Z}^r|_U\)
	(\cite[Proposition 1.6]{hartshorne1980stable}),
	so \(\mathcal{O}_{Z}^r\to\mathcal{E}^{\vee\vee}\) is an isomorphism.
	On the other hand, \(\mathcal{E}\) is torsion-free by assumption,
	so the natural morphism \(\mathcal{E}\to\mathcal{E}^{\vee\vee}\) is injective.
	Therefore,
	\[
		\mathcal{O}_{Z}^r\longrightarrow\mathcal{E}\longrightarrow\mathcal{E}^{\vee\vee}
	\]
	are isomorphisms,
	contradicting the assumption \(\mathcal{Q}\neq 0\).
\end{proof}


\begin{lemma}\label{lem:locally_free_bs_vertical}
	Assume \(\mathcal{E},s\) is a stable pair on a variety \(Z\).
	Assume \(\mathcal{E}\) is locally free, and
	\[
		X=\mathbb{P}(\mathcal{E})\xlongrightarrow{f} Z, \quad \mathcal{O}_{X}(1)
	\]
	the associated projection and line bundle.
	Then the base locus \(\Bs(\mathcal{O}_{X}(1))\) is vertical over \(Z\).
	More precisely,
	\[
		f(\Bs(\mathcal{O}_{X}(1))) \subseteq \Supp\coker(s).
	\]
\end{lemma}

\begin{proof}
	Pulling back \(\mathcal{O}_{Z}^r \xlongrightarrow{s} \mathcal{E}\)
	to \(X\) we get
	\begin{equation}\label{eq:natural_surjection}
		\mathcal{O}_{X}^r \longrightarrow f^{*}\mathcal{E} \longrightarrow \mathcal{O}_{X}(1)
	\end{equation}
	where the second morphism is a natural surjective morphism
	\cite[Chapter~II, Proposition~7.11]{hartshorne1977algebraic}.
	Since \(\mathcal{O}_{Z}^r \xlongrightarrow{s} \mathcal{E}\)
	is surjective on
	\[
		U\coloneqq Z\setminus \Supp\coker(s)
	\]
	the induced morphism \(\mathcal{O}_{X}^r \longrightarrow \mathcal{O}_{X}(1)\)
	is surjective on \(f^{-1}U\).
	So \(\mathcal{O}_{X}(1)\) is generated by global sections on \(f^{-1}U\),
	hence \(\Bs(\mathcal{O}_{X}(1))\subseteq X\setminus f^{-1}U\).
\end{proof}

\begin{corollary}\label{cor:stable_pair_nef}
	Assume \(\mathcal{E},s\) is a stable pair on a smooth projective curve \(Z\),
	then the associated \(\mathcal{O}_{X}(1)\) is nef (i.e., \(\mathcal{E}\) is nef).
\end{corollary}

\begin{proof}
	Recall that an invertible sheaf \(\mathcal{L}\) on a projective scheme \(Y\) is nef
	if \(\deg \mathcal{L}|_C\geq 0\) for every curve \(C\subseteq Y\).

	In our case,
	\[
		0=\dim\Supp\coker(s)<\dim Z=1,
	\]
	so by \cref{lem:locally_free_bs_vertical},
	\[
		\Bs(\mathcal{O}_{X}(1)) \subseteq \text{union of fibres of } f.
	\]
	Thus if \(C\subseteq X\) is a curve with \(\deg\mathcal{O}_{X}(1)|_C<0\),
	then \(C\subseteq \Bs(\mathcal{O}_{X}(1)) \subseteq \text{union of fibres of } f\).
	But \(\mathcal{O}_{X}(1)\) is ample over \(Z\), so there is no such \(C\).
	This shows that \(\mathcal{O}_{X}(1)\) is nef.
\end{proof}

\begin{example}
	If \(Z\) is a smooth projective curve and
	\(\mathcal{E}=\mathcal{O}_{Z}(n_1)\oplus \dots\oplus \mathcal{O}_{Z}(n_r),s\) is stable,
	then \(n_{i}\geq 0\) for every \(i\).

	If not, say, \(n_1<0\),
	then \(\mathcal{O}_{X}(1)|_S\) is anti-ample
	where \(S\hookrightarrow X\) is the section determined by \(\mathcal{O}_{Z}(n_1)\).
	But then \(\mathcal{O}_{X}(1)\) is not nef, a contradiction.
	See also \cite[Proposition~6.1.2]{lazarsfeld2004positivitya}.

	In particular, any stable pair on \(\mathbb{P}^1\) is of the form above.
\end{example}

\begin{example}
	In general,
	\(\mathcal{O}_{X}(1)\) may not be base point free.
	Indeed,
	assume
	\[
		\begin{array}{l}
			Z = \text{elliptic curve}, \vspace{0.2em} \\
			\mathcal{E} = \mathcal{O}_{Z} \oplus \mathcal{O}_{Z}(p), \text{ where } p \text{ is a point on } Z.
		\end{array}
	\]
	The identity morphism \(\mathcal{O}_{Z}\to \mathcal{O}_{Z}\) and
	the morphism \(\mathcal{O}_{Z}\to \mathcal{O}_{Z}(p)\) corresponding to
	any non-zero section determine a stable pair \(\mathcal{E},s\)
	with \(\coker(s)\) supported at \(p\).
	Now \(X=\mathbb{P}_Z(\mathcal{E})\) is a ruled surface over \(Z\).
	Consider the section \(S\subseteq X\) given by the surjection
	\(\mathcal{E}\to\mathcal{O}_{Z}(p)\)
	\cite[Chapter~V, Proposition~2.6]{hartshorne1977algebraic}.
	Then \(\mathcal{O}_{X}(1)|_S\) is isomorphic to \(\mathcal{O}_{Z}(p)\)
	which is not base point free,
	hence \(\mathcal{O}_{X}(1)\) is not base point free.
\end{example}

\begin{remark}\label{rem:rational_map}
	Assume \(\mathcal{E},s\) is a stable pair on a variety \(Z\).
	Then the morphism \(\mathcal{O}_{Z}^r\xlongrightarrow{s}\mathcal{E}\)
	determines a rational map
	\[
		\xymatrix{
		X=\mathbb{P}(\mathcal{E}) \ar@{-->}[rr] \ar[rd] & & \mathbb{P}_Z^{r-1} \ar[ld] \\
		& Z
		}
	\]
	Indeed,
	the summands \(\mathcal{O}_{Z}\) of \(\mathcal{O}_{Z}^r\)
	determine sections \(s_{1},\dots,s_{r}\) of \(\mathcal{E}\)
	which generate \(\mathcal{E}\) outside \(\Supp\coker(s)\).
	We can view \(s_{1}, \dots, s_{r}\) as sections of \(\mathcal{O}_{X}(1)\) generating it outside
	\[
		f^{-1}\Supp\coker(s).
	\]
	Therefore,
	they determine a rational map as in the diagram above
	which is a morphism outside \(f^{-1}\Supp\coker(s)\).

	Alternatively, the morphism \(\mathcal{O}_{Z}^r\to \mathcal{E}\)
	is an isomorphism outside \(\Supp\coker(s)\)
	hence determines the rational map above.
\end{remark}

\begin{definition}
	Assume \(Z\) is a variety,
	\(\mathcal{E},s\) and \(\mathcal{G},t\) stable pairs.
	We say the pairs are \emph{equivalent} if there is a commutative diagram
	\[
		\xymatrix{
		\mathcal{O}_Z^r \ar[r]^s \ar@{=}[d]_{\id} & \mathcal{E} \ar[d]_{\rotatebox{90}{\(\sim\)}}^{\text{isomorphism}} \\
		\mathcal{O}_{Z}^r \ar[r]_t & \mathcal{G}
		}
	\]
	where the vertical arrows are isomorphisms
	and the left one is the identity.
	The equivalence class of \(\mathcal{E},s\) is denoted by \([\mathcal{E},s]\).

	If \(s_{1},\dots,s_{r}\in H^{0}(Z,\mathcal{E})\) and
	\(t_{1},\dots,t_{r}\in H^{0}(Z,\mathcal{G})\)
	are the sections of \(\mathcal{E},\mathcal{G}\) determined by \(s,t\),
	then the above is equivalent to saying that
	there exists an isomorphism \(\mathcal{E}\to\mathcal{G}\)
	sending \(s_{i}\) to \(t_{i}\) for every \(i\).
\end{definition}

\begin{lemma}\label{lem:scaling_matrix_equivalent}
	Assume \(Z\) is a variety,
	\(\mathcal{E},s\) and \(\mathcal{G},t\) are stable pairs.
	Assume that there is a commutative diagram
	\[
		\xymatrix{
		\mathcal{O}_Z^r \ar[r]^s \ar[d]_{\rotatebox{90}{\(\sim\)}} & \mathcal{E} \ar[d]_{\rotatebox{90}{\(\sim\)}}^{\text{isomorphism}} \\
		\mathcal{O}_{Z}^r \ar[r]_t & \mathcal{G}
		}
	\]
	where the vertical arrows are isomorphisms and the left one is given by a matrix
	\[
		\begin{bmatrix}
			\lambda & 0       & \cdots & 0       \\
			0       & \lambda & \cdots & 0       \\
			\vdots  & \vdots  & \ddots & \vdots  \\
			0       & 0       & \cdots & \lambda
		\end{bmatrix}
	\]
	for some \(\lambda \in k\setminus \{0\}\).
	Then
	\[
		[\mathcal{E},s] = [\mathcal{G},t].
	\]
\end{lemma}

\begin{proof}
	If \(s_{1},\dots,s_{r}\in H^{0}(Z,\mathcal{E})\) and
	\(t_{1},\dots,t_{r}\in H^{0}(Z,\mathcal{G})\) are the sections of
	\(\mathcal{E},\mathcal{G}\) determined by \(s,t\) respectively,
	then the above is equivalent to saying that there exists \(\lambda\in k\setminus \{0\}\)
	and an isomorphism \(\mathcal{E}\to\mathcal{G}\) sending \(s_{i}\) to \(\lambda t_{i}\),
	for each \(i\).

	Consider the morphism \(\mathcal{G}\to \mathcal{G}\) which sends a section \(u\)
	on an open set to \(\frac{1}{\lambda}u\).
	This is an isomorphism and composing it with \(\mathcal{E}\to \mathcal{G}\)
	gives an isomorphism \(\mathcal{E}\to \mathcal{G}\) sending \(s_{i}\) to \(t_{i}\).
	Thus \([\mathcal{E},s]=[\mathcal{G},t]\).
\end{proof}

\begin{lemma}\label{lem:class_equiv_inclusion}
	Assume \(Z\) is a variety.
	Then to give a class \([\mathcal{E},s]\) of rank \(r\) with \(\mathcal{E}\) reflexive
	is equivalent to giving a reflexive \(\mathcal{K}\) and an inclusion
	\[
		\mathcal{K}\hooklongrightarrow \mathcal{O}_{Z}^r
	\]
	with cokernel of rank zero.
\end{lemma}

\begin{proof}
	Given \(\mathcal{E},s\), dualising
	\[
		\mathcal{O}_{Z}^r\xlongrightarrow{s}\mathcal{E}\longrightarrow\coker(s)\longrightarrow 0
	\]
	we get
	\[
		0 \longrightarrow \mathcal{E}^{\vee} \longrightarrow \mathcal{O}_{Z}^r
	\]
	with \(\rank(\mathcal{O}_{Z}^r/\mathcal{E}^{\vee})=0\).
	Moreover, if \([\mathcal{E},s]=[\mathcal{G},t]\),
	then we have a diagram
	\[
		\xymatrix{
			\mathcal{O}_{Z}^r \ar[r]^s \ar[rd]_t & \mathcal{E} \ar[d] \\
			& \mathcal{G}
		}
	\]
	with \(\mathcal{E}\to \mathcal{G}\) an isomorphism,
	and dualising we get
	\[
		\xymatrix{
			\mathcal{E}^{\vee} \ar@{^{(}->}[r] & \mathcal{O}_{Z}^r \\
			\mathcal{G}^{\vee} \ar[u] \ar@{^{(}->}[ru]
		}
	\]
	which means the two inclusions are the same,
	i.e., their images coincide.

	Conversely, given \(\mathcal{K}\hookrightarrow \mathcal{O}_{Z}^r\),
	dualising we get the class of a stable pair
	\(\mathcal{O}_{Z}^r\xlongrightarrow{s}\mathcal{K}^{\vee}\).
\end{proof}

\begin{definition}
	For a smooth projective variety \(Z\),
	\(M_Z(\ch)\) denotes the moduli space of stable pair classes \([\mathcal{E},s]\)
	with Chern character \(\ch\) on \(Z\).
	This is a projective scheme;
	see \cite{potier1993systemes}, \cite[Theorem~3.7]{sheshmani2016higher} and \cref{lem:auto_stable_pair}.
	When \(Z\) is a smooth projective curve, we simply write it as \(M_Z(r,n)\)
	where \(r\) is the rank and \(n\) is the degree of \(\mathcal{E}\).
	In this case the morphism
	\[
		M_Z(r,n)\longrightarrow \Hilb_Z^n
	\]
	sends \([\mathcal{E},s]\) to the divisor determined by \(\coker(s)\)
	(Quot-to-Chow morphism).
\end{definition}

\begin{lemma}\label{lem:moduli_quot_scheme}
	Assume \(Z\) is a smooth projective curve.
	Then \(M_Z(r,n)\) is isomorphic to \(\Quot(\mathcal{O}_{Z}^r,n)\),
	hence smooth.
\end{lemma}

\begin{proof}
	Assume \([\mathcal{E},s]\in M_Z(r,n)\).
	By \cref{lem:class_equiv_inclusion},
	this class naturally corresponds to the induced inclusion
	\[
		\mathcal{E}^{\vee}\hooklongrightarrow \mathcal{O}_{Z}^r,
	\]
	and if \([\mathcal{E},s]=[\mathcal{G},t]\),
	then
	\[
		\mathcal{E}^{\vee}\hooklongrightarrow \mathcal{O}_{Z}^r \quad \text{and}
		\quad \mathcal{G}^{\vee}\hooklongrightarrow \mathcal{O}_{Z}^r
	\]
	have equal images.
	Therefore the class \([\mathcal{E},s]\) corresponds uniquely
	to a rank zero quotient of \(\mathcal{O}_{Z}^r\)
	of degree \(n\),
	hence to a point of the Quot-scheme \(\Quot(\mathcal{O}_{Z}^r,n)\).

	For smoothness of \(\Quot(\mathcal{O}_{Z}^r,n)\),
	see \cite[Proposition~2.2.8]{huybrechts2010geometry},
	noting that for any short exact sequence
	\[
		0\longrightarrow \mathcal{K}\longrightarrow \mathcal{O}_{Z}^r\longrightarrow \mathcal{Q}\longrightarrow 0
	\]
	with \(\mathcal{Q}\) being rank zero one has
	\(\Ext^1(\mathcal{K},\mathcal{Q})\simeq H^{1}(Z,\mathcal{K}^{\vee}\otimes\mathcal{Q})=0\)
	by the Grothendieck vanishing.
\end{proof}

\begin{lemma}\label{lem:fibre_normal}
	Assume \(Z\) is a smooth projective curve.
	Assume that for some \(r,n\),
	the fibre of
	\[
		M_Z(r,n)\xlongrightarrow{\pi}\Hilb_Z^n
	\]
	over some point \(h\) is of dimension \(n(r-1)\) and smooth in codimension one.
	Then this fibre is an irreducible normal variety.
\end{lemma}

\begin{proof}
	First note that \(\Hilb_Z^n\) is just the \(n\)-th symmetric product of \(Z\),
	hence it is smooth of dimension \(n\)
	(for \(Z=\mathbb{P}^1\), \(\Hilb_Z^n\simeq \mathbb{P}^n\)).
	Pick general ample divisors \(H_{1}, \dots, H_{n}\) through \(h\).
	Then \(H_{i}\) are smooth intersecting transversally at \(h\).

	Since \(\dim\Quot(\mathcal{O}_{Z}^r,n)=nr\),
	general fibres of \(\pi\) are of dimension \(n(r-1)\)
	(in fact, general fibres of \(\pi\) are \((\mathbb{P}^{r-1})^n\)).
	By upper semi-continuity of fibre dimension,
	every fibre of \(\pi\) over some neighbourhood of \(h\) is of dimension \(n(r-1)\).
	Thus we can see that
	\[
		\text{fibre over } h = \pi^{*}H_1 \cap \dots \cap \pi^{*}H_n
	\]
	is a locally complete intersection and
	hence Cohen-Macaulay \cite[Chapter~II, Proposition~8.23]{hartshorne1977algebraic}.
	And it is smooth in codimension one by assumption.
	Therefore,
	it is normal by Serre criterion \cite[Chapter~II, Proposition~8.23]{hartshorne1977algebraic}.
	Here we are using smoothness of \(M_Z(r,n)\) (see \cref{lem:moduli_quot_scheme}).
	Then the fibre is irreducible as \(\pi\) has connected fibres.
\end{proof}

\section{Models associated to a sheaf stable pair}

\begin{definition}\label{def:associated_model}
	Assume \([\mathcal{E},s]\) is a stable pair of rank \(r\) on a variety \(Z\),
	with \(\mathcal{E}\) \emph{locally free}.
	Let \(s_{1}, \dots, s_{r}\) be the sections of \(\mathcal{E}\) determined by \(s\).
	We introduce the notation
	\[
		\begin{array}{l}
			X = \mathbb{P}(\mathcal{E}) \xlongrightarrow{f} Z, \vspace{0.2em}           \\
			\mathcal{O}_{X}(1) = \text{the associated invertible sheaf}, \vspace{0.2em} \\
			D_{i} = \text{divisor of } s_{i}, \vspace{0.2em}                            \\
			A = \text{divisor of } s_1+\dots+s_r.
		\end{array}
	\]
	Here we view \(s_{i}\) as sections of \(\mathcal{O}_{X}(1)\)
	via the morphism \eqref{eq:natural_surjection},
	so \(D_{i}\) is the divisor of this section (similarly for \(s_1+\dots+s_r\)).
	We call
	\[
		X, D_1,\dots,D_r, A \xlongrightarrow{f} Z
	\]
	the \emph{model associated to \([\mathcal{E},s]\)}.

	Now assume \(\mathcal{G},t\) is another stable pair on \(Z\) with \(\mathcal{G}\) locally free.
	Let
	\[
		Y, E_{1}, \dots, E_{r}, C \longrightarrow Z
	\]
	be its associated model.
	We say the above two associated models are \emph{isomorphic}
	if there is an isomorphism over \(Z\)
	\[
		\xymatrix{
			X \ar[rr] \ar[rd] & & Y \ar[ld] \\
			& Z
		}
	\]
	mapping \(\mathcal{O}_{X}(1)\) to \(\mathcal{O}_{Y}(1)\) and
	mapping \(D_{i}\) to \(E_{i}\) and \(A\) to \(C\).
\end{definition}

\begin{lemma}
	Assume \(\mathcal{E},s\) and \(\mathcal{G},t\) are stable pairs on a normal variety \(Z\),
	with \(\mathcal{E},\mathcal{G}\) locally free.
	Then \([\mathcal{E},s]=[\mathcal{G},t]\) if and only if the associated models of
	\(\mathcal{E},s\) and \(\mathcal{G},t\) are isomorphic.
\end{lemma}

\begin{proof}
	Assume \([\mathcal{E},s]=[\mathcal{G},t]\).
	Then there exists an isomorphism \(\mathcal{E}\to \mathcal{G}\)
	sending the corresponding sections
	\(s_{1}, \dots, s_{r}\) to \(t_{1}, \dots, t_{r}\).
	This induces an isomorphism
	\begin{equation}\label{eq:induced_isom}
		\xymatrix{
			\mathbb{P}(\mathcal{G}) \ar[rr]^{\sim} \ar[rd] & & \mathbb{P}(\mathcal{E}) \ar[ld] \\
			& Z
		}
	\end{equation}
	mapping \(\mathcal{O}_{\mathbb{P}(\mathcal{G})}(1)\)
	to \(\mathcal{O}_{\mathbb{P}(\mathcal{E})}(1)\)
	and mapping the divisor of \(t_{i}\) to the divisor of \(s_{i}\).
	It also maps the divisor of \(t_1+\dots+t_r\) to the divisor of \(s_1+\dots+s_r\).

	Conversely, assume the associated models are isomorphic.
	So we have an isomorphism as in \eqref{eq:induced_isom}
	mapping \(\mathcal{O}_{\mathbb{P}(\mathcal{G})}(1)\)
	onto \(\mathcal{O}_{\mathbb{P}(\mathcal{E})}(1)\)
	and mapping the divisor of \(t_{i}\) to the divisor of \(s_{i}\)
	and divisor of \(t_1+\dots+t_r\) to divisor of \(s_1+\dots+s_r\).
	This gives an isomorphism \(\mathcal{E}\to \mathcal{G}\).
	However, divisor of a section does not determine the section,
	it determines it only up to scaling
	\cite[Chapter~II, Proposition~7.7]{hartshorne1977algebraic}
	(as \(\mathbb{P}(\mathcal{E})\), \(\mathbb{P}(\mathcal{G})\) are normal),
	i.e., there exist \(\lambda_{i}, \lambda \in k\setminus \{0\}\)
	such that \(s_{i}\) is mapped to \(\lambda_{i}t_{i}\) and
	\(s_1+\dots+s_r\) is mapped to \(\lambda(t_1+\dots+t_r)\).
	But then
	\[
		\sum \lambda_{i}t_{i} = \lambda(t_1+\dots+t_r).
	\]
	However, the sections \(t_1, \dots, t_r\) are linearly independent over \(k\)
	as \(\mathcal{G},t\) is a stable pair
	so \(\mathcal{O}_{Z}^r\xlongrightarrow{t}\mathcal{G}\) is injective.
	Therefore, \(\lambda_{i}=\lambda\) and \(s_{i}\) is mapped to \(\lambda t_{i}\) for all \(i\).
	This shows \([\mathcal{E},s]=[\mathcal{G},t]\) by \cref{lem:scaling_matrix_equivalent}.
\end{proof}

Next we discuss connection with the theory of stable minimal models.

\begin{definition}
	Assume \(Z\) is a variety.
	In this paper a (lc) \emph{stable minimal model} over \(Z\) is of the form
	\[
		(X,B), A \xlongrightarrow{f} Z
	\]
	where
	\[
		\begin{cases}
			(X,B) \text{ is a log canonical pair equipped with a projective morphism $X \xlongrightarrow{f} Z$}, \\
			K_X+B \text{ is semi-ample over $Z$},                                                                \\
			A\geq 0 \text{ is an integral divisor on } X,                                                        \\
			K_X+B+uA \text{ is ample over } Z \text{ for } 0<u\ll 1,                                             \\
			(X,B+uA) \text{ is log canonical for } 0<u\ll 1.
		\end{cases}
	\]
\end{definition}

For more details see \cite{birkar2022moduli}.
Note however that the above definition is in the relative situation while the setting in \cite{birkar2022moduli} is global.
Also one should not confuse the above notation with the one in \cite{birkar2022moduli};
in the setting above, $Z$ is a fixed base and $K_X+B$ is semi-ample over $Z$ defining a contraction $X\longrightarrow S/Z$;
we have suppressed $S$ in the notation because in this paper we usually deal with the situation where $S=Z$.

\begin{example}\label{exa:stable_minimal_model}
	Assume \(Z=\Spec k\) is a point,
	and \(\mathcal{E},s\) is stable of rank \(r\).
	Then \(\mathcal{E}\) is a \(k\)-vector space of dimension \(r\).
	And \(\mathcal{O}_{Z}^r\xlongrightarrow{s}\mathcal{E}\) is an isomorphism.
	Let
	\[
		X, D_{1}, \dots, D_{r}, A \longrightarrow Z
	\]
	be the associated model.
	Put \(B=\sum D_{i}\).

	We want to argue that \((X,B), A\) is a stable minimal model.
	Identifying \(\mathcal{E}\) with \(k^r\)
	via \(\mathcal{O}_{Z}^r\xlongrightarrow{s}\mathcal{E}\),
	we can assume
	\[
		X=\mathbb{P}^{r-1}=\Proj k[s_{1}, \dots, s_{r}].
	\]
	Then \(D_{1}, \dots, D_{r}\) are the standard coordinate hyperplanes,
	hence \((X,B)\) is lc
	(this is the standard toric structure on \(\mathbb{P}^{r-1}\)).
	To show \((X,B),A\) is a stable minimal model,
	it is enough to show \(s_1+\dots+s_r\) does not identically vanish on \(\bigcap_{i\in I}D_{i}\)
	for any subset \(I\subseteq \{1,\dots,r\}\).
	Assume not.
	Then extending \(I\),
	we can assume \(\abs{I}=r-1\) in which case \(\bigcap_{i \in I} D_{i}\) is one point.
	We may assume \(I=\{1,\dots,r-1\}\).
	Then \(s_1+\dots+s_r\) vanishes on \(\bigcap_{i \in I} D_{i}=(0:\dots:0:1)\)
	which is not the case, a contradiction.
\end{example}

\begin{remark}
	Given a stable pair \(\mathcal{E},s\) on a variety \(Z\) with \(\mathcal{E}\) locally free,
	we can define a model \((X,B), A \to Z\) as in the previous example.
	The example shows that this model is a stable minimal model over some open subset of \(Z\).
	But we will see that it is often not a stable minimal model over the whole \(Z\)
	even when \(Z\) is a smooth curve and \(r=2\).
	It is however possible to modify the above model birationally and
	get a stable minimal model over the whole \(Z\).
	This is our next aim.
\end{remark}

\begin{construction}\label{setup:construction}
	Assume \(Z\) is a normal variety and \(\mathcal{E},s\) is a stable pair on \(Z\).
	We will associate a stable minimal model \((X',B'),A'\xlongrightarrow{f'}Z\).
	Let \(Q=\Supp\coker(s)\), and \(U=Z\setminus Q\).
	Since \(\mathcal{O}_{Z}^r\xlongrightarrow{s}\mathcal{E}\) is an isomorphism on \(U\),
	\(\mathcal{E}\) is locally free on \(U\).
	So by \cref{def:associated_model},
	we have an associated model
	\[
		X^U, D^U_{1}, \dots, D^U_{r}, A^U \longrightarrow U
	\]
	over \(U\).
	If \(F\) is the fibre over any closed point \(z\in U\),
	\[
		(F,\sum D^U_i|_F), A^U|_F
	\]
	is a stable minimal model, by \cref{exa:stable_minimal_model}.
	Letting \(B^U=\sum D_{i}^U\),
	we see that \((X^U,B^U),A^U\) is a stable minimal model over the \emph{smooth locus} of \(U\).
	But it may not be a stable minimal model over the whole \(Z\) (or even \(U\)).

	Let \(X\) be a \emph{compactification} of \(X^U\) over \(Z\).
	Denote \(X\to Z\) by \(f\).
	Take a log resolution \(X''\xlongrightarrow{\varphi}X\).
	We can take this resolution so that it is an isomorphism over the smooth locus of \(U\)
	and so that letting
	\[
		\begin{array}{l}
			B''=\Supp((B^U)^{\sim} + \Exc(\varphi) + \varphi^{-1}f^{-1}Q), \vspace{0.2em} \\
			A''=(A^U)^{\sim},
		\end{array}
		\quad (\sim \text{denotes birational transform})
	\]
	the pair \((X'',B''+A'')\) is log smooth.

	Run an MMP on \(K_{X''}+B''\) over \(Z\).
	Since \((X^U,B^U),A^U\to U\) is a stable minimal model over the smooth locus of \(U\),
	and since \(\varphi\) is an isomorphism over this locus,
	the MMP does not modify \(X''\) over this locus.

	Assume the MMP terminates with a good minimal model \(X'''\).
	Then \(K_{X'''}+B'''\) is semi-ample over \(Z\) defining a contraction \(X'''\to T\to Z\).
	The MMP is also an MMP on \(K_{X''}+B''+uA''\) for any \(0<u\ll 1\).
	Now run another MMP on \(K_{X'''}+B'''+uA'''\) over \(T\).
	Assume this terminates with a good minimal model \(X'\) over \(T\).
	Then \(A'\) is semi-ample over \(T\) defining a birational contraction,
	since \(A'\) is big over \(Z\) and hence big over \(T\).
	Replacing \(X'\) with the base of this contraction we get
	\[
		(X',B'), A' \longrightarrow Z
	\]
	which is a stable minimal model,
	noting that \(A'\) does not pass through any non-klt centre of \((X',B')\).
	Over the smooth locus of \(U\),
	\(X\dashrightarrow X'\) and \(X''\dashrightarrow X'\) are isomorphisms.
\end{construction}

\begin{proposition}
	Let \(Z\) be a smooth variety and \(\mathcal{E},s\) a stable pair of rank \(r\)
	with \(\mathcal{E}\) locally free
	and \(Q\) the divisor of \(\coker(s)\) being simple normal crossing.
	Then the associated stable minimal model
	\[
		(X',B'), A' \longrightarrow Z
	\]
	exists and it depends only on \(Z,r\) and \(\Supp Q\).
\end{proposition}

\begin{proof}
	Consider the identity morphism \(\mathcal{O}_{Z}^r\to\mathcal{O}_{Z}^r\) and
	the associated model
	\[
		X', D'_{1}, \dots, D'_{r}, A' \xlongrightarrow{f'} Z
	\]
	as in \cref{def:associated_model}.
	Let \(B'=\sum D_{i}'+f'^{*}\Supp Q\).
	Let
	\[
		X, D_{1}, \dots, D_{r}, A \xlongrightarrow{f} Z
	\]
	be the associated model of \(\mathcal{E},s\) as in \cref{def:associated_model}.
	Set \(B=\Supp(\sum D_{i}+f^{-1}Q)\).

	By \cref{rem:rational_map}, there is a rational map
	\[
		\xymatrix{
			X \ar[rd]_f \ar@{-->}[rr]^{\alpha} & & X' \ar[ld]^{f'} \\
			& Z
		}
	\]
	induced by the sections \(s_{1}, \dots, s_{r}\) determined by \(s\).

	Let \(t_{1}, \dots, t_{r}\) be the sections of \(\mathcal{O}_{Z}^r\)
	determined by the summand \(\mathcal{O}_{Z}\).
	The morphism
	\[
		\mathcal{O}_{Z}^r \xlongrightarrow{s} \mathcal{E}
	\]
	maps \(t_{i}\) to \(s_{i}\).
	Viewing \(t_{i}\) as sections of \(\mathcal{O}_{X'}(1)\) and
	\(s_{i}\) as sections of \(\mathcal{O}_{X}(1)\),
	the above rational map \(\alpha\) pulls back \(t_{i}\) to \(s_{i}\).
	This shows \(D_{i}=\alpha^{*}D_{i}'\), \(A=\alpha^{*}A'\).
	Moreover, \(\alpha\) is an isomorphism over \(U\)
	where \(U=Z\setminus\Supp Q\).

	Let \(X''\xlongrightarrow{\varphi}X\) be a log resolution of \((X,B+A)\)
	so that it is an isomorphism over \(U\).
	Let
	\[
		B''=\Supp(B^{\sim}+\Exc(\varphi)+\varphi^{-1}f^{-1}Q)
	\]
	and let \(A''\) be the horizontal part of \(A^{\sim}\).
	We can assume \(X''\xlongrightarrow{\psi}X'\) is a morphism.

	Then \(\psi_{*}B''=B'\).
	In fact, \(B''=\Supp(B'^{\sim}+\Exc(\psi))\).
	Thus, we can run an MMP on \(K_{X''}+B''\) over \(X'\)
	ending with a dlt model of \((X',B')\), say \((Y,B_Y)\).
	Moreover, since \(A'\) does not contain any non-klt centre of \((X',B')\),
	the pullback of \(A'\) to \(Y\), say \(A_Y\), is the birational transform of \(A'\).
	Therefore, \((X',B'+uA')\) is the lc model of
	both \((Y,B_Y+uA_Y)\) and \((X'',B''+uA'')\) for any \(0<u\ll 1\).
	So \((X',B'),A'\to Z\) is the stable minimal model associated to \(\mathcal{E},s\).
\end{proof}

\begin{remark}
	Under the assumption of \cref{setup:construction},
	assume there exists an effective Cartier divisor \(L\) on \(Z\) with \(\Supp L=Q\)
	and \(U\) is smooth.
	For example this holds when \(X\) is smooth.
	Then one can show that we can run MMPs
	as in \cref{setup:construction} terminating with good minimal models
	as required.
\end{remark}

\begin{lemma}\label{lem:intersection_of_two_sections_deg}
	Let \(\mathcal{E},s\) be a stable pair of rank \(2\) on a smooth projective curve \(Z\).
	Let
	\[
		X, D_1,D_2, A \xlongrightarrow{f} Z
	\]
	be the model of \(\mathcal{E},s\) as in \cref{def:associated_model}.
	Then \(D_1\cdot D_2=\deg \mathcal{E}\).
\end{lemma}

\begin{proof}
	We have a formula
	\[
		K_X+D_1+D_2\sim f^{*}(K_Z+\det \mathcal{E}).
	\]
	Let \(T_1\) be the horizontal part of \(D_1\).
	We can write \(D_1=T+P\) where \(P\geq 0\) is vertical.
	Then
	\begin{align*}
		(K_X+D_1+D_2) & = f^{*}(K_Z+\det\mathcal{E})\cdot D_1=f^{*}(K_Z+\det\mathcal{E})\cdot T_1 \\
		              & = 2g-2 + \deg \mathcal{E}
	\end{align*}
	where \(g\) is the genus of \(Z\).
	Moreover,
	\begin{align*}
		(K_X+D_1+D_2)\cdot D_1 & = (K_X+D_1) \cdot D_1 + D_1\cdot D_2                    \\
		                       & = (K_X+T_1+P)\cdot (T_1+P) + D_1\cdot D_2               \\
		                       & = (K_X+T_1)\cdot T_1 + (K_X+2T_1)\cdot P + D_1\cdot D_2 \\
		                       & = 2g-2 + D_1\cdot D_2
	\end{align*}
	where we use the fact \((K_X+2T_1)\equiv 0 /Z\).
	Therefore, \(D_1\cdot D_2=\deg \mathcal{E}\).
\end{proof}

\section{Stable pairs on curves with fixed cokernel divisor}

In this section we study stable pairs on a curve
with cokernel given by a fixed divisor on the curve.
This is motivated by the construction of stable minimal models
explained in the previous section.

Given a stable pair \(\mathcal{E},s\) on a smooth curve,
assume \(Q=\sum_{1}^{\ell}n_{j}q_{j}\) is the divisor determined by \(\coker(s)\).
Going from the associated model
\[
	X, D_{1}, \dots, D_{r}, A \longrightarrow Z
\]
to the stable minimal model
\[
	(X',B'), A' \longrightarrow Z
\]
resolves the singularities of sections \(s_{1}, \dots, s_{r}\) determined by \(s\),
and untwists \(\mathcal{E}\) to turn it into \(\mathcal{O}_{Z}^r\).
This procedure makes modifications only over the points \(q_{j}\) and
in an independent way over different \(q_{j}\).
This leads to the following statement which will be proved in a more direct fashion.

\begin{theorem}\label{thm:moduli_of_product}
	Let \(Z\) be a smooth curve and \(Q=\sum_{1}^{\ell}n_{j}q_{j}\) an effective divisor.
	Let
	\begin{align*}
		M_{Q} & = \{[\mathcal{E},s] \mid \mathcal{E},s
		\text{ is a stable pair of rank } r \text{ with cokernel divisor } Q\}, \\
		M_{j} & = \{[\mathcal{E},s] \mid \mathcal{E},s
		\text{ is a stable pair of rank } r \text{ with cokernel divisor } n_{j}q_{j}\}.
	\end{align*}
	Then \(M_{Q}\) and \(M_{j}\) carry natural projective scheme structures and
	\[
		M_{Q}\simeq M_1\times \dots \times M_{\ell}.
	\]
\end{theorem}

We will make some preparations before giving the proof.
We will apply the theorem to understand the fibres of
\[
	M_Z(r,n)\longrightarrow \Hilb_Z^n
\]
when \(Z\) is a smooth projective curve.
In fact this connection already appears in the proof of the theorem.

\begin{lemma}\label{lem:glueing_sheaf_morphisms}
	Assume \(Z\) is a scheme and \(V,W\subseteq Z\) are open subsets such that \(Z=V\cup W\).
	Also assume that
	\[
		\mathcal{F}_V\longrightarrow\mathcal{E}_V\quad\text{and}
		\quad\mathcal{F}_W\longrightarrow \mathcal{E}_W
	\]
	are morphisms of \(\mathcal{O}_{V}\)-modules and \(\mathcal{O}_{W}\)-modules respectively
	such that after restriction to \(V\cap W\) we have a commutative diagram
	\[
		\xymatrix@C=5em{
		\mathcal{F}_V|_{V\cap W}\ar[r]^-{\textup{isomorphism}} \ar[d] & \mathcal{F}_W|_{V\cap W} \ar[d] \\
		\mathcal{E}_V|_{V\cap W} \ar[r]_-{\textup{isomorphism}} & \mathcal{E}_W|_{V\cap W}.
		}
	\]
	Then these uniquely determine a morphism of \(\mathcal{O}_{X}\)-modules
	\[
		\mathcal{F}\longrightarrow \mathcal{E}
	\]
	whose restriction to \(V,W\) coincide with \(\mathcal{F}_V\to\mathcal{E}_V\)
	and \(\mathcal{F}_W\to\mathcal{E}_W\) respectively.
\end{lemma}

\begin{proof}
	The sheaves \(\mathcal{F}_V\) and \(\mathcal{F}_W\) glue
	via the isomorphism \(\mathcal{F}_V|_{V\cap W}\to \mathcal{F}_W|_{V\cap W}\)
	to make an \(\mathcal{O}_{X}\)-module \(\mathcal{F}\).
	Similarly, we get an \(\mathcal{O}_{X}\)-module \(\mathcal{E}\)
	from \(\mathcal{E}_V\) and \(\mathcal{E}_W\).
	For similar reason, the morphisms
	also glue to give \(\mathcal{F}\to \mathcal{E}\) which is uniquely determined.
\end{proof}

\begin{proof}[Proof of \cref{thm:moduli_of_product}]
	First we reduce to the case when \(Z\) is projective.
	Let \(\overline{Z}\) be the compactification of \(Z\).
	Let \(V=Z\) and \(W=\overline{Z}\setminus\Supp Q\).
	Given any stable pair \(\mathcal{E},s\) on \(Z\) with cokernel divisor \(Q\),
	considering
	\[
		\mathcal{O}_{V}^r\xlongrightarrow{s}\mathcal{E}_V\coloneqq\mathcal{E}\quad\text{and}\quad
		\mathcal{O}_{W}^r\xlongrightarrow{\id} \mathcal{O}_{W}^r
	\]
	and the diagram
	\[
		\xymatrix{
		\mathcal{O}^r_{V\cap W} \ar[r]^{\text{id}} \ar[d]_{s|_{V\cap W}} & \mathcal{O}^r_{V\cap W} \ar[d]^{\id} \\
		\mathcal{E}_V|_{V\cap W} \ar[r]_{\alpha} & \mathcal{O}^r_{V\cap W}
		}
	\]
	where \(\alpha\) is determined by the other three arrows (which are all isomorphisms),
	and applying \cref{lem:glueing_sheaf_morphisms},
	we can extend \(\mathcal{E},s\) to \(\overline{Z}\) preserving the cokernel divisor.
	Note that \(\alpha\) is simply the inverse of \(s|_{V\cap W}\).
	Similar remarks apply to stable pairs with cokernel divisor \(n_{j}q_{j}\).
	Thus from now on we assume \(Z\) is projective.
	Then we can identify
	\[
		\begin{array}{l}
			M_Q=\text{fibre of } M_Z(r,n)\longrightarrow\Hilb_Z^n
			\text{ over } Q\qquad (n=\sum n_{j}) \\
			M_j=\text{fibre of } M_Z(r,n)\longrightarrow\Hilb_Z^{n_{j}}
			\text{ over } n_{j}q_{j}
		\end{array}
	\]
	so we get projective scheme structures on \(M_Q\) and \(M_j\).

	Next we will define a map on \(k\)-rational points
	\[
		M_Q \longrightarrow M_1\times \dots \times M_{\ell}.
	\]
	Pick \([\mathcal{E},s]\in M_Q\).
	Let \(V=Z\setminus \{q_{2}, \dots, q_{\ell}\}\) and \(W=Z\setminus \{q_1\}\).
	Then the morphisms
	\[
		\mathcal{O}_{V}^r\xlongrightarrow{s|_V} \mathcal{E}|_V \quad \text{and} \quad
		\mathcal{O}_{W}^r\xlongrightarrow{\id} \mathcal{O}_{W}^r
	\]
	glue together via
	\[
		\xymatrix{
		\mathcal{O}^r_{V\cap W} \ar[r]^{\text{id}} \ar[d]_{s|_{V\cap W}} & \mathcal{O}^r_{V\cap W} \ar[d]^{\id} \\
		\mathcal{E}|_{V\cap W} \ar[r]_{\beta} & \mathcal{O}^r_{V\cap W}
		}
	\]
	where \(\beta\) is determined by other arrows.
	This gives a stable pair \(\mathcal{E}^1,s^1\) with cokernel divisor \(n_1q_1\).
	A similar construction produces \(\mathcal{E}^j,s^j\) with cokernel divisor \(n_{j}q_{j}\).

	Assume \([\mathcal{E},s]=[\mathcal{G},t]\in M_Q\).
	Then there is a commutative diagram
	\[
		\xymatrix{
		\mathcal{O}_{Z}^r \ar[r]^s \ar@{=}[d] & \mathcal{E} \ar[d]^{\text{isomorphism}} \\
		\mathcal{O}_{Z}^r \ar[r]_t & \mathcal{G}
		}
	\]
	Let again \(V=Z\setminus \{q_{2}, \dots, q_{\ell}\}\) and \(W=Z\setminus \{q_1\}\),
	and consider
	\begin{equation}\label{eq:two_commutative_diagrams}
		\xymatrix{
			\mathcal{O}_{V}^r \ar[r]^{s|_V} \ar[d]_{\id} & \mathcal{E}|_V \ar[d] \\
			\mathcal{O}_{V}^r \ar[r]_{t|_V} & \mathcal{G}|_V
		}
		\hspace{5em}
		\xymatrix{
			\mathcal{O}_{W}^r \ar[r]^{\id} \ar[d]_{\id} & \mathcal{O}_W^r \ar[d]^{\id} \\
			\mathcal{O}_{W}^r \ar[r]_{\id} & \mathcal{O}_W^r.
		}
	\end{equation}
	Restricting to \(V\cap W\) we get a commutative diagram
	\[
		\xymatrix@R=1.5em@C=1.5em{
		\mathcal{O}_{V\cap W}^r \ar[rrr]^{\id} \ar[rd]^{s|_{V\cap W}} \ar[dd]_{\id} & & & \mathcal{O}_{V\cap W}^r \ar[rd]^{\id} \ar@{.>}[dd]^(.6){\id} \\
		& \mathcal{E}|_{V\cap W} \ar[rrr] \ar[dd] & & & \mathcal{O}_{V\cap W}^r \ar[dd]^{\id} \\
		\mathcal{O}_{V\cap W}^r \ar@{.>}[rrr]^(.6){\id} \ar[rd]_{t|_{V\cap W}} & & & \mathcal{O}_{V\cap W}^r \ar@{.>}[rd]^{\id} \\
		& \mathcal{G}|_{V\cap W} \ar[rrr] & & & \mathcal{O}_{V\cap W}^r
		}
	\]
	where \(\mathcal{E}|_{V\cap W}\to \mathcal{O}_{V\cap W}^r\) and
	\(\mathcal{G}|_{V\cap W}\to \mathcal{O}_{V\cap W}^r\) are inverses of \(s|_{V\cap W}\) and \(t|_{V\cap W}\).
	Here we are making use of the fact that \(V\cap W=Z\setminus\Supp Q\),
	so \(s|_{V\cap W}\) and \(t|_{V\cap W}\) are both isomorphisms.
	Therefore, the diagrams in \eqref{eq:two_commutative_diagrams} glue together giving a commutative diagram
	\[
		\xymatrix{
		\mathcal{O}_{Z}^r \ar[r]^{s^1} \ar@{=}[d] & \mathcal{E}^1 \ar[d]^{\text{isomorphism}} \\
		\mathcal{O}_{Z}^r \ar[r]_{t^1} & \mathcal{G}^1
		}
	\]
	showing \([\mathcal{E}^1,s^1]=[\mathcal{G}^1,t^1]\).
	The same holds for every \(j\).
	Thus we have a well-defined map
	\[
		M_Q \longrightarrow M_1\times \dots \times M_{\ell}.
	\]
	We claim that this map is bijective.
	Pick \([\mathcal{E}_j,s_j]\in M_j\).
	Inductively, we can assume that there exists a unique stable pair \(\widetilde{\mathcal{E}},\widetilde{s}\)
	on \(Z\) such that
	\[
		\coker(\widetilde{s}) = \sum_{j=2}^{\ell}n_{j}q_{j} \quad \text{and} \quad
		\widetilde{\mathcal{E}}^j,\widetilde{s}^j=\mathcal{E}_j,s_j \quad \forall 2\leq j\leq \ell.
	\]
	Using the notation \(V,W\) from above,
	the two morphisms
	\[
		\mathcal{O}_{V}^r\xlongrightarrow{s_1|_V}\mathcal{E}_1 \quad \text{and} \quad
		\mathcal{O}_{W}^r\xlongrightarrow{\widetilde{s}|_W}\widetilde{\mathcal{E}}|_W
	\]
	restricted to \(V\cap W\) give a diagram
	\[
		\xymatrix@C=5em{
		\mathcal{O}_{V\cap W}^r \ar[r]^{\id} \ar[d]_{s_1|_{V\cap W}} & \mathcal{O}_{V\cap W}^r \ar[d]^{\widetilde{s}|_{V\cap W}} \\
		\mathcal{E}_1|_{V\cap W} \ar[r]_-{\text{isomorphism}} & \widetilde{\mathcal{E}}|_{V\cap W}
		}
	\]
	where the lower isomorphism is determined by the other three isomorphisms.
	So the above two morphisms glue to give a morphism \(\mathcal{O}_{Z}^r\xlongrightarrow{s}\mathcal{E}\).
	From the construction we see that \(\mathcal{E}^j,s^j=\mathcal{E}_j,s_j\) for all \(j\).

	One can also show that the class \([\mathcal{E},s]\) is also uniquely determined
	by the classes \([\mathcal{E}_j,s_j]\).
	We leave this to the reader.

	It remains to show \(M_Q\to M_1\times\dots\times M_{\ell}\) is a morphism of schemes.
	Note that \(M_Q\) and \(M_j\) are (fine) moduli spaces with universal families \(\mathcal{M}_j\to M_j\)
	giving \(\prod\mathcal{M}_j\to \prod M_j\).
	The latter is a family for moduli functor of \(M_Q\),
	hence there is a moduli map \(\prod M_j\to M\).
	One may check that this coincides with inverse of \(M_Q\to \prod M_j\) constructed above.
\end{proof}

\begin{remark}\label{rem:reduced_cokernel_divisor_algebraic}
	Assume \(Z\) is a smooth curve,
	and \(Q=nq\geq 0\) a divisor supported at one point \(q\).
	We outline an algebraic method to calculate
	\[
		M_{Q}=\{[\mathcal{E},s] \mid \mathcal{E},s \text{ is a stable pair of rank } r \text{ with cokernel divisor } Q\}.
	\]
	We can compactify \(Z\) hence assume it is projective.
	Then by \cref{lem:moduli_quot_scheme} it is enough to parametrises quotients
	\[
		\mathcal{O}_{Z}^r\longrightarrow \mathcal{L}
	\]
	with \(\mathcal{L}\) of rank zero whose first Chern class is \(Q\)
	\cite[Example~15.2.16(b)]{fulton1998intersection}.
	Two such quotients are considered the same if their kernels coincide.

	Since \(\mathcal{L}\) is supported at \(q\),
	the quotient is determined by the induced quotient
	\[
		\mathcal{O}_{q}^r\longrightarrow \mathcal{L}_q
	\]
	where \(\mathcal{O}_{q}\) is the local ring of \(Z\) at \(q\);
	this is because \(\mathcal{L}_q=H^0(Z,\mathcal{L})\) and
	giving a surjection \(\mathcal{O}_{Z}^r\to \mathcal{L}\) (resp.\ \(\mathcal{O}_{q}^r\to \mathcal{L}_q\))
	is the same as giving \(r\) sections of \(\mathcal{L}\) generating \(\mathcal{L}\)
	(resp.\ \(r\) elements of \(\mathcal{L}_q\) generating \(\mathcal{L}_q\)).

	Now \(\mathcal{O}_{q}\) is a PID,
	so \(\mathcal{L}_q\) is an \(\mathcal{O}_{q}\)-module of the form
	\[
		\mathcal{O}_{q}/I_1\oplus\dots\oplus\mathcal{O}_{q}/I_d
	\]
	where \(d\leq r\) and
	\[
		\sum_1^d \length(\mathcal{O}_{q}/I_j)=n.
	\]
	If \(t\) is a local parametre at \(q\),
	then \(I_j=\langle t^{\ell_j}\rangle\) for some \(\ell_{j}\leq n\),
	hence
	\[
		\mathcal{O}_{q}/I_j\simeq k[t]/\langle t^{\ell_j}\rangle \quad \text{and} \quad \sum_1^d \ell_{j}=n.
	\]
	So \(M_Q\) parametrises quotients
	\[
		\bigoplus_1^r k[t] \longrightarrow \bigoplus_1^d k[t]/\langle t^{\ell_j}\rangle \quad \text{with} \quad \sum_1^d \ell_{j}=n.
	\]
	In turn this corresponds to quotients
	\begin{equation}\label{eq:quotients_over_k}
		k[t]^{\oplus r} \longrightarrow \bigoplus_1^r k[t]/\langle t^n\rangle \longrightarrow \bigoplus_1^d k[t]/\langle t^{\ell_j}\rangle
		\quad \text{with} \quad \sum_1^d \ell_{j}=n.
	\end{equation}
	In particular, this shows that \(M_Q\) depends only on \(n\)
	and thus to calculate \(M_Q\) we could assume \(Z=\mathbb{P}^1\).
\end{remark}

\begin{remark}\label{rem:mq_smooth_cod_one}
	This is a continuation of \cref{rem:reduced_cokernel_divisor_algebraic}.
	If \(n=1\),
	the quotients of \eqref{eq:quotients_over_k} are of the form
	\[
		k^{\oplus r} \longrightarrow k,
	\]
	which in turn implies \(M_q\simeq\mathbb{P}^{r-1}\).

	\medskip

	In the following we assume \(r,n\geq 2\).
	There are \(p_r(n)\geq 2\) possibilities for \(L=\oplus k[t]/\langle t^{\ell_j}\rangle\)
	satisfies \eqref{eq:quotients_over_k},
	where \(p_r(n)\) is the number of partitions of \(n\) into \(r\) non-zero parts.

	\medskip
	(1)
	First, we take \(L=k[t]/\langle t^n\rangle\).
	A quotient as in \eqref{eq:quotients_over_k} is determined by
	\[
		e_i = a_{i,0} + a_{i,1}t + \dots + a_{i,n-1}t^{n-1}, \quad 1\leq i\leq r
	\]
	such that \(a_{i,0}\neq 0\) for at least one \(i\).
	Note that \(e_i\) is invertible in \(L\) if and only if \(a_{i,0}\neq 0\).

	The quotients with \(e_i\) invertible are parametrised by \(\mathbb{A}^{n(r-1)}\) via
	\[
		\begin{pmatrix}
			a_{1,0}   & a_{1,1}   & \cdots & a_{1,n-1}   \\
			\vdots    & \vdots    & \ddots & \vdots      \\
			a_{i-1,0} & a_{i-1,1} & \cdots & a_{i-1,n-1} \\
			1         & 0         & \cdots & 0           \\
			a_{i+1,0} & a_{i+1,1} & \cdots & a_{i+1,n-1} \\
			\vdots    & \vdots    & \ddots & \vdots      \\
			a_{r,0}   & a_{r,1}   & \cdots & a_{r,n-1}
		\end{pmatrix}
	\]
	whose entries are the coefficients of \(e_i^{-1}e_j\) in \(L\) for \(1\leq j\leq r\).
	All of these glue together to form a smooth open subset \(S_1\) of \(M_Q\)
	of dimension \(n(r-1)\).

	\medskip
	(2)
	Next we take \(L=k[t]/\langle t^{n-1}\rangle\oplus k\).
	Then a quotient as in \eqref{eq:quotients_over_k} is determined
	\[
		e_i = a_{i,0} + a_{i,1}t + \dots + a_{i,n-2}t^{n-2} \text{ and } b_i, \quad 1\leq i\leq r
	\]
	such that
	\[
		\rk
		\begin{pmatrix}
			a_{1,0} & b_1    \\
			a_{2,0} & b_2    \\
			\vdots  & \vdots \\
			a_{r,0} & b_r
		\end{pmatrix}
		=2.
	\]
	The quotients with
	\[
		\begin{pmatrix}
			a_{i,0} & b_i \\
			a_{j,0} & b_j
		\end{pmatrix}
	\]
	invertible are parametrised by \(\mathbb{A}^{n(r-1)-2}\) via
	\[
		\begin{pmatrix}
			a_{1,0}   & a_{1,1}   & \cdots & a_{1,n-2}   & b_1     \\
			\vdots    & \vdots    & \ddots & \vdots      & \vdots  \\
			a_{i-1,0} & a_{i-1,1} & \cdots & a_{i-1,n-2} & b_{i-1} \\
			1         & 0         & \cdots & 0           & 0       \\
			a_{i+1,0} & a_{i+1,1} & \cdots & a_{i+1,n-2} & b_{i+1} \\
			\vdots    & \vdots    & \ddots & \vdots      & \vdots  \\
			a_{j-1,0} & a_{j-1,1} & \cdots & a_{j-1,n-2} & b_{j-1} \\
			0         & a_{j,1}   & \cdots & a_{j,n-2}   & 1       \\
			a_{j+1,0} & a_{j+1,1} & \cdots & a_{j+1,n-2} & b_{j+1} \\
			\vdots    & \vdots    & \ddots & \vdots      & \vdots  \\
			a_{r,0}   & a_{r,1}   & \cdots & a_{r,n-2}   & b_r
		\end{pmatrix}.
	\]
	All of these glue together to form  smooth locally closed subset \(S_2\) of \(M_Q\)
	of dimension \(n(r-1)-2\).

	We conclude that \(M_Q\) has a stratification with locally closed smooth subsets
	\(\{S_i\}_1^{p_r(n)}\) where \(\dim S_1=n(r-1)\) and \(\dim S_j\leq n(r-1)-2\)
	for \(1\leq i\leq p_r(n)\) and \(1<j\leq p_r(n)\).
	In particular, \(M_Q\) is of dimension \(n(r-1)\) and smooth in codimension one.
	Therefore, \(M_Q\) is normal by \cref{lem:fibre_normal}.
\end{remark}

\begin{proof}[Proof of \cref{thm:main_rank_2_on_curves}]
	(1) follows from \cref{lem:moduli_quot_scheme}.
	(2) is an immediate consequence of \cref{thm:moduli_of_product,rem:reduced_cokernel_divisor_algebraic}.

	For (3), as in \cref{rem:reduced_cokernel_divisor_algebraic},
	\(F_{j}\) depends only on \(n_{j}\) and hence we may take \(Z=\mathbb{P}^1\).
	If \(n_j=1\), then \(F_j\simeq \mathbb{P}^{r-1}\) and everything is clear.
	So we assume \(n_j\geq 2\) and then the normality of \(F_{j}\)
	follows from \cref{rem:mq_smooth_cod_one}.
	Note that every fibre of \(\pi\) has dimension equal \(n(r-1)\).
	So \(\pi\) is flat \cite[Theorem~18.16]{eisenbud1995commutative},
	and then \(K_{F_j}\) is Cartier
	since \(K_X\) is Cartier \cite[Theorem~23.4]{matsumura1987commutative}.
\end{proof}

\section{Stable pairs of rank two over curves}

In this section we continue our study from the previous section
but for rank two case with irreducible cokernel support.

\begin{theorem}\label{thm:stable_rank_2_over_curve_irreducible_cokernel}
	Assume \(Z\) is a smooth curve,
	\(Q=nq\) with \(n\geq 0\) and \(q\) a point.
	Then
	\[
		M_{Q}=\{[\mathcal{E},s]\mid\mathcal{E},s \text{ is a stable pair of rank } 2 \text{ with cokernel divisor } Q\}
	\]
	has a natural stratification by locally closed subsets
	\[
		M_Q=\bigcup_{\substack{m\geq 0 \textup{ such that}\\ m+2\ell=n \textup{ for some } \ell\geq 0}} G_m
	\]
	where
	\[
		G_m\subseteq M_{mq}=\{[\mathcal{G},t] \text{ stable of rank } 2 \text{ with cokernel divisor } mq\}
	\]
	is an open subset.
\end{theorem}

\begin{proof}
	As in the proof of \cref{thm:moduli_of_product}
	we can compactify \(Z\) hence assume it is projective.
	If \(Q=0\), then \(\mathcal{O}_{Z}^2\xlongrightarrow{s}\mathcal{E}\) is an isomorphism,
	so there exists only one equivalent class \([\mathcal{E},s]\).
	In this case \(M_Q\) is one point corresponding to \(m=0=n\).
	So we assume \(Q\neq 0\).

	Let \(X, D_1,D_2, A \to Z\) be the model associated to \(\mathcal{E},s\).
	Since \(\mathcal{E},s\) is stable,
	\(D_1,D_2,A\) have no common horizontal component over \(Z\).
	But they may have common components mapping to \(q\).

	Let \(X_0=X\), \(D_{i,0}=D_{i}\).
	Also let \(R_0\) be the largest divisor such that \(R_0\leq D_{i,0}\) for \(i=1,2\).
	Let \(M_{i,0}=D_{i}-R_0\).
	If \(s_{i}\) is the section corresponding to \(D_{i,0}\),
	then \(R_0\) is the fixed part of the linear system generated by \(s_1,s_2\).

	Now \(M_{1,0}\) and \(M_{2,0}\) have no common component and \(M_{i,0}=D_{i}\) over \(Z\setminus \{q\}\).
	Moreover, since \(D_{1,0}\sim D_{2,0}\), \(M_{1,0}\sim M_{2,0}\).
	Assume \(M_{1,0}\cap M_{2,0}=\emptyset\).
	In this case, \(R_0=\frac{n}{2}F\) where \(F\) is the fibre of \(X\to Z\) over \(q\).
	Indeed, this follows from \cref{lem:intersection_of_two_sections_deg} and
	\begin{align*}
		n & = D_{1,0}\cdot D_{2,0} = (M_{1,0}+R_0)\cdot (M_{2,0}+R_0) = M_{1,0}\cdot R_0 + M_{2,0}\cdot R_0 \\
		  & = 2 M_{1,0}\cdot R_0
	\end{align*}
	and \(M_{1,0}\cdot F=1\).
	Moreover,
	since \(M_{1,0}\sim M_{2,0}\) and \(M_{1,0}\cap M_{2,0}=\emptyset\),
	the linear system \(\abs{M_{1,0}}=\abs{M_{2,0}}\) gives an isomorphism \(X\simeq Z\times\mathbb{P}^1\)
	and \(\mathcal{E}\otimes\mathcal{O}_{Z}(-\frac{n}{2})\simeq \mathcal{O}_{Z}^2\).
	Thus the case \(M_{1,0}\cap M_{2,0}=\emptyset\) corresponds to \(m=0\), \(2\ell=n\) and
	there is only one class \([\mathcal{E},s]\) with this property,
	and this gives \(G_0=\pt\) when \(n\) is even.
	This case cannot happen when \(n\) is odd.

	Assume that \(M_{1,0}\cap M_{2,0}\neq \emptyset\).
	Note that at least one of \(M_{1,0}\) and \(M_{2,0}\) is purely horizontal,
	i.e., have no vertical component.
	So \(M_{1,0}\cap F\) or \(M_{2,0}\cap F\) is one point (with multiplicity one).
	Moreover, \(M_{1,0},M_{2,0}\) cannot intersect outside \(F\).
	This implies \(M_{1,0}\cap M_{2,0}\) is one point (but maybe with multiplicity \(>1\)).
	And if \(M_{i,0}\) is purely horizontal,
	then \(M_{i,0}\cap F = M_{1,0}\cap M_{2,0}\) (as sets).

	Let \(X_1\xlongrightarrow{\varphi_1}X_0\) be the blowup of the intersection point \(M_{1,0}\cap M_{2,0}\),
	say \(r_0\),
	and \(E_1\) the exceptional divisor.
	Let
	\[
		M_{i,1}=\varphi_1^{*}M_{i,0}-R_1
	\]
	where \(R_1\) is the largest divisor with \(R_1\leq \varphi_1^{*}M_{i,0}\), \(i=1,2\).
	By the previous paragraph,
	\(F\not\subseteq M_{i,0}\) for some \(i\) and \(M_{i,0}\cdot F=1\),
	so \(M_{i,0}\) is smooth at \(r_0\) for this \(i\),
	so \(R_1=E_1\) and
	\[
		M_{i,1}\cdot E_1=1 \text{ for both } i=1,2.
	\]
	Also \(M_{i,1}\) is purely horizontal for some \(i\),
	intersecting each fibre of \(X_1\to Z\) at one point transversally.
	So \(M_{1,1}\cap M_{2,1}\) is empty or consists of one point only (maybe with multiplicity \(>1\)).

	If \(M_{1,1}\cap M_{2,1}=\emptyset\),
	we stop.
	If not, we blowup \(M_{1,1}\cap M_{2,1}\) and repeat the above.
	This gives a sequence
	\begin{equation}\label{eq:xm_to_x0}
		\xymatrix{
		X_m \ar[r]^-{\varphi_m} & X_{m-1} \ar[r] & \cdots \ar[r] & X_1 \ar[r]^-{\varphi_1} & X_0=X
		}
	\end{equation}
	such that \(X_{j+1}\to X_{j}\) blows up a point on \(E_{j}=\Exc(\varphi_{j})\)
	not belonging to any other component of the fibre of \(X_{j}\to Z\) over \(q\),
	and such that \(M_{1,m}\cap M_{2,m}=\emptyset\).
	So the fibre of \(X_m\to Z\) over \(q\) is a reduced chain of \(\mathbb{P}^1\):
	
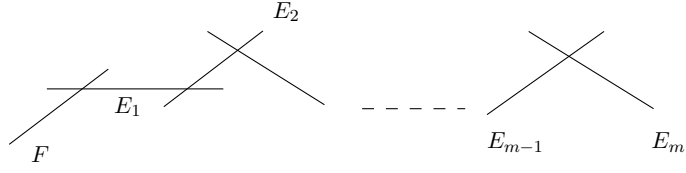
\begin{center}
		\begin{tikzpicture}[x=0.75pt,y=0.75pt,yscale=-1,xscale=1]
			\draw (0,75) -- (50,37);
			\draw (10,75) node [anchor=north west][inner sep=0.75pt] [xscale=0.8,yscale=0.8] {\(F\)};
			\draw (19,47) -- (108,47);
			\draw (52,50) node [anchor=north west][inner sep=0.75pt] [xscale=0.8,yscale=0.8] {\(E_{1}\)};
			\draw (78,56) -- (128,17.75);
			\draw (132,2) node [anchor=north west][inner sep=0.75pt] [xscale=0.8,yscale=0.8] {\(E_{2}\)};
			\draw (100,18) -- (159,55);
			\draw [dash pattern={on 4.5pt off 4.5pt}]  (178,57) -- (230,57);
			\draw (241,60) -- (300,18);
			\draw (240,68) node [anchor=north west][inner sep=0.75pt] [xscale=0.8,yscale=0.8] {\(E_{m-1}\)};
			\draw (262,18) -- (325,57);
			\draw (323,68) node [anchor=north west][inner sep=0.75pt] [xscale=0.8,yscale=0.8] {\(E_{m}\)};
		\end{tikzpicture}
	\end{center}
	with self-intersections \(-1,-2,\dots,-2,-1\), respectively.
	Also \(E_m\not\subseteq M_{i,m}\) for \(i=1,2\),
	and \(M_{i,m}\) is purely horizontal for some \(i\),
	hence this \(M_{i,m}\) intersects \(E_m\) but no other \(E_{i}\) or \(F\).

	Now we can run an MMP on \(X_m\) over \(Z\).
	First we contract \(F\), then \(E_1,\dots\), then \(E_{m-1}\).
	We arrive at a surface \(X'\) such that
	\(E_m\) is the fibre of \(X'\to Z\) over \(q\).
	Since one of \(M_{i,m}\) does not intersect any of the curved contracted by the MMP,
	\(M_{1,m}\) and \(M_{2,m}\) remains disjoint in the process.
	Therefore,
	\[
		X'\simeq Z\times\mathbb{P}^1.
	\]
	Assume \(R_0=\ell F\).
	Then, by \cref{lem:intersection_of_two_sections_deg},
	\begin{align*}
		n & = D_{1,0}\cdot D_{2,0} = (M_{1,0}+R_0)\cdot (M_{2,0}+R_0)    \\
		  & = M_{1,0}\cdot M_{2,0} + M_{1,0}\cdot R_0 + M_{2,0}\cdot R_0 \\
		  & = M_{1,0}\cdot M_{2,0} + 2\ell
	\end{align*}
	where we use the facts
	\[
		M_{1,0}\sim M_{2,0} \quad \text{and} \quad M_{i,0}\cdot F=1.
	\]
	Also we can see from the construction of sequence \eqref{eq:xm_to_x0}
	that \(m=M_{1,0}\cdot M_{2,0}\).
	Thus \(n=m+2\ell\).
	If \(s_1,s_2\) are the sections determined by \(s\),
	then we get sections \(t_1,t_2\) of \(\mathcal{E}\otimes\mathcal{O}_{Z}(-\ell q)\) whose
	zero divisors on \(X\) are \(M_{1,0}\) and \(M_{2,0}\).
	In fact,
	\[
		\mathcal{O}_{Z}^2 \xlongrightarrow{t=(t_1,t_2)} \mathcal{E}\otimes \mathcal{O}_{Z}(-\ell q)
	\]
	gives a stable pair \(\mathcal{E}\otimes \mathcal{O}_{Z}(-\ell q),t\) with cokernel divisor \(mq\).

	Conversely, from any stable pair \(\mathcal{G},t\) with cokernel divisor \(mq\)
	we get a stable pair \(\mathcal{E},s\) with cokernel divisor \(Q=nq\) as above.
	This identifies
	\[
		G_m = \{[\mathcal{E},s]\in M_Q \text{ with } M_{1,0}\cdot M_{2,0}=m\}
	\]
	with the open subset
	\[
		\{[\mathcal{G},t]\in M_{mq} \text{ with } M_1 \text{ and } M_2 \text{ have no common component}\}
	\]
	where \(M_1,M_2\) are the divisors determined by \(t\).
	Note that the case \(m=n\), \(\ell=0\) gives an open subset \(G_n\subseteq M_Q\).
\end{proof}

\begin{remark}
	We use notation in the proof above.
	Let \(D_{i}'\) be the birational transform of \(M_{i,0}\),
	and \(A'\) birational transform of \(A_0=A-R_0\).
	Also let \(F'\) be the fibre of \(X'\to Z\) over \(q\).
	We will see below that
	\[
		(X', B'=D_1'+D_2'+F'), A' \longrightarrow Z
	\]
	is nothing but the stable minimal model associated to \(\mathcal{E},s\).

	It is clear from the construction above that we can recover
	\[
		X, M_{1,0},M_{2,0}, A_0 \longrightarrow Z
	\]
	from
	\[
		X', D_1',D_2', A' \longrightarrow Z
	\]
	by reversing \(X_m\to X'\) and then reversing \(X_m\to X\).

	In fact, \(G_m\) is parametrising all such precesses.
	This amounts to choosing and blowing up a point on \(F'\),
	then blowup a point on the exceptional curve not belonging to the birational transform of \(F'\),
	and so on.

	We argue that \(M_{1,0},M_{2,0},A_0\) are the pullbacks of \(D_1',D_2',A'\)
	under the birational map
	\[
		\xymatrix{
			X \ar[rd] \ar@{-->}[rr] & & X' \ar[ld] \\
			& Z.
		}
	\]
	One way to see this is to note that \(X\dashrightarrow X'\simeq Z\times\mathbb{P}^1\) is the map
	defined by the sections \(t_1,t_2\) of \(\mathcal{E}\otimes \mathcal{O}_{Z}(-\ell q)\)
	corresponding to \(M_{1,0},M_{2,0}\).

	Here is another argument.
	If \(M_{i,0}\) is purely horizontal, then
	\[
		M_{i,m}=\psi^{*}\psi_{*}M_{i,m}=\psi^{*}D_{i}'
	\]
	where \(\psi\) denotes \(X_m\to X'\),
	hence the same holds for both \(i=1,2\) noting that \(\psi_{*}M_{i,m}=D_{i}'\).
	This shows \(M_{i,0}\) is the pullback of \(D_{i}'\).

	On the other hand,
	let \(A_m\geq 0\) be the divisor on \(X_m\) such that \(A_m\sim M_{i,m}\) for \(i=1,2\) and
	such that pushdown of \(A_m\) to \(X\) is \(A_0\).
	Then \(E_m\not\subseteq A_m\) and \(A'\) is the pushdown of \(A_m\).
	Thus \(D_1'\sim D_2'\sim A'\),
	so if \(\overline{A}\) is the pullback of \(A'\) to \(X\),
	then
	\[
		\overline{A}\sim M_{1,0}\sim M_{2,0}\sim A_0 \quad \text{and} \quad
		\overline{A}=A_0 \text{ over } Z\setminus \{q\},
	\]
	so \(\overline{A}=A_0\).
\end{remark}

\begin{remark}
	We explain how the construction in the proof of
	\cref{thm:stable_rank_2_over_curve_irreducible_cokernel}
	relates to stable minimal models.
	We use notation introduced in the proof.

	By construction, \((X_m,B_m)\) is log smooth where
	\[
		B_m=\Supp(D_1^{\sim}+D_2^{\sim}+\Exc(\varphi)+\varphi^{-1}f^{-1}Q)
	\]
	and where \(\sim\) denotes birational transform and \(\varphi\) denotes \(X_m\to X\).
	Also recall \(A_m\sim M_{i,m}\).
	Let \(B_j,A_j\) be the pushdowns of \(B_m,A_m\) to \(X_{j}\) via \(X_{m}\to X_{j}\).
	We claim that
	\[
		(X_m,B_m+uA_m^h)
	\]
	is lc but
	\[
		(X_j,B_j+uA_j^h) \quad 0\leq j<m
	\]
	is not lc for any \(0<u\ll 1\) where \(A_{j}^h\) denotes the horizontal part of \(A_{j}\).
	The pushdown of \(B_m\) to \(X'\) is \(B'=D_1'+D_2'+F'\),
	and the pushdown of \(A_{m}^h\) is \(A'\).
	Moreover, \(D_1',D_2',A'\) are disjoint,
	and \((X',B'+uA')\) is lc.
	Since \(X_m\to X'\) is a sequence of blowups, \((X_m,B_m+uA_m^h)\) is lc.

	Now assume both \(M_{1,m}\) and \(M_{2,m}\) are purely horizontal.
	Then \(M_{1,j}\) and \(M_{2,j}\) intersects at some point on the fibre of \(X_{j}\to Z\) over \(q\),
	so \((X_{j},B_{j})\) is not lc.
	Then we can assume one of \(M_{1,m}\) and \(M_{2,m}\) is purely horizontal but the other is not.
	The same holds for \(M_{1,j}\) and \(M_{2,j}\).
	This implies \(A_0\) is also purely horizontal hence \(A_{j}\) is purely horizontal, so \(A_{j}^h=A_{j}\).

	Then again \(M_{1,j}\) and \(M_{2,j}\) intersect at some point and \(A_{j}\) also passes through this point.
	Thus
	\[
		(X_j,B_j+uA_j^h)
	\]
	is not lc.
	This proves the claim.

	Moreover, \(A'\) does not contain any non-klt centre of \((X',B')\),
	so
	\[
		(X',B'), A' \longrightarrow Z
	\]
	is a stable minimal model.

	Since \(B_m\) equals \(B'^{\sim}\) plus the exceptional divisor of \(X_m\to X'\),
	running MMPs as in \cref{setup:construction} ends with \((X',B')\),
	so the above is the stable minimal model of \(\mathcal{E},s\).

	In summary:
	each \(X_{j+1}\to X_{j}\) is the blowup of a point where
	\[
		(X_{j},B_{j}+uA_{j}^h)
	\]
	fails to be lc.
	When we arrive at \(X_m\),
	\[
		(X_m,B_m+uA_m^h)
	\]
	is lc,
	and using MMP we modify it to get a stable minimal model.
	So the whole process is dictated by transforming \((X_0,B_0+uA_0^h)\) into a stable minimal model.
\end{remark}

\begin{setup}
	\textbf{The variety of invariant subspaces.}
	Here is another point of view from embeddings.

	\begin{lemma}
		Let \(Z\) be a smooth projective curve,
		and let \(\alpha\colon \mathcal{K}\hookrightarrow \mathcal{O}_Z^2\) be a subsheaf of co-length \(n\).
		Suppose that \(\wedge^2\alpha\) given by a section \(s\in H^0(Z,\mathcal{O}_Z(n))\).
		Then there is a diagram
		\[
			\xymatrix@C=3em{
			\mathcal{O}_Z(-n)^{\oplus 2} \ar@/_1.5pc/[rr]_{s\oplus s} \ar@{^{(}->}[r] & \mathcal{K} \ar@{^{(}->}[r]^-{\alpha} & \mathcal{O}_Z^2. \\
			}
		\]
	\end{lemma}

	\begin{proof}
		Let \(\mathcal{E}=\mathcal{K}^{\vee}\).
		Note that \(\mathcal{E}\simeq \mathcal{K}\otimes\mathcal{O}_Z(n)\).
		The adjoint \(\alpha^{\vee}\colon \mathcal{O}_Z^2\longrightarrow \mathcal{E}\),
		which is defined via the adjunct matrix of \(\alpha\),
		can be twisted by \(\mathcal{O}_Z(-n)\).
		The fact that composition is \(s\oplus s\) is essentially Cramer's rule.
	\end{proof}

	Therefore \(\mathcal{K}\) is uniquely determined by a subsheaf of
	the Artinian sheaf
	\[
		\mathcal{A}_s\coloneqq(\mathcal{O}_Z/s\mathcal{O}_Z(-n))^{\oplus 2}
	\]
	of length \(n\), and
	\[
		\mathcal{O}_Z^{2}/\alpha\mathcal{K}\hooklongrightarrow (\mathcal{O}_Z/s\mathcal{O}_Z(-n))^{\oplus 2}=\mathcal{A}_s
	\]
	while \(\length(\mathcal{A}_s)=2n\).
	Suppose that \(s\) vanishes at points \(q_{1}, \dots, q_{r}\) with multiplicities \(n_{1}, \dots, n_{\ell}\),
	then \(s=s_1\cdots s_{\ell}\) where \(s_j\) vanishes exactly at \(q_j\) with multiplicity \(n_j\).

	Denote by \(\Quot(s)\) the space of subsheaves \(\mathcal{K}\) that correspond to a fixed \(s\in H^0(Z,\mathcal{O}_Z(n))\).
	Then we have the factorisation property
	\begin{align*}
		\mathcal{A}_s & = \mathcal{A}_{s_1} \oplus \dots \oplus \mathcal{A}_{s_{\ell}}, \\
		\Quot(s)      & = \Quot(s_1) \times \dots \times \Quot(s_{\ell}).
	\end{align*}
	So it suffices to assume \(s\) vanishes at a single point \(q\) with multiplicity \(n\).
	In other words \(s=t^n\) where \(t\in H^0(\mathcal{O}(1))\) vanishes at \(q\).
	Then \(\Quot(t^n)\) is the variety of \(n\)-dimensional subspaces
	\(V\subseteq (k[t]/\langle t^n\rangle)^{\oplus 2}\) satisfying \(t\cdot V\subseteq V\).

	Write
	\[
		(k[t]/\langle t^n\rangle)^{\oplus 2} = W \oplus tW \oplus \dots \oplus t^{n-1}W
	\]
	with \(\dim W=2\).
	Then \(V\) is a module over \(k[t]\) and by classification of modules over PID,
	\[
		V\simeq k[t]/\langle t^m\rangle \oplus k[t]/\langle t^{n-m}\rangle, \quad m=0,1,\dots, \rdown{\frac{n}{2}}.
	\]
	There are at most \(2\) terms since the null space \(\operatorname{Null}(V\xlongrightarrow{t\cdot} V)\hookrightarrow t^{n-1}W\)
	and \(\dim W=2\).

	\begin{lemma}
		For fixed \(m=0,1,\dots,\rdown{\frac{n}{2}}\eqqcolon b\),
		\[
			V=k[t]\cdot v + (t^{n-m}W\oplus \dots \oplus t^{n-1}W)
		\]
		and \(v\in t^mW\oplus t^{m+1}W\oplus \dots\oplus t^{n-1}W\) has a non-zero projection \(v\pmod{t^{m+1}}\in t^mW\).
	\end{lemma}

	Note that we only care about the components \(v_m+v_{m+1}+\dots+v_{n-m-1}\) of \(v\)
	in \(t^mW\oplus t^{m+1}W\oplus \dots\oplus t^{n-m-1}W\).
	Since we only consider \(V\) as a subspace of \((k[t]/\langle t^n\rangle)^{\oplus 2}\),
	the component \(v_m\) is well-defined only as a point in \(\mathbb{P}(t^mW)=\mathbb{P}^1\).
	Other components \(v_{m+1}+\dots+v_{n-m-1}\) are determined up to an action of the multiplication group
	\(1+kt+\dots+kt^{n-2m-1}\simeq \mathbb{A}^{n-2m-1}\).
	So \(\dim\Quot(t^n)=n\) and there is a stratification
	\[
		\Quot(t^{n})=\Quot(t^{n},0)\cup\Quot(t^{n},1)\cup\dots\cup\Quot(t^{n},b)
	\]
	according to the value of \(m=0,1,\dots,b\) and each stratum \(\Quot(t^{n},m)\)
	is an \(\mathbb{A}^{n-2m-1}\)-bundle over \(\mathbb{P}^1\),
	except that for \(n=2b\) and \(m=b\) one has \(V=t^bW\oplus\dots\oplus t^{2b-1}W\).
	So the stratum \(\Quot(t^{2b},b)\) is a single point.

	For example, in the special case \(n=3\) one has \(b=1\) and \(m=0,1\).
	The \(3\)-dimensional space \(\Quot(t^3)\simeq\Quot(t^3,0)\cup\Quot(t^3,1)\).
	Since \(\Quot(t^3,0)\) is an \(\mathbb{A}^2\)-bundle over \(\mathbb{P}^1\) and
	\(\Quot(t^3,1)\) is isomorphic to \(\mathbb{P}^1\),
	the dimension of stratum drops by \(2\).

\end{setup}

\section{Degree one stable pairs on curves}

In this and following sections we give an in depth analysis of the moduli spaces \(M_Z(2,n)\)
and the fibre of
\[
	M_Z(2,n)\longrightarrow \Hilb_Z^n
\]
for rank two stable pairs on curves and of degree \(n\leq 3\).

\begin{example}[Degree one]\label{exa:degree_one}
	Assume \(Z\) is a smooth projective curve and
	assume \(\mathcal{E},s\) is a stable pair of rank \(2\) and degree \(1\).
	Let
	\[
		X, D_1,D_2, A \longrightarrow Z
	\]
	be the associated model.
	By \cref{lem:intersection_of_two_sections_deg},
	\(D_1\cdot D_2=1\),
	so \(D_1\cap D_2\subseteq\) one fibre of \(f\),
	and divisor \(Q=\coker(s)\) is one point \(q\).
	Let \(F\) be the fibre over \(q\).

	\medskip
	\textbf{Case 1:}
	\(D_1,D_2,A\) have no vertical components.
	Then the associated stable minimal model is obtained as in this picture:
	
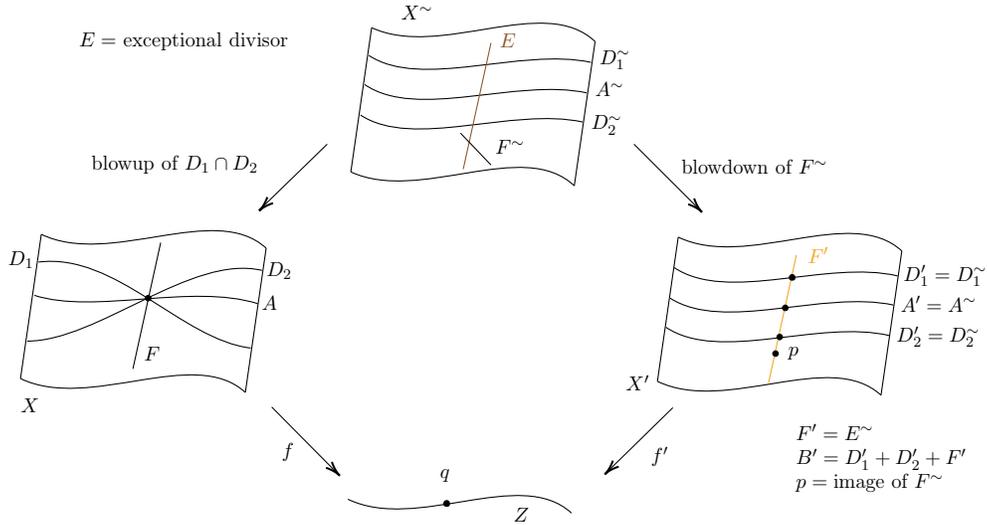
\begin{center}
		\begin{tikzpicture}[x=0.75pt,y=0.75pt,yscale=-0.7,xscale=0.7]
			\draw (283,3) node [anchor=north west][inner sep=0.75pt] [xscale=0.7,yscale=0.7] {\(X^{\sim}\)};
			\draw (263,20) .. controls (313,43) and (387,-5) .. (424,30);
			\draw (248,124) .. controls (298,147) and (372,99) .. (409,134);
			\draw (263,20) -- (248,124);
			\draw (424,30) -- (409,134);
			\draw (261,40) .. controls (299,65) and (370,34) .. (421,45);
			\draw (426,35) node [anchor=north west][inner sep=0.75pt] [xscale=0.7,yscale=0.7] {\(D_{1}^{\sim}\)};
			\draw (258,61) .. controls (296,87) and (367,56) .. (418,67);
			\draw (423,58) node [anchor=north west][inner sep=0.75pt] [xscale=0.7,yscale=0.7] {\(A^{\sim}\)};
			\draw (255,83) .. controls (293,108) and (364,77) .. (415,88);
			\draw (420,82) node [anchor=north west][inner sep=0.75pt] [xscale=0.7,yscale=0.7] {\(D_{2}^{\sim}\)};
			\draw [color={rgb, 255:red, 139; green, 87; blue, 42}, draw opacity=1 ] (349,31) -- (329,122);
			\draw (354,23) node [anchor=north west][inner sep=0.75pt] [color={rgb, 255:red, 139; green, 87; blue, 42}, opacity=1,xscale=0.7, yscale=0.7] {\(E\)};
			\draw (327,96) -- (349,119);
			\draw (351,100) node [anchor=north west][inner sep=0.75pt] [xscale=0.7,yscale=0.7] {\(F^{\sim}\)};

			\draw (9,286) node [anchor=north west][inner sep=0.75pt] [xscale=0.7,yscale=0.7] {\(X\)};
			\draw (25,169) .. controls (75,193) and (150,145) .. (187,179);
			\draw  (10,273) .. controls (60,297) and (135,249) .. (172,283);
			\draw  (25,169) -- (10,273);
			\draw (187,179) -- (172,283);
			\draw (14.9,246) .. controls (66,246) and (132.9,183) .. (184,195);
			\draw (0,180) node [anchor=north west][inner sep=0.75pt] [xscale=0.7,yscale=0.7] {\(D_{1}\)};
			\draw (20,213) .. controls (64,227) and (133,204) .. (181,219);
			\draw (182.8,212) node [anchor=north west][inner sep=0.75pt] [xscale=0.7,yscale=0.7] {\(A\)};
			\draw (23,189) .. controls (79,182) and (125.9,245) .. (175.9,251);
			\draw (186,188) node [anchor=north west][inner sep=0.75pt] [xscale=0.7,yscale=0.7] {\(D_{2}\)};
			\draw (111,175) -- (91,266);
			\draw (98,249) node [anchor=north west][inner sep=0.75pt] [xscale=0.7,yscale=0.7] {\(F\)};
			\draw [fill={rgb, 255:red, 0; green, 0; blue, 0}, fill opacity=1] (100,215) .. controls (100,214) and (101,213) .. (102,213) .. controls (103,213) and (104,214) .. (104,215) .. controls (104,216) and (103,217) .. (102,217) .. controls (101,217) and (100,216) .. (100,215) -- cycle;

			\draw (445,269) node [anchor=north west][inner sep=0.75pt] [xscale=0.7,yscale=0.7] {\(X'\)};
			\draw (483.9,171) .. controls (533.9,195) and (608,147) .. (645,181);
			\draw (468.9,275) .. controls (518.9,299) and (593,251) .. (630,285);
			\draw (483.9,171) -- (468.9,275);
			\draw (645,181) -- (630,285);
			\draw (481,193) .. controls (520,219) and (591,187) .. (642,199);
			\draw (645,190) node [anchor=north west][inner sep=0.75pt] [xscale=0.7,yscale=0.7] {\(D_{1} '=D_{1}^{\sim}\)};
			\draw (478,215) .. controls (517,240) and (588,209) .. (639,220);
			\draw (642.9,212.2) node [anchor=north west][inner sep=0.75pt] [xscale=0.7,yscale=0.7] {\(A'=A^{\sim}\)};
			\draw (475,236) .. controls (514,262) and (585,230) .. (636,242);
			\draw (640,234) node [anchor=north west][inner sep=0.75pt] [xscale=0.7,yscale=0.7] {\(D_{2} '=D_{2}^{\sim}\)};
			\draw [color={rgb, 255:red, 245; green, 166; blue, 35}, draw opacity=1] (569,184) -- (549,276);
			\draw (576,177) node [anchor=north west][inner sep=0.75pt] [color={rgb, 255:red, 245; green, 166; blue, 35 }, opacity=1, xscale=0.7, yscale=0.7] {\(F'\)};
			\draw [fill={rgb, 255:red, 0; green, 0; blue, 0}, fill opacity=1] (564,200) .. controls (564,199) and (565,198) .. (566,198) .. controls (567,198) and (568,199) .. (568,200) .. controls (568,202) and (567,202) .. (566,202) .. controls (565,202) and (564,202) .. (564,200) -- cycle;
			\draw [fill={rgb, 255:red, 0; green, 0; blue, 0}, fill opacity=1] (559,222) .. controls (559,221) and (560,220) .. (561,220) .. controls (562,220) and (563,221) .. (563,222) .. controls (563,223.3) and (562,224) .. (561,224) .. controls (560,224) and (559,223.3) .. (559,222) -- cycle;
			\draw [fill={rgb, 255:red, 0; green, 0; blue, 0}, fill opacity=1] (552,255) .. controls (552,254) and (553,253) .. (554,253) .. controls (555,253) and (556,254) .. (556,255) .. controls (556,256) and (555,257) .. (554,257) .. controls (553,257) and (552,256) .. (552,255) -- cycle;
			\draw [fill={rgb, 255:red, 0; green, 0; blue, 0}, fill opacity=1] (555,243) .. controls (555,242) and (556,241) .. (557,241) .. controls (558,241) and (559,242) .. (559,243) .. controls (559,244) and (558,245) .. (557,245) .. controls (556,245) and (555,244) .. (555,243) -- cycle;
			\draw (562,249) node [anchor=north west][inner sep=0.75pt] [xscale=0.7,yscale=0.7] {\(p\)};

			\draw (245.9,360) .. controls (295.9,384) and (370,336) .. (407,370);
			\draw (364,365) node [anchor=north west][inner sep=0.75pt] [xscale=0.7,yscale=0.7] {\(Z\)};
			\draw  [fill={rgb, 255:red, 0; green, 0; blue, 0}, fill opacity=1] (315,363) .. controls (315,362) and (316,361) .. (317,361) .. controls (318.6,361) and (319,362) .. (319,363) .. controls (319,365) and (318.6,365) .. (317,365) .. controls (316,365) and (315,365) .. (315,363) -- cycle;
			\draw (311,336.9) node [anchor=north west][inner sep=0.75pt] [xscale=0.7,yscale=0.7] {\(q\)};

			\draw (191,293.5) -- (239,342);
			\draw [shift={(240,343)}, rotate = 225.1] [color={rgb, 255:red, 0; green, 0; blue, 0}][line width=0.75] (11,-3) .. controls (7,-1) and (3,0) .. (0,0) .. controls (3,0) and (7,1) .. (11,3);
			\draw (197,317.9) node [anchor=north west][inner sep=0.75pt] [xscale=0.7,yscale=0.7] {\(f\)};
			\draw (452,104) -- (500,151.92);
			\draw [shift={(501,153)}, rotate = 225.1] [color={rgb, 255:red, 0; green, 0; blue, 0}][line width=0.75] (11,-3) .. controls (7,-1) and (3,0) .. (0,0) .. controls (3,0) and (7,1) .. (11,3);
			\draw (485,114) node [anchor=north west][inner sep=0.75pt] [xscale=0.7,yscale=0.7] {blowdown of \(F^{\sim}\)};
			\draw (231,104) -- (183,150.6);
			\draw [shift={(182,152)}, rotate = 316] [color={rgb, 255:red, 0; green, 0; blue, 0}][line width=0.75] (11,-3) .. controls (7,-1) and (3,0) .. (0,0) .. controls (3,0) and (7,1) .. (11,3);
			\draw (60,111) node [anchor=north west][inner sep=0.75pt] [xscale=0.7,yscale=0.7] {blowup of \(D_{1}\cap D_{2}\)};
			\draw (480,294) -- (432,341);
			\draw [shift={(431,342)}, rotate = 316] [color={rgb, 255:red, 0; green, 0; blue, 0}][line width=0.75] (11,-3) .. controls (7,-1) and (3,0) .. (0,0) .. controls (3,0) and (7,1) .. (11,3);
			\draw (463,321) node [anchor=north west][inner sep=0.75pt] [xscale=0.7,yscale=0.7] {\(f'\)};

			\draw (51,22) node [anchor=north west][inner sep=0.75pt] [xscale=0.7,yscale=0.7] {\(E=\) exceptional divisor};
			\draw (561,304) node [anchor=north west][inner sep=0.75pt] [xscale=0.7,yscale=0.7] {$\begin{array}{l}
						F'=E^{\sim}           \\
						B'=D_{1} '+D_{2} '+F' \\
						p=\text{image of } F^{\sim}
					\end{array}$};
		\end{tikzpicture}
	\end{center}
	Here \(X'\simeq Z\times\mathbb{P}^1\) and
	\(D_1',D_2',A'\) are fibres on the projection \(X'\to \mathbb{P}^1\).

	\medskip
	\textbf{Case 2:}
	\(D_1\) or \(D_2\) has a vertical component.
	Since \(D_1\cdot D_2=1\) and since \(D_1,D_2\) have no common horizontal component,
	only one of \(D_1,D_2\) can have a vertical component,
	say \(D_1\).
	For similar reason,
	\(A\) cannot have a vertical component.
	Then the associated stable minimal model is obtained as in this picture:
	
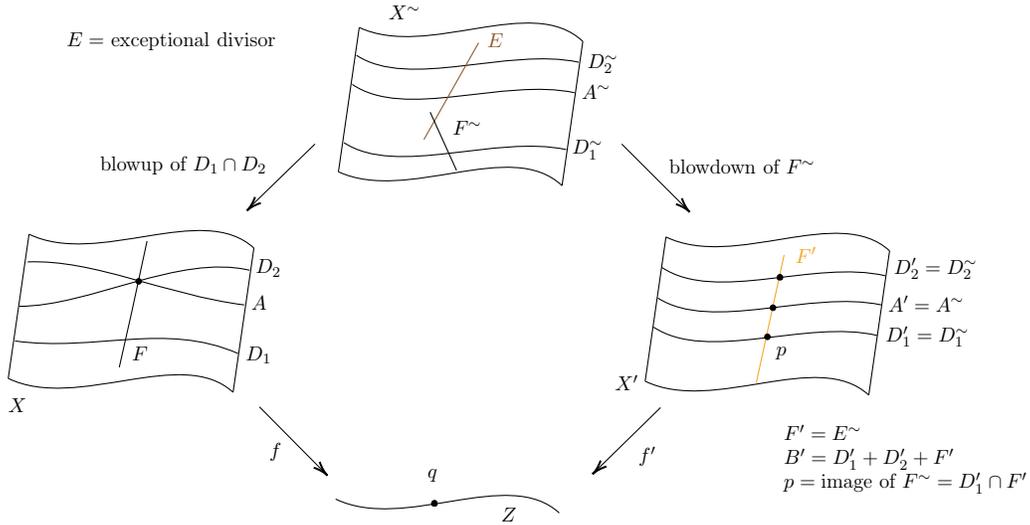
\begin{center}
		\begin{tikzpicture}[x=0.75pt,y=0.75pt,yscale=-0.7,xscale=0.7]
			\draw (283,3) node [anchor=north west][inner sep=0.75pt] [xscale=0.7,yscale=0.7] {\(X^{\sim}\)};
			\draw (263,20) .. controls (312.9,43) and (387,-5) .. (424,30);
			\draw (247.9,124) .. controls (297.9,147) and (372,99) .. (409,134);
			\draw (262.9,20) -- (247.9,124);
			\draw (424,30) -- (409,134);
			\draw (260.9,40) .. controls (299,65) and (370,34) .. (421,45);
			\draw (426,38) node [anchor=north west][inner sep=0.75pt] [xscale=0.7,yscale=0.7] {\(D_{2}^{\sim}\)};
			\draw (257.9,61) .. controls (296,87) and (367,56) .. (418,67);
			\draw (422,60) node [anchor=north west][inner sep=0.75pt] [xscale=0.7,yscale=0.7] {\(A^{\sim}\)};
			\draw (251.9,103) .. controls (290,128) and (361,97) .. (412,108);
			\draw (415,100) node [anchor=north west][inner sep=0.75pt] [xscale=0.7,yscale=0.7] {\(D_{1}^{\sim}\)};
			\draw [color={rgb, 255:red, 139; green, 87; blue, 42}, draw opacity=1]   (348.9,31) -- (309.5,100.75);
			\draw (354,23) node [anchor=north west][inner sep=0.75pt] [color={rgb, 255:red, 139; green, 87; blue, 42}, opacity=1, xscale=0.7, yscale=0.7] {\(E\)};
			\draw (314,81) -- (333,123);
			\draw (329,86) node [anchor=north west][inner sep=0.75pt] [xscale=0.7,yscale=0.7] {\(F^{\sim}\)};

			\draw (8.8,286) node [anchor=north west][inner sep=0.75pt] [xscale=0.7,yscale=0.7] {\(X\)};
			\draw (25,169) .. controls (75,193) and (149.9,145) .. (187,179);
			\draw (10,273) .. controls (60,297) and (135,249) .. (172,283);
			\draw (25,169) -- (10,273);
			\draw (187,179) -- (172,283);
			\draw (17.9,221) .. controls (69,221) and (132.9,183) .. (183.9,195);
			\draw (187,186) node [anchor=north west][inner sep=0.75pt] [xscale=0.7,yscale=0.7] {\(D_{2}\)};
			\draw (24,189) .. controls (72,187) and (123,212.75) .. (180,220);
			\draw (184,212) node [anchor=north west][inner sep=0.75pt] [xscale=0.7,yscale=0.7] {\(A\)};
			\draw (14.9,246) .. controls (73,254) and (123,232) .. (175.5,255);
			\draw (180,249) node [anchor=north west][inner sep=0.75pt] [xscale=0.7,yscale=0.7] {\(D_{1}\)};
			\draw (110,174) -- (99,225) -- (90,265);
			\draw (98,249) node [anchor=north west][inner sep=0.75pt] [xscale=0.7,yscale=0.7] {\(F\)};
			\draw [fill={rgb, 255:red, 0; green, 0; blue, 0}, fill opacity=1] (102,203) .. controls (102,202) and (103,201) .. (104,201) .. controls (105,201) and (106,202) .. (106,203) .. controls (106,204) and (105,205) .. (104,205) .. controls (103,205) and (102,204) .. (102,203) -- cycle;

			\draw (446,269) node [anchor=north west][inner sep=0.75pt] [xscale=0.7,yscale=0.7] {\(X'\)};
			\draw (483.9,171) .. controls (533.9,195) and (608,147) .. (645,181);
			\draw (468.9,275) .. controls (518.9,299) and (593,251) .. (630,285);
			\draw (483.9,171) -- (468.9,275);
			\draw (645,181) -- (630,285);
			\draw (481,193) .. controls (520,219) and (591,187) .. (642,199);
			\draw (647,185) node [anchor=north west][inner sep=0.75pt] [xscale=0.7,yscale=0.7] {\(D_{2} '=D_{2}^{\sim}\)};
			\draw (478,215) .. controls (517,240) and (588,209) .. (639,220);
			\draw (642.9,212.2) node [anchor=north west][inner sep=0.75pt] [xscale=0.7,yscale=0.7] {\(A'=A^{\sim}\)};
			\draw (475,236) .. controls (514,262) and (585,230) .. (636,242);
			\draw (641,234) node [anchor=north west][inner sep=0.75pt] [xscale=0.7,yscale=0.7] {\(D_{1} '=D_{1}^{\sim}\)};
			\draw [color={rgb, 255:red, 245; green, 166; blue, 35}, draw opacity=1]   (569,184) -- (549,276);
			\draw (576,177) node [anchor=north west][inner sep=0.75pt] [color={rgb, 255:red, 245; green, 166; blue, 35}, opacity=1, xscale=0.7, yscale=0.7] {\(F'\)};
			\draw  [fill={rgb, 255:red, 0; green, 0; blue, 0}, fill opacity=1] (564,200) .. controls (564,199) and (565,198) .. (566,198) .. controls (567,198) and (568,199) .. (568,200) .. controls (568,202) and (567,202) .. (566,202) .. controls (565,202) and (564,202) .. (564,200) -- cycle;
			\draw  [fill={rgb, 255:red, 0; green, 0; blue, 0}, fill opacity=1] (559,222) .. controls (559,221) and (560,220) .. (561,220) .. controls (562,220) and (563,221) .. (563,222) .. controls (563,223.3) and (562,224) .. (561,224) .. controls (560,224) and (559,223.3) .. (559,222) -- cycle;
			\draw  [fill={rgb, 255:red, 0; green, 0; blue, 0}, fill opacity=1] (555,243) .. controls (555,242) and (556,241) .. (557,241) .. controls (558,241) and (559,242) .. (559,243) .. controls (559,244) and (558,245) .. (557,245) .. controls (556,245) and (555,244) .. (555,243) -- cycle;
			\draw (562,249) node [anchor=north west][inner sep=0.75pt] [xscale=0.7,yscale=0.7] {\(p\)};

			\draw (245.9,360) .. controls (295.9,384) and (370,336) .. (407,370);
			\draw (364,365) node [anchor=north west][inner sep=0.75pt] [xscale=0.7,yscale=0.7] {\(Z\)};
			\draw [fill={rgb, 255:red, 0; green, 0; blue, 0}, fill opacity=1] (315,363) .. controls (315,362) and (316,361) .. (317,361) .. controls (318.6,361) and (319,362) .. (319,363) .. controls (319,365) and (318.6,365) .. (317,365) .. controls (316,365) and (315,365) .. (315,363) -- cycle;
			\draw (311,336.9) node [anchor=north west][inner sep=0.75pt] [xscale=0.7,yscale=0.7] {\(q\)};

			\draw (191,293.5) -- (239,342);
			\draw [shift={(240,343)}, rotate = 225.1] [color={rgb, 255:red, 0; green, 0; blue, 0}][line width=0.75]  (11,-3) .. controls (7,-1) and (3,0) .. (0,0) .. controls (3,0) and (7,1) .. (11,3);
			\draw (197,317.9) node [anchor=north west][inner sep=0.75pt] [xscale=0.7,yscale=0.7] {\(f\)};
			\draw (452,104) -- (500,151.92);
			\draw [shift={(501,153)}, rotate = 225.1] [color={rgb, 255:red, 0; green, 0; blue, 0}][line width=0.75] (11,-3) .. controls (7,-1) and (3,0) .. (0,0) .. controls (3,0) and (7,1) .. (11,3);
			\draw (485,114) node [anchor=north west][inner sep=0.75pt] [xscale=0.7,yscale=0.7] {blowdown of \(F^{\sim}\)};
			\draw (231,104) -- (183,150.6);
			\draw [shift={(182,152)}, rotate = 316] [color={rgb, 255:red, 0; green, 0; blue, 0}][line width=0.75] (11,-3) .. controls (7,-1) and (3,0) .. (0,0) .. controls (3,0) and (7,1) .. (11,3);
			\draw (75,112) node [anchor=north west][inner sep=0.75pt] [xscale=0.7,yscale=0.7] {blowup of \(D_{1}\cap D_{2}\)};
			\draw (480,294) -- (432,341);
			\draw [shift={(431,342)}, rotate = 316] [color={rgb, 255:red, 0; green, 0; blue, 0}][line width=0.75] (11,-3) .. controls (7,-1) and (3,0) .. (0,0) .. controls (3,0) and (7,1) .. (11,3);
			\draw (463,321) node [anchor=north west][inner sep=0.75pt] [xscale=0.7,yscale=0.7] {\(f'\)};

			\draw (51,22) node [anchor=north west][inner sep=0.75pt] [xscale=0.7,yscale=0.7] {\(E=\) exceptional divisor};
			\draw (561,304) node [anchor=north west][inner sep=0.75pt] [xscale=0.7,yscale=0.7] {$\begin{array}{l}
						F'=E^{\sim}           \\
						B'=D_{1} '+D_{2} '+F' \\
						p=\text{image of } F^{\sim} =D_{1} '\cap F'
					\end{array}$};
		\end{tikzpicture}
	\end{center}
	Again \(X'\simeq Z\times\mathbb{P}^1\) and \(D_1',D_2',A'\) are fibres of the projection \(X'\to \mathbb{P}^1\).

	\medskip
	\textbf{Case 3:}
	\(A\) has a vertical component.
	In this case \(D_1,D_2\) cannot have a vertical component.
	Then the associated stable minimal model is obtained as in this picture:
	
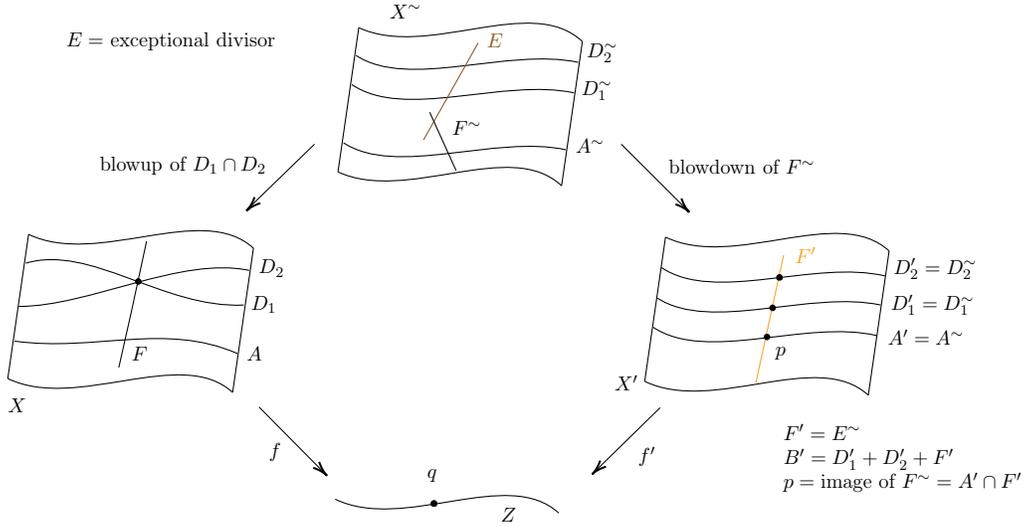
\begin{center}
		\begin{tikzpicture}[x=0.75pt,y=0.75pt,yscale=-0.7,xscale=0.7]
			\draw (283.9,2) node [anchor=north west][inner sep=0.75pt] [xscale=0.7,yscale=0.7] {\(X^{\sim}\)};
			\draw (262.9,20) .. controls (312.9,43) and (387,-5) .. (424,30);
			\draw (247.9,124) .. controls (297.9,147) and (372,99) .. (409,134);
			\draw (262.9,20) -- (247.9,124);
			\draw (424,30) -- (409,134);
			\draw (260.9,40) .. controls (299,65) and (370,34) .. (421,45);
			\draw (426,29.9) node [anchor=north west][inner sep=0.75pt] [xscale=0.7,yscale=0.7] {\(D_{2}^{\sim}\)};
			\draw (257.9,61) .. controls (296,87) and (367,56) .. (418,67);
			\draw (422,57) node [anchor=north west][inner sep=0.75pt] [xscale=0.7,yscale=0.7] {\(D_{1}^{\sim}\)};
			\draw (251.9,103) .. controls (290,128) and (361,97) .. (412,108);
			\draw (418,99) node [anchor=north west][inner sep=0.75pt] [xscale=0.7,yscale=0.7] {\(A^{\sim}\)};
			\draw [color={rgb, 255:red, 139; green, 87; blue, 42}, draw opacity=1]   (348.9,31) -- (309.5,101);
			\draw (354,23) node [anchor=north west][inner sep=0.75pt] [color={rgb, 255:red, 139; green, 87; blue, 42}, opacity=1, xscale=0.7, yscale=0.7] {\(E\)};
			\draw (314,81) -- (333,123);
			\draw (329,86) node [anchor=north west][inner sep=0.75pt] [xscale=0.7,yscale=0.7] {\(F^{\sim}\)};

			\draw (8.8,286) node [anchor=north west][inner sep=0.75pt] [xscale=0.7,yscale=0.7] {\(X\)};
			\draw (25,169) .. controls (75,193) and (149.9,145) .. (187,179);
			\draw (10,273) .. controls (60,297) and (135,249) .. (172,283);
			\draw (25,169) -- (10,273);
			\draw (187,179) -- (172,283);
			\draw (17.9,221) .. controls (69,221) and (134.6,184.6) .. (183.9,195);
			\draw (189.6,186) node [anchor=north west][inner sep=0.75pt] [xscale=0.7,yscale=0.7] {\(D_{2}\)};
			\draw (23,189.8) .. controls (70,176) and (121.8,222.6) .. (180,220);
			\draw (184,212) node [anchor=north west][inner sep=0.75pt] [xscale=0.7,yscale=0.7] {\(D_{1}\)};
			\draw (14.9,246) .. controls (73,254) and (123,232) .. (175.5,255);
			\draw (181,249) node [anchor=north west][inner sep=0.75pt] [xscale=0.7,yscale=0.7] {\(A\)};
			\draw (110,174) -- (99,225) -- (90,265);
			\draw (98,249) node [anchor=north west][inner sep=0.75pt] [xscale=0.7,yscale=0.7] {\(F\)};
			\draw [fill={rgb, 255:red, 0; green, 0; blue, 0}, fill opacity=1] (102,203) .. controls (102,202) and (103,201) .. (104,201) .. controls (105,201) and (106,202) .. (106,203) .. controls (106,204) and (105,205) .. (104,205) .. controls (103,205) and (102,204) .. (102,203) -- cycle;

			\draw (446,269) node [anchor=north west][inner sep=0.75pt] [xscale=0.7,yscale=0.7] {\(X'\)};
			\draw (483.9,171) .. controls (533.9,195) and (608,147) .. (645,181);
			\draw (468.9,275) .. controls (518.9,299) and (593,251) .. (630,285);
			\draw (483.9,171) -- (468.9,275);
			\draw (645,181) -- (630,285);
			\draw (481,193) .. controls (520,219) and (591,187) .. (642,199);
			\draw (647,185) node [anchor=north west][inner sep=0.75pt] [xscale=0.7,yscale=0.7] {\(D_{2} '=D_{2}^{\sim}\)};
			\draw (478,215) .. controls (517,240) and (588,209) .. (639,220);
			\draw (645,211) node [anchor=north west][inner sep=0.75pt] [xscale=0.7,yscale=0.7] {\(D_{1} '=D_{1}^{\sim}\)};
			\draw (475,236) .. controls (514,262) and (585,230) .. (636,242);
			\draw (643,236.2) node [anchor=north west][inner sep=0.75pt] [xscale=0.7,yscale=0.7] {\(A'=A^{\sim}\)};
			\draw [color={rgb, 255:red, 245; green, 166; blue, 35}, draw opacity=1] (569,184) -- (549,276);
			\draw (576,177) node [anchor=north west][inner sep=0.75pt] [color={rgb, 255:red, 245; green, 166; blue, 35}, opacity=1, xscale=0.7, yscale=0.7] {\(F'\)};
			\draw [fill={rgb, 255:red, 0; green, 0; blue, 0}, fill opacity=1] (564,200) .. controls (564,199) and (565,198) .. (566,198) .. controls (567,198) and (568,199) .. (568,200) .. controls (568,202) and (567,202) .. (566,202) .. controls (565,202) and (564,202) .. (564,200) -- cycle;
			\draw [fill={rgb, 255:red, 0; green, 0; blue, 0}, fill opacity=1] (559,222) .. controls (559,221) and (560,220) .. (561,220) .. controls (562,220) and (563,221) .. (563,222) .. controls (563,223.3) and (562,224) .. (561,224) .. controls (560,224) and (559,223) .. (559,222) -- cycle;
			\draw [fill={rgb, 255:red, 0; green, 0; blue, 0}, fill opacity=1] (555,243) .. controls (555,242) and (556,241) .. (557,241) .. controls (558,241) and (559,242) .. (559,243) .. controls (559,244) and (558,245) .. (557,245) .. controls (556,245) and (555,244) .. (555,243) -- cycle;
			\draw (562,249) node [anchor=north west][inner sep=0.75pt] [xscale=0.7,yscale=0.7] {\(p\)};

			\draw (245.9,360) .. controls (295.9,384) and (370,336) .. (407,370);
			\draw (364,365) node [anchor=north west][inner sep=0.75pt] [xscale=0.7,yscale=0.7] {\(Z\)};
			\draw [fill={rgb, 255:red, 0; green, 0; blue, 0}, fill opacity=1] (315,363) .. controls (315,362) and (316,361) .. (317,361) .. controls (318.6,361) and (319,362) .. (319,363) .. controls (319,365) and (318.6,365) .. (317,365) .. controls (316,365) and (315,365) .. (315,363) -- cycle;
			\draw (311,336.9) node [anchor=north west][inner sep=0.75pt] [xscale=0.7,yscale=0.7] {\(q\)};

			\draw (191,293.5) -- (239,342);
			\draw [shift={(240,343)}, rotate = 225.1] [color={rgb, 255:red, 0; green, 0; blue, 0 } ][line width=0.75]  (11,-3) .. controls (7,-1) and (3,0) .. (0,0) .. controls (3,0) and (7,1) .. (11,3)  ;
			\draw (197,317.9) node [anchor=north west][inner sep=0.75pt] [xscale=0.7,yscale=0.7] {\(f\)};
			\draw (452,104) -- (500,151.92);
			\draw [shift={(501,153)}, rotate = 225.1] [color={rgb, 255:red, 0; green, 0; blue, 0}] [line width=0.75] (11,-3) .. controls (7,-1) and (3,0) .. (0,0) .. controls (3,0) and (7,1) .. (11,3);
			\draw (485,114) node [anchor=north west][inner sep=0.75pt] [xscale=0.7,yscale=0.7] {blowdown of \(F^{\sim}\)};
			\draw (231,104) -- (183,150.6);
			\draw [shift={(182,152)}, rotate = 316] [color={rgb, 255:red, 0; green, 0; blue, 0}] [line width=0.75] (11,-3) .. controls (7,-1) and (3,0) .. (0,0) .. controls (3,0) and (7,1) .. (11,3);
			\draw (75,112) node [anchor=north west][inner sep=0.75pt] [xscale=0.7,yscale=0.7] {blowup of \(D_{1} \cap D_{2}\)};
			\draw (480,294) -- (432,341);
			\draw [shift={(431,342)}, rotate = 316] [color={rgb, 255:red, 0; green, 0; blue, 0}] [line width=0.75] (11,-3) .. controls (7,-1) and (3,0) .. (0,0) .. controls (3,0) and (7,1) .. (11,3);
			\draw (463,321) node [anchor=north west][inner sep=0.75pt] [xscale=0.7,yscale=0.7] {\(f'\)};

			\draw (51,22) node [anchor=north west][inner sep=0.75pt] [xscale=0.7,yscale=0.7] {\(E=\) exceptional divisor};
			\draw (561,304) node [anchor=north west][inner sep=0.75pt] [xscale=0.7,yscale=0.7] {$\begin{array}{l}
						F'=E^{\sim}           \\
						B'=D_{1} '+D_{2} '+F' \\
						p=\text{image of } F^{\sim} =A '\cap F'
					\end{array}$};
		\end{tikzpicture}
	\end{center}
	Again \(X'\simeq Z\times\mathbb{P}^1\) and \(D_1',D_2',A'\) are fibres of the projection \(X'\to \mathbb{P}^1\).

	\medskip

	In summary,
	the class \([\mathcal{E},s]\) is determined by the model \(X,D_1,D_2,A\to Z\) which is in turn
	determined by the fixed model \(X',D_1',D_2',A'\to Z\) and the point \(p\) on the fibre \(F'\).
	Therefore, the fibre of
	\[
		M_Z(2,1)\longrightarrow \Hilb_Z^1=Z
	\]
	over each point \(q\) is \(\mathbb{P}^1\)
	and \(M_Z(2,1)\simeq Z\times \mathbb{P}^1\).

	This also completes the proof of \cref{thm:main_m21}.
\end{example}

\begin{example}[Degree one on \(\mathbb{P}^1\)]
	Assume \(Z=\mathbb{P}^1\).
	Then for any \([\mathcal{E},s]\in M_Z(2,1)\), \(\mathcal{E}\simeq \mathcal{O}_{Z}\oplus \mathcal{O}_{Z}(1)\)
	as \(\mathcal{E}\) is nef (\cref{cor:stable_pair_nef}).
	Moreover, by \cref{lem:class_equiv_inclusion},
	\(M_Z(2,1)\simeq\Quot(\mathcal{O}_{Z}^2,1)\),
	more precisely,
	\(M_Z(2,1)\) is parametrising embeddings
	\[
		\mathcal{O}_{Z}\oplus \mathcal{O}_{Z}(-1) \subseteq \mathcal{O}_{Z}\oplus \mathcal{O}_{Z}.
	\]
	Tensoring with \(\mathcal{O}_{Z}(1)\) we get
	\[
		\mathcal{O}_{Z}(1)\oplus \mathcal{O}_{Z} \subseteq \mathcal{O}_{Z}(1)\oplus \mathcal{O}_{Z}(1).
	\]
	Since both sheaves are generated by global sections,
	the above inclusion can be recovered from
	\[
		H^{0}(Z,\mathcal{O}_{Z}(1)\oplus\mathcal{O}_{Z})\subseteq H^{0}(Z,\mathcal{O}_{Z}(1)\oplus \mathcal{O}_{Z}(1)).
	\]
	Thus we get an embedding into a Grassmannian:
	\[
		M_Z(2,1) \hooklongrightarrow \Grass(3,4)\simeq \mathbb{P}^3.
	\]
	Since \(M_Z(2,1)\) has a \(\mathbb{P}^1\)-fibration onto \(Z\),
	which is of Picard number two,
	\(M_Z(2,1)\) is a hypersurface in \(\mathbb{P}^3\) of degree \(2\).
	Therefore,
	\[
		M_Z(2,1) \simeq \mathbb{P}^1 \times \mathbb{P}^1.
	\]
	This is of course consistent with \cref{exa:degree_one}
	where giving a class \([\mathcal{E},s]\in M_Z(2,1)\) is the same as a picking point \(q\in Z=\mathbb{P}^1\)
	and then picking a point \(p\) on the fibre \(F'=\mathbb{P}^1\) of \(X'\to Z\).
\end{example}

\begin{remark}\label{rem:the_edge_case}
	Here is another point of view.

	Assume \(Z=\mathbb{P}^1\).
	Let \(\mathcal{G}_1=\mathcal{O}_Z^{r-1}\), \(\mathcal{G}_2=\mathcal{O}_Z(n)\)
	and \(\mathcal{E}\coloneqq \mathcal{G}_1\oplus \mathcal{G}_2\).
	We consider the moduli space \(\mathcal{M}\) of maps \(\mathcal{O}_Z^r\xrightarrow{s}\mathcal{E}\)
	with the condition that its cokernel has \(0\)-dimensional support and the following equivalence holds
	\[
		\xymatrix@C=3em{
		\mathcal{O}_Z^r \ar[r]^-{s} \ar@{=}[d]_{\id} & \mathcal{E} \ar[d]^{\rotatebox{90}{\(\sim\)}} \\
		\mathcal{O}_Z^{r} \ar[r]_-{s'} & \mathcal{E}.
		}
	\]
	This \(\mathcal{M}\) can be viewed as a subvariety of the Quot-scheme \(\Quot(\mathcal{O}_{Z}^r,n)\),
	while the latter is smooth of dimension \(nr\).

	Such a map \(s\colon \mathcal{O}_Z^r\to \mathcal{E}\) is given by
	\[
		M=
		\begin{pmatrix}
			\alpha_1 & \alpha_2 & \cdots & \alpha_r \\
			\beta_1  & \beta_2  & \cdots & \beta_r
		\end{pmatrix}
		\in \overline{\mathcal{M}}
	\]
	where \(\alpha_{i}\in H^0(\mathcal{G}_2)\eqqcolon H\) and \(\beta_{i}\in H^0(\mathcal{G}_1)\)
	such that its determinant is a non-zero vector in \(H\).
	In particular,
	those \(\beta_i\)'s determine an \((r-1)\)-dimensional subspace in \(V\coloneqq H^0(\mathcal{O}_Z^r)\),
	denoted by \(\beta_M\).
	Then there is a morphism
	\begin{equation*}
		\overline{\pi}\colon \overline{\mathcal{M}}\longrightarrow H\times V, \quad
		M=
		\begin{pmatrix}
			\alpha_1 & \alpha_2 & \cdots & \alpha_r \\
			\beta_1  & \beta_2  & \cdots & \beta_r
		\end{pmatrix}
		\longmapsto (\det M,\beta_M).
	\end{equation*}
	Note that
	\[
		\Aut(\mathcal{E})=(\Aut(\mathcal{G}_1)\times\Aut(\mathcal{G}_2))\ltimes\Hom(\mathcal{G}_1,\mathcal{G}_2)
		=(k^{*}\times\GL(r-1))\ltimes\Hom(\mathcal{O}_Z^{r-1},\mathcal{O}_Z(n))
	\]
	acts on \(\overline{\mathcal{M}}\) by matrix multiplication
	\[
		\begin{pmatrix}
			\gamma_1 & \varphi  \\
			0        & \gamma_2
		\end{pmatrix}
		\begin{pmatrix}
			\alpha_1 & \alpha_2 & \cdots & \alpha_r \\
			\beta_1  & \beta_2  & \cdots & \beta_r
		\end{pmatrix}
		=
		\begin{pmatrix}
			\gamma_1\alpha_1+\varphi\beta_1 & \gamma_1\alpha_2+\varphi\beta_2 & \cdots & \gamma_1\alpha_r+\varphi\beta_r \\
			\gamma_2\beta_1                 & \gamma_2\beta_2                 & \cdots & \gamma_2\beta_r
		\end{pmatrix}
	\]
	where \(\gamma_1\in k^*\), \(\gamma_2\in\GL(r-1)\) and \(\varphi\in\Hom(\mathcal{O}_Z^{r-1},\mathcal{O}_Z(n))\).
	It is not hard to see that \(\det M\) is invariant up to scaling of \(\gamma_1\cdot\det\gamma_2\)
	under the action of \(\Aut(\mathcal{E})\),
	while \(\beta_M\) is invariant as a linear subspace of \(V\).
	Hence there is an induced morphism
	\(\pi\colon\mathcal{M}=\overline{\mathcal{M}}/\Aut(\mathcal{E})\to \mathbb{P}(H)\times\mathbb{P}(V)\).
	The injectivity and surjectivity can be derived from the definition.
	The surjectivity of \(\pi\) is clear.
	Assume that \(\pi(M)=\pi(M')\) for some \(M,M'\in \overline{\mathcal{M}}\) with
	\[
		M=
		\begin{pmatrix}
			\alpha_1 & \alpha_2 & \dots & \alpha_r \\
			\beta_1  & \beta_2  & \dots & \beta_r
		\end{pmatrix}
		\quad \text{and}\quad
		M'=
		\begin{pmatrix}
			\alpha_1' & \alpha_2' & \dots & \alpha_r' \\
			\beta_1'  & \beta_2'  & \dots & \beta_r'
		\end{pmatrix}
		.
	\]
	After multiplying by matrices from \(\Aut(\mathcal{E})\)
	we may assume that both \((\beta_2,\beta_3,\dots,\beta_r)\) and \((\beta_2',\beta_3',\dots,\beta_r')\)
	are the identity matrix \(I_{r-1}\),
	and \(\det M=\det M'\).
	Then \(\beta_1=\beta_1'\) and we may take \(\gamma_1=1\) and \(\gamma_2=I_{r-1}\).
	It is not hard to calculate that \(\det M=\alpha_1-\sum_{j=1}^{r-1}\beta_{1j}\alpha_{j+1}\).
	Now we define \(\varphi=(\varphi_1,\varphi_2,\dots,\varphi_{r-1})\in\Hom(\mathcal{O}_{Z}^{r-1},\mathcal{O}_{Z}(n))\)
	by letting \(\varphi_j=\alpha_{j+1}'-\alpha_{j+1}\).
	Then
	\[
		\begin{pmatrix}
			1 & \varphi \\
			0 & I_{r-1}
		\end{pmatrix}
		M=M'
	\]
	and hence \(\pi\) is injective.
	Since \(\mathcal{M}\) is connected
	and \(\mathbb{P}(H)\times\mathbb{P}(V)\) is smooth,
	the bijective morphism \(\pi\) is an isomorphism.

	In the case \(n=1\), any quotient of \(\mathcal{O}_Z^r\) of degree \(1\)
	has kernel \(\mathcal{O}_Z^{r-1}\oplus\mathcal{O}_Z(-1)\).
	Those quotients are corresponding to stable pairs
	\(\mathcal{O}_Z^r\to\mathcal{O}_Z^{r-1}\oplus\mathcal{O}_Z(1)\).
	In particular, we obtain \(\Quot(\mathcal{O}_{Z}^r,1)\simeq\mathbb{P}^{r-1}\times\mathbb{P}^1\).
\end{remark}

\begin{example}[Reduced cokernel divisor]\label{exa:reduced_cokernel_divisor}
	Assume \(Z\) is a smooth projective curve,
	\(\mathcal{E},s\) a stable pair with associated model
	\[
		X, D_1,D_2, A \xlongrightarrow{f} Z
	\]
	such that the cokernel divisor \(Q\) is reduced,
	say \(Q=q_1+\dots+q_n\).
	Assume \(F_{1}, \dots, F_{n}\) are the fibres passing through the points in \(D_1\cap D_2\),
	i.e., fibres over the \(q_{i}\).
	Then the stable minimal model \((X',B'=D_1'+D_2'),A'\to Z\) is obtained by blowing up \(D_1\cap D_2\)
	followed by blowing down \(F^{\sim}_{1}, \dots, F^{\sim}_{n}\).
	This picture illustrates the case \(n=4\) when \(D_1,D_2,A\) have no vertical component.

	
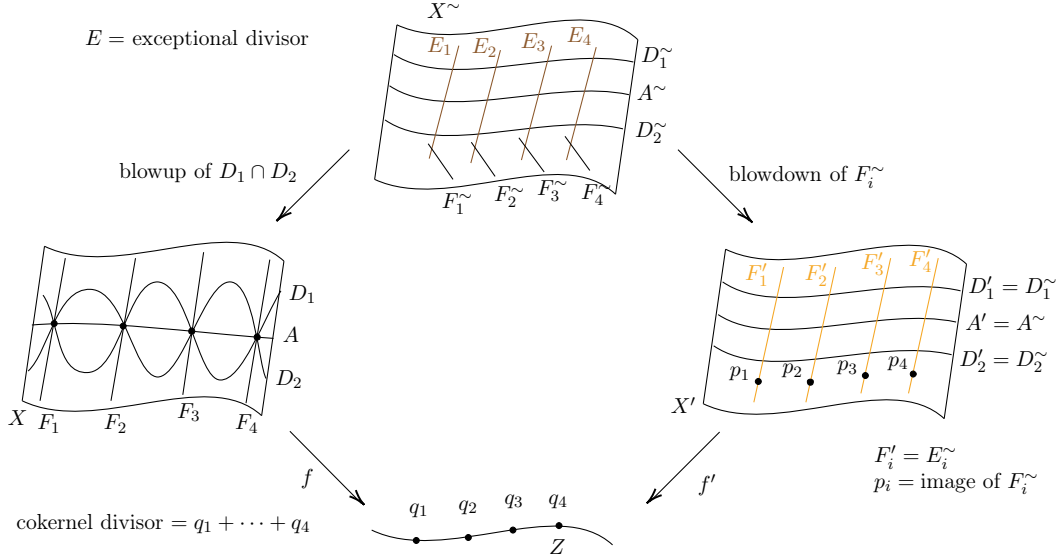
\begin{center}
		\begin{tikzpicture}[x=0.75pt,y=0.75pt,yscale=-0.75,xscale=0.75]
			\draw (283,4) node [anchor=north west][inner sep=0.75pt] [xscale=0.75,yscale=0.75] {\(X^{\sim}\)};
			\draw (264,20) .. controls (314,43) and (389,-5) .. (426,30);
			\draw (249,124) .. controls (299,147) and (374,99) .. (411,134);
			\draw (264,20) -- (249,124);
			\draw (426,30) -- (411,134);
			\draw (262,40) .. controls (301,65) and (372,34) .. (423,45);
			\draw (427,34) node [anchor=north west][inner sep=0.75pt] [xscale=0.75,yscale=0.75] {\(D_{1}^{\sim}\)};
			\draw (259,61) .. controls (298,87) and (369,56) .. (420,67);
			\draw (424,59) node [anchor=north west][inner sep=0.75pt] [xscale=0.75,yscale=0.75] {\(A^{\sim}\)};
			\draw (256,84) .. controls (294,110) and (365,78) .. (416,90);
			\draw (423,84) node [anchor=north west][inner sep=0.75pt] [xscale=0.75,yscale=0.75] {\(D_{2}^{\sim}\)};
			\draw [color={rgb, 255:red, 139; green, 87; blue, 42}, draw opacity=1]   (305.6,34) -- (285.7,110);
			\draw (283,27) node [anchor=north west][inner sep=0.75pt] [color={rgb, 255:red, 139; green, 87; blue, 42 }, opacity=1, xscale=0.75, yscale=0.75] {\(E_{1}\)};
			\draw [color={rgb, 255:red, 139; green, 87; blue, 42}, draw opacity=1]   (333.6,37) -- (313.7,113);
			\draw (313,28) node [anchor=north west][inner sep=0.75pt] [color={rgb, 255:red, 139; green, 87; blue, 42 }, opacity=1, xscale=0.75, yscale=0.75] {\(E_{2}\)};
			\draw [color={rgb, 255:red, 139; green, 87; blue, 42}, draw opacity=1]   (368,34) -- (347.7,111);
			\draw (346,24) node [anchor=north west][inner sep=0.75pt] [color={rgb, 255:red, 139; green, 87; blue, 42 }, opacity=1, xscale=0.75, yscale=0.75] {\(E_{3}\)};
			\draw [color={rgb, 255:red, 139; green, 87; blue, 42}, draw opacity=1]   (398,35) -- (378,111);
			\draw (377,20) node [anchor=north west][inner sep=0.75pt] [color={rgb, 255:red, 139; green, 87; blue, 42 }, opacity=1, xscale=0.75, yscale=0.75] {\(E_{4}\)};
			\draw (285,100) -- (301.7,124);
			\draw (293,132) node [anchor=north west][inner sep=0.75pt] [xscale=0.75,yscale=0.75] {\(F_{1}^{\sim}\)};
			\draw (314,99) -- (330,122);
			\draw (328,127.57) node [anchor=north west][inner sep=0.75pt] [xscale=0.75,yscale=0.75] {\(F_{2}^{\sim}\)};
			\draw (346,96) -- (363,119);
			\draw (357,123) node [anchor=north west][inner sep=0.75pt] [xscale=0.75,yscale=0.75] {\(F_{3}^{\sim}\)};
			\draw (377,95) -- (393.7,118);
			\draw (386.54,126) node [anchor=north west][inner sep=0.75pt] [xscale=0.75,yscale=0.75] {\(F_{4}^{\sim}\)};

			\draw (1,278) node [anchor=north west][inner sep=0.75pt] [xscale=0.75,yscale=0.75] {\(X\)};
			\draw (27,169) .. controls (77,193) and (151,145) .. (188,179);
			\draw (12,273) .. controls (62,297) and (136,249) .. (173,283);
			\draw (27,169) -- (12,273);
			\draw (188,179) -- (173,283);
			\draw (26,201.89) .. controls (37,217) and (34,255) .. (56,254) .. controls (78,253) and (82.7,193) .. (103,193) .. controls (123.59,194) and (129.81,258) .. (145,259) .. controls (160,260) and (175,216) .. (186,199);
			\draw (190,194) node [anchor=north west][inner sep=0.75pt] [xscale=0.75,yscale=0.75] {\(D_{1}\)};
			\draw (18.5,222) .. controls (66,218) and (124,227) .. (181,231);
			\draw (186,222) node [anchor=north west][inner sep=0.75pt] [xscale=0.75,yscale=0.75] {\(A\)};
			\draw (16,253) .. controls (30,240) and (34.7,197) .. (58,198) .. controls (81,199) and (87,255.89) .. (103,257) .. controls (119.59,257) and (135,192) .. (152,193) .. controls (168,194) and (171,249) .. (176,258);
			\draw (181,251) node [anchor=north west][inner sep=0.75pt] [xscale=0.75,yscale=0.75] {\(D_{2}\)};
			\draw (41,177) -- (24,273);
			\draw (21,281) node [anchor=north west][inner sep=0.75pt] [xscale=0.75,yscale=0.75] {\(F_{1}\)};
			\draw (88.2,177) -- (71.7,273);
			\draw (65.7,281) node [anchor=north west][inner sep=0.75pt] [xscale=0.75,yscale=0.75] {\(F_{2}\)};
			\draw (134,173) -- (120,269);
			\draw (115,273) node [anchor=north west][inner sep=0.75pt] [xscale=0.75,yscale=0.75] {\(F_{3}\)};
			\draw (178.5,178) -- (162.5,274.6);
			\draw (154,281) node [anchor=north west][inner sep=0.75pt] [xscale=0.75,yscale=0.75] {\(F_{4}\)};
			\draw  [fill={rgb, 255:red, 0; green, 0; blue, 0}, fill opacity=1] (31.5,221) .. controls (31.5,220) and (32.5,219) .. (33.5,219) .. controls (34.5,219) and (35.5,220) .. (35.5,221) .. controls (35.5,222) and (34.5,223) .. (33.5,223) .. controls (32.5,223) and (31.5,222) .. (31.5,221) -- cycle;
			\draw  [fill={rgb, 255:red, 0; green, 0; blue, 0}, fill opacity=1] (77.9,222.7) .. controls (77.9,221.7) and (78.9,220.7) .. (79.9,220.7) .. controls (80.9,220.7) and (81.9,221.7) .. (81.9,222.7) .. controls (81.9,223.7) and (80.9,224.7) .. (79.9,224.7) .. controls (78.9,224.7) and (77.9,223.7) .. (77.9,222.7) -- cycle;
			\draw  [fill={rgb, 255:red, 0; green, 0; blue, 0}, fill opacity=1] (124,226) .. controls (124,225) and (125,224) .. (126,224) .. controls (127,224) and (128,225) .. (128,226) .. controls (128,228) and (127,228) .. (126,228) .. controls (125,228) and (124,228) .. (124,226) -- cycle;
			\draw  [fill={rgb, 255:red, 0; green, 0; blue, 0}, fill opacity=1] (168,230) .. controls (168,229) and (169,228) .. (170,228) .. controls (171,228) and (172,229) .. (172,230) .. controls (172,232) and (171,232) .. (170,232) .. controls (169,232) and (168,232) .. (168,230) -- cycle;

			\draw (447,269) node [anchor=north west][inner sep=0.75pt] [xscale=0.75,yscale=0.75] {\(X'\)};
			\draw (485,171) .. controls (535,195) and (610,147) .. (647,181);
			\draw (470,275) .. controls (520,299) and (595,251) .. (632,285);
			\draw (485,171) -- (470,275);
			\draw (647,181) -- (632,285);
			\draw (483,193) .. controls (521,219) and (592,187) .. (643,199);
			\draw (647,188) node [anchor=north west][inner sep=0.75pt] [xscale=0.75,yscale=0.75] {\(D_{1}'=D_{1}^{\sim}\)};
			\draw (480,215) .. controls (518,240) and (589,209) .. (640,220);
			\draw (645,213) node [anchor=north west][inner sep=0.75pt] [xscale=0.75,yscale=0.75] {\(A'=A^{\sim}\)};
			\draw (477,236) .. controls (515,262) and (586,230) .. (637,242);
			\draw (641,237) node [anchor=north west][inner sep=0.75pt] [xscale=0.75,yscale=0.75] {\(D_{2}'=D_{2}^{\sim}\)};
			\draw [color={rgb, 255:red, 245; green, 166; blue, 35}, draw opacity=1]   (524,183) -- (504,274);
			\draw (497,179) node [anchor=north west][inner sep=0.75pt] [color={rgb, 255:red, 245; green, 166; blue, 35 }, opacity=1, xscale=0.75, yscale=0.75] {\(F_{1}'\)};
			\draw [color={rgb, 255:red, 245; green, 166; blue, 35}, draw opacity=1]   (559,183) -- (539,274);
			\draw (537,179) node [anchor=north west][inner sep=0.75pt] [color={rgb, 255:red, 245; green, 166; blue, 35 }, opacity=1, xscale=0.75, yscale=0.75] {\(F_{2}'\)};
			\draw [color={rgb, 255:red, 245; green, 166; blue, 35}, draw opacity=1]   (596,177) -- (576,268);
			\draw (575,172) node [anchor=north west][inner sep=0.75pt] [color={rgb, 255:red, 245; green, 166; blue, 35 }, opacity=1, xscale=0.75, yscale=0.75] {\(F_{3}'\)};
			\draw [color={rgb, 255:red, 245; green, 166; blue, 35}, draw opacity=1]   (628,177) -- (608,268);
			\draw (607,170) node [anchor=north west][inner sep=0.75pt] [color={rgb, 255:red, 245; green, 166; blue, 35 }, opacity=1, xscale=0.75, yscale=0.75] {\(F_{4}'\)};
			\draw  [fill={rgb, 255:red, 0; green, 0; blue, 0}, fill opacity=1] (505,260) .. controls (505,259) and (506,258) .. (507,258) .. controls (508,258) and (509,259) .. (509,260) .. controls (509,261) and (508,262) .. (507,262) .. controls (506,262) and (505,261) .. (505,260) -- cycle;
			\draw (486.6,249) node [anchor=north west][inner sep=0.75pt] [xscale=0.75,yscale=0.75] {\(p_{1}\)};
			\draw  [fill={rgb, 255:red, 0; green, 0; blue, 0}, fill opacity=1] (540,260) .. controls (540,259) and (541,258.5) .. (542,258.5) .. controls (543,258.5) and (544,259) .. (544,260) .. controls (544,262) and (543,262.5) .. (542,262.5) .. controls (541,262.5) and (540,262) .. (540,260) -- cycle;
			\draw (522,247) node [anchor=north west][inner sep=0.75pt] [xscale=0.75,yscale=0.75] {\(p_{2}\)};
			\draw  [fill={rgb, 255:red, 0; green, 0; blue, 0}, fill opacity=1] (577,256) .. controls (577,255) and (578,254) .. (579,254) .. controls (580,254) and (581,255) .. (581,256) .. controls (581,257) and (580,258) .. (579,258) .. controls (578,258) and (577,257) .. (577,256) -- cycle;
			\draw (559,244.9) node [anchor=north west][inner sep=0.75pt] [xscale=0.75,yscale=0.75] {\(p_{3}\)};
			\draw  [fill={rgb, 255:red, 0; green, 0; blue, 0}, fill opacity=1] (609.2,255) .. controls (609.2,253.9) and (610,253) .. (611.2,253) .. controls (612,253) and (613.2,253.9) .. (613.2,255) .. controls (613.2,256) and (612,257) .. (611.2,257) .. controls (610,257) and (609.2,256) .. (609.2,255) -- cycle;
			\draw (593,242) node [anchor=north west][inner sep=0.75pt] [xscale=0.75,yscale=0.75] {\(p_{4}\)};

			\draw (247,360) .. controls (297,384) and (372,336) .. (409,370);
			\draw (365,365) node [anchor=north west][inner sep=0.75pt] [xscale=0.75,yscale=0.75] {\(Z\)};
			\draw  [fill={rgb, 255:red, 0; green, 0; blue, 0}, fill opacity=1] (275,367) .. controls (275,366) and (276,365) .. (277,365) .. controls (278,365) and (279,366) .. (279,367) .. controls (279,368) and (278,369) .. (277,369) .. controls (276,369) and (275,368) .. (275,367) -- cycle;
			\draw (271,341) node [anchor=north west][inner sep=0.75pt] [xscale=0.75,yscale=0.75] {\(q_{1}\)};
			\draw  [fill={rgb, 255:red, 0; green, 0; blue, 0}, fill opacity=1] (310,365) .. controls (310,364) and (311,363) .. (312,363) .. controls (313,363) and (314,364) .. (314,365) .. controls (314,366) and (313,367) .. (312,367) .. controls (311,367) and (310,366) .. (310,365) -- cycle;
			\draw (304.2,339) node [anchor=north west][inner sep=0.75pt] [xscale=0.75,yscale=0.75] {\(q_{2}\)};
			\draw  [fill={rgb, 255:red, 0; green, 0; blue, 0}, fill opacity=1] (340,360) .. controls (340,359) and (341,358) .. (342,358) .. controls (343,358) and (344,359) .. (344,360) .. controls (344,361) and (343,362) .. (342,362) .. controls (341,362) and (340,361) .. (340,360) -- cycle;
			\draw (335,335) node [anchor=north west][inner sep=0.75pt] [xscale=0.75,yscale=0.75] {\(q_{3}\)};
			\draw  [fill={rgb, 255:red, 0; green, 0; blue, 0}, fill opacity=1] (371,357) .. controls (371,356) and (372,355) .. (373,355) .. controls (373.8,355) and (375,356) .. (375,357) .. controls (375,358) and (373.8,359) .. (373,359) .. controls (372,359) and (371,358) .. (371,357) -- cycle;
			\draw (364.7,335) node [anchor=north west][inner sep=0.75pt] [xscale=0.75,yscale=0.75] {\(q_{4}\)};

			\draw (192,293.5) -- (240,342);
			\draw [shift={(242,343)}, rotate = 225.1] [color={rgb, 255:red, 0; green, 0; blue, 0}][line width=0.75] (11,-3) .. controls (7,-1) and (3,0) .. (0,0) .. controls (3,0) and (7,1) .. (11,3);
			\draw (198,317.9) node [anchor=north west][inner sep=0.75pt] [xscale=0.75,yscale=0.75] {\(f\)};
			\draw (453,104) -- (501,152);
			\draw [shift={(503,153)}, rotate = 225.1] [color={rgb, 255:red, 0; green, 0; blue, 0}][line width=0.75] (11,-3) .. controls (7,-1) and (3,0) .. (0,0) .. controls (3,0) and (7,1) .. (11,3);
			\draw (486,114) node [anchor=north west][inner sep=0.75pt] [xscale=0.75,yscale=0.75] {blowdown of \(F_{i}^{\sim}\)};
			\draw (232,104) -- (185,150.6);
			\draw [shift={(183,152)}, rotate = 316] [color={rgb, 255:red, 0; green, 0; blue, 0}][line width=0.75] (11,-3) .. controls (7,-1) and (3,0) .. (0,0) .. controls (3,0) and (7,1) .. (11,3);
			\draw (76,112) node [anchor=north west][inner sep=0.75pt] [xscale=0.75,yscale=0.75] {blowup of \(D_{1}\cap D_{2}\)};
			\draw (481,294) -- (434,341);
			\draw [shift={(432,342)}, rotate = 316] [color={rgb, 255:red, 0; green, 0; blue, 0}][line width=0.75]    (11,-3) .. controls (7,-1) and (3,0) .. (0,0) .. controls (3,0) and (7,1) .. (11,3);
			\draw (465,321) node [anchor=north west][inner sep=0.75pt] [xscale=0.75,yscale=0.75] {\(f'\)};

			\draw (53,22) node [anchor=north west][inner sep=0.75pt] [xscale=0.75,yscale=0.75] {\(E=\) exceptional divisor};
			\draw (7,348) node [anchor=north west][inner sep=0.75pt] [xscale=0.75,yscale=0.75] {cokernel divisor \(=q_{1}+\cdots+q_{4}\)};
			\draw (577,301) node [anchor=north west][inner sep=0.75pt] [xscale=0.75,yscale=0.75] {\(\begin{array}{l}
					F_{i} '=E_{i}^{\sim} \\
					p_{i} =\text{image of } F_{i}^{\sim}
				\end{array}\)};
		\end{tikzpicture}
	\end{center}

	The stable minimal model \((X',B'),A'\to Z\) satisfies
	\[
		X'=Z\times\mathbb{P}^1, \quad B'=D_1'+D_2'+\sum_{1}^{n}F_{i}'
	\]
	and this is equipped with the marked points \(p_{i}\in F_{i}'\).
	When \(D_1,D_2,A\) have no vertical component we have
	\[
		p_{i}\in F_{i}'\setminus (D_1'\cup D_2'\cup A') \simeq \mathbb{P}^1\setminus \{0,1,\infty\}.
	\]
	The marked stable minimal model
	\[
		(X',B'), A', p_{1}, \dots, p_{n}
	\]
	determines the equivalence class \([\mathcal{E},s]\).
	Moreover, starting with \(n\) distinct points \(q_{1}, \dots, q_{n}\in Z\)
	and considering a marked stable minimal model as above,
	the same process produces an equivalence class \([\mathcal{E},s]\).
	In other words,
	all such marked stable minimal models are parametrised by
	\[
		\Hilb_{Z,g}^n \times \mathbb{P}^1
	\]
	where \(\Hilb_{Z,g}\subseteq \Hilb_Z^n\) is the locus corresponding to reduced divisors
	(all coefficients equal \(1\))
	on \(Z\) of degree \(n\).

	The classes \([\mathcal{E},s]\) with the properties above form an open subset of \(M_Z(2,n)\).
	Thus \(M_Z(2,n)\) is birational to the product above.

	Note also that \(p_{i}\in D_1'\) if and only if \(D_1\) has a vertical component over \(q_{i}\).
	A similar remark applies to \(D_2'\) and \(A'\).
\end{example}

\section{Degree two stable pairs on curves}

\begin{example}[Degree two on \(\mathbb{P}^1\), geometric treatment]
	Assume \(Z=\mathbb{P}^1\) and \([\mathcal{E},s]\in M_Z(2,2)\).
	Then \(\mathcal{E}=\mathcal{O}_{Z}(1)\oplus \mathcal{O}_{Z}(1)\) or
	\(\mathcal{E}=\mathcal{O}_{Z}\oplus \mathcal{O}_{Z}(2)\) as \(\mathcal{E}\) is nef.

	\medskip

	(1)
	Assume \(\mathcal{E},s\) has reduced cokernel divisor \(Q=q_1+q_2\) as in \cref{exa:reduced_cokernel_divisor}.
	Let
	\[
		X, D_1,D_2, A \longrightarrow Z
	\]
	be the associated model and
	\[
		(X',B'), A', p_1,p_2 \longrightarrow Z
	\]
	be the associated marked stable minima model.
	We claim that
	\[
		\mathcal{E}=\mathcal{O}_{Z}\oplus \mathcal{O}_{Z}(2) \iff
		p_1,p_2 \in \text{the same fibre of the second projection } X'\longrightarrow \mathbb{P}^1.
	\]
	\begin{proof}
		\((\Longrightarrow)\)
		There exists a \((-2)\)-curve \(S\subseteq X\) corresponding to
		the summand \(\mathcal{O}_{Z}\) of \(\mathcal{E}\),
		i.e., corresponding to the surjection \(\mathcal{E}\twoheadrightarrow\mathcal{O}_{Z}\).
		Let \(F_1,F_2\) be the fibre of \(X\to Z\) over the points \(q_1,q_2\) respectively.
		There is a section \(u\in H^0(\mathcal{E})\simeq H^0(\mathcal{O}_{X}(1))\) whose divisor is \(S+F_1+F_2\).
		None of \(D_1,D_2\) can pass through \(S\cap F_1\) or \(S\cap F_2\)
		otherwise \(2=\deg \mathcal{E}=D_{i}\cdot(S+F_1+F_2)>2\), a contradiction.
		So if \(S'\subseteq X'\) is the birational transform of \(S\),
		then \(S'^2=0\) and \(p_1,p_2\in S'\).
		But we can check \(S'\sim D_i'\),
		so \(S'\cap D_i'=\emptyset\),
		hence \(S'\) is a fibre of the second projection \(X'=Z\times\mathbb{P}^1\to \mathbb{P}^1\).

		\medskip
		\((\Longleftarrow)\)
		Let \(S'\) be the fibre containing \(p_1,p_2\).
		Then we can check that \(S^2=-2\) where \(S\subseteq X\) is the birational transform of \(S'\).
		So this is possible only if \(\mathcal{E}=\mathcal{O}_{Z}\oplus \mathcal{O}_{Z}(2)\).
	\end{proof}

	The claim shows that the points \([\mathcal{E},s]\in M_Z(2,2)\) with
	\(\mathcal{E}=\mathcal{O}_{Z}\oplus \mathcal{O}_{Z}(2)\) and cokernel divisor \(q_1+q_2\) are
	parametrised by a copy of \(\mathbb{P}^1\).

	\medskip
	(2)
	Now assume \(\coker(s)\) divisor is non-reduced, say \(Q=2q\).

	\medskip

	\textbf{Case \Romannum{1}:}
	The fibre \(F\) over \(q\) is a component of both \(D_1,D_2\).
	Then \(X=X'\) and \(D_{i}'\) is the horizontal part of \(D_{i}\) (similarly for \(A\)).

	\medskip

	\textbf{Case \Romannum{2}:}
	\(F\) is not a component of \(D_1,D_2,A\).
	Then \(D_1\) and \(D_2\) are tangent to each other at some point:
	indeed if \(A\) intersects \(D_1\) at any point,
	then it also intersects \(D_2\) at the same point because \(A\) is the divisor of \(s_1+s_2\).
	Then the stable minimal model is obtained as follows:

	
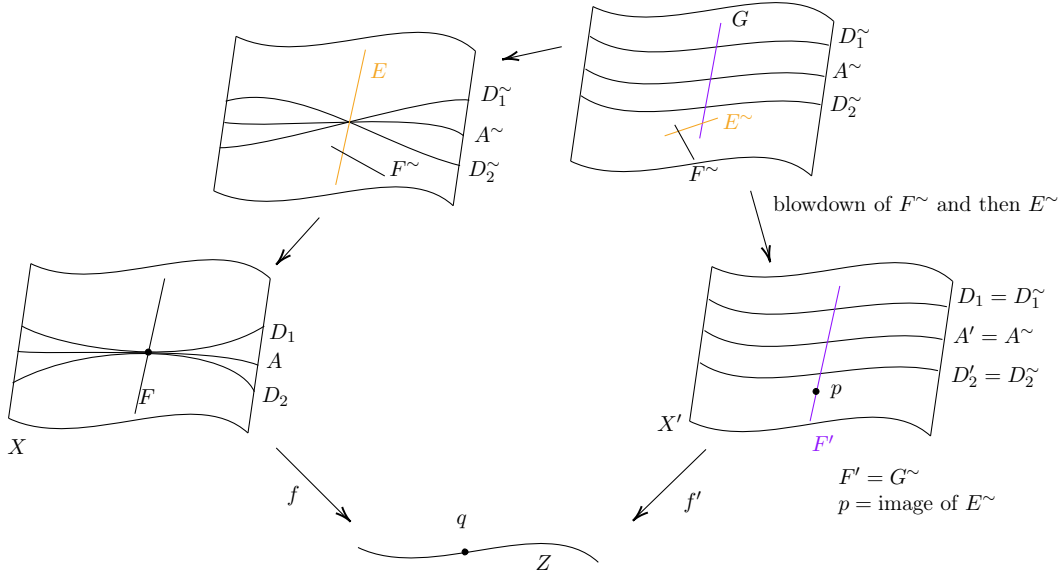
\begin{center}
		\begin{tikzpicture}[x=0.75pt,y=0.75pt,yscale=-0.75,xscale=0.75]
			\draw (396,3) .. controls (446,27) and (520.6,-21) .. (558,13);
			\draw (381,107) .. controls (431,131) and (505.6,83) .. (543,117);
			\draw (396,3) -- (381,107);
			\draw (558,13) -- (543,117);
			\draw (394,23) .. controls (433,49) and (503.6,17) .. (554.6,29);
			\draw (559.6,16) node [anchor=north west][inner sep=0.75pt] [xscale=0.75,yscale=0.75] {\(D_{1}^{\sim}\)};
			\draw (391,45) .. controls (429.6,70) and (500.6,39) .. (551.6,50);
			\draw (555.54,40) node [anchor=north west][inner sep=0.75pt] [xscale=0.75,yscale=0.75] {\(A^{\sim}\)};
			\draw (388,63) .. controls (426,89) and (498,58) .. (549,69);
			\draw (554,62) node [anchor=north west][inner sep=0.75pt] [xscale=0.75,yscale=0.75] {\(D_{2}^{\sim}\)};
			\draw [color={rgb, 255:red, 144; green, 19; blue, 254}, draw opacity=1]   (482,14) -- (467.76,91.7);
			\draw (488,5) node [anchor=north west][inner sep=0.75pt] [xscale=0.75,yscale=0.75] {\(G\)};
			\draw [color={rgb, 255:red, 245; green, 166; blue, 35}, draw opacity=1]   (480,78) -- (444.2,90);
			\draw (482,74) node [anchor=north west][inner sep=0.75pt] [color={rgb, 255:red, 245; green, 166; blue, 35 }, opacity=1, xscale=0.75, yscale=0.75] {\(E^{\sim}\)};
			\draw    (464,106) -- (451,82.58);
			\draw (459,109) node [anchor=north west][inner sep=0.75pt] [xscale=0.75,yscale=0.75] {\(F^{\sim}\)};

			\draw (1,293) node [anchor=north west][inner sep=0.75pt] [xscale=0.75,yscale=0.75] {\(X\)};
			\draw (18,176) .. controls (68,200) and (142,152) .. (179,186);
			\draw (3,280) .. controls (53,304) and (127,256) .. (164,290);
			\draw (18,176) -- (3,280);
			\draw (179,186) -- (164,290);
			\draw (5.71,256.29) .. controls (61.59,225.82) and (155,232.53) .. (168.31,262.2);
			\draw (177,215.5) node [anchor=north west][inner sep=0.75pt] [xscale=0.75,yscale=0.75] {\(D_{1}\)};
			\draw (12,218) .. controls (36.5,232) and (126.5,249) .. (175,218);
			\draw (172,258) node [anchor=north west][inner sep=0.75pt] [xscale=0.75,yscale=0.75] {\(D_{2}\)};
			\draw (9,235) .. controls (30,238) and (138,231) .. (171,244.4);
			\draw (175,235) node [anchor=north west][inner sep=0.75pt] [xscale=0.75,yscale=0.75] {\(A\)};
			\draw (108,186) -- (88,277);
			\draw (89,259) node [anchor=north west][inner sep=0.75pt] [xscale=0.75,yscale=0.75] {\(F\)};
			\draw  [fill={rgb, 255:red, 0; green, 0; blue, 0 }, fill opacity=1] (95,235.5) .. controls (95,234.4) and (96,233.5) .. (97,233.5) .. controls (98,233.5) and (99,234.4) .. (99,235.5) .. controls (99,237) and (98,237.5) .. (97,237.5) .. controls (96,237.5) and (95,237) .. (95,235.5) -- cycle;

			\draw (438,276) node [anchor=north west][inner sep=0.75pt] [xscale=0.75,yscale=0.75] {\(X'\)};
			\draw (476,178) .. controls (526,202) and (600.6,154) .. (638,188);
			\draw (461,282) .. controls (511,306) and (585.6,258) .. (623,292);
			\draw (476,178) -- (461,282);
			\draw (638,188) -- (623,292);
			\draw (474,200) .. controls (512,225) and (583,194) .. (634,205);
			\draw (639.6,191) node [anchor=north west][inner sep=0.75pt] [xscale=0.75,yscale=0.75] {\(D_{1} =D_{1}^{\sim}\)};
			\draw (471,221) .. controls (509,247) and (580,216) .. (631,227);
			\draw (637,217) node [anchor=north west][inner sep=0.75pt] [xscale=0.75,yscale=0.75] {\(A'=A^{\sim}\)};
			\draw (468,242.64) .. controls (506,268) and (577,236.89) .. (628,248);
			\draw (635.5,242.84) node [anchor=north west][inner sep=0.75pt] [xscale=0.75,yscale=0.75]	{\(D_{2}'=D_{2}^{\sim}\)};
			\draw [color={rgb, 255:red, 144; green, 19; blue, 254}, draw opacity=1]   (562,191) -- (542,282);
			\draw (541.87,289.84) node [anchor=north west][inner sep=0.75pt] [color={rgb, 255:red, 144; green, 19; blue, 254 }, opacity=1, xscale=0.75, yscale=0.75] {\(F'\)};
			\draw  [fill={rgb, 255:red, 0; green, 0; blue, 0}, fill opacity=1] (544,262) .. controls (544,261) and (545,260) .. (546,260) .. controls (547,260) and (548,261) .. (548,262) .. controls (548,263) and (547,263.97) .. (546,263.97) .. controls (545,263.97) and (544,263) .. (544,262) -- cycle;
			\draw (554.6,256) node [anchor=north west][inner sep=0.75pt] [xscale=0.75,yscale=0.75] {\(p\)};

			\draw (238,367) .. controls (288,391) and (363,343) .. (399.6,377);
			\draw (356,371) node [anchor=north west][inner sep=0.75pt] [xscale=0.75,yscale=0.75] {\(Z\)};
			\draw  [fill={rgb, 255:red, 0; green, 0; blue, 0}, fill opacity=1] (308,370) .. controls (308,369) and (309,368) .. (309.69,368) .. controls (310.8,368) and (312,369) .. (312,370) .. controls (312,371) and (310.8,372) .. (309.69,372) .. controls (309,372) and (308,371) .. (308,370) -- cycle;
			\draw (303,343.54) node [anchor=north west][inner sep=0.75pt] [xscale=0.75,yscale=0.75] {\(q\)};

			\draw (156,23) .. controls (206,47) and (280,-1) .. (317,33);
			\draw (141,127) .. controls (191,151) and (265,103) .. (302,137);
			\draw (156,23) -- (141,127);
			\draw (317,33) -- (302,137);
			\draw (145,98) .. controls (181,98) and (286.87,62) .. (312.87,66);
			\draw (319,55) node [anchor=north west][inner sep=0.75pt] [xscale=0.75,yscale=0.75] {\(D_{1}^{\sim}\)};
			\draw (148.2,81) .. controls (185,86) and (289,72) .. (309,90);
			\draw (314.8,83) node [anchor=north west][inner sep=0.75pt] [xscale=0.75,yscale=0.75] {\(A^{\sim}\)};
			\draw (149,66.7) .. controls (197,53) and (254.64,97) .. (306.64,110);
			\draw (310.8,106) node [anchor=north west][inner sep=0.75pt] [xscale=0.75,yscale=0.75] {\(D_{2}^{\sim}\)};
			\draw [color={rgb, 255:red, 245; green, 166; blue, 35}, draw opacity=1]   (243,32) -- (223,123);
			\draw (245,40) node [anchor=north west][inner sep=0.75pt] [color={rgb, 255:red, 245; green, 166; blue, 35 }, opacity=1, xscale=0.75, yscale=0.75] {\(E\)};
			\draw (220,97) -- (256,117);
			\draw (258,106) node [anchor=north west][inner sep=0.75pt] [xscale=0.75,yscale=0.75] {\(F^{\sim}\)};

			\draw (183,300) -- (231,348);
			\draw [shift={(233,349.64)}, rotate = 225.1] [color={rgb, 255:red, 0; green, 0; blue, 0}][line width=0.75] (11,-3) .. controls (7,-1) and (3,0) .. (0,0) .. controls (3,0) and (7,1) .. (11,3);
			\draw (188.93,324.54) node [anchor=north west][inner sep=0.75pt] [xscale=0.75,yscale=0.75] {\(f\)};
			\draw (502,126.97) -- (515,173);
			\draw [shift={(515.2,175)}, rotate = 254] [color={rgb, 255:red, 0; green, 0; blue, 0}][line width=0.75] (11,-3) .. controls (7,-1) and (3,0) .. (0,0) .. controls (3,0) and (7,1) .. (11,3);
			\draw (516,129) node [anchor=north west][inner sep=0.75pt] [xscale=0.75,yscale=0.75] {blowdown of \(F^{\sim}\) and then \(E^{\sim}\)};
			\draw (212,145) -- (185,175);
			\draw [shift={(183,176.64)}, rotate = 312] [color={rgb, 255:red, 0; green, 0; blue, 0}][line width=0.75] (11,-3) .. controls (7,-1) and (3,0) .. (0,0) .. controls (3,0) and (7,1) .. (11,3);
			\draw (472,301) -- (424,347.57);
			\draw [shift={(422.93,349)}, rotate = 316] [color={rgb, 255:red, 0; green, 0; blue, 0}][line width=0.75] (11,-3) .. controls (7,-1) and (3,0) .. (0,0) .. controls (3,0) and (7,1) .. (11,3);
			\draw (455.6,327.87) node [anchor=north west][inner sep=0.75pt] [xscale=0.75,yscale=0.75] {\(f'\)};
			\draw (375,30) -- (337,38.2);
			\draw [shift={(335,38.64)}, rotate = 347] [color={rgb, 255:red, 0; green, 0; blue, 0}][line width=0.75] (11,-3) .. controls (7,-1) and (3,0) .. (0,0) .. controls (3,0) and (7,1) .. (11,3);

			\draw (553,311) node [anchor=north west][inner sep=0.75pt] [xscale=0.75,yscale=0.75] {$ \begin{array}{l}
						F'=G^{\sim} \\
						p=\text{image of } E^{\sim}
					\end{array}$};
		\end{tikzpicture}
	\end{center}

	\medskip

	\textbf{Case \Romannum{3}:}
	\(F\) is a component of one of \(D_1,D_2,A\).
	(Now if \(F\) is a component of two of \(D_1,D_2,A\),
	then we are done in Case \Romannum{1}).
	Say \(F\) is a component of \(D_1\).
	Then \(D_2,A\) are tangent.
	Then the stable minimal model is obtained as follows:

	
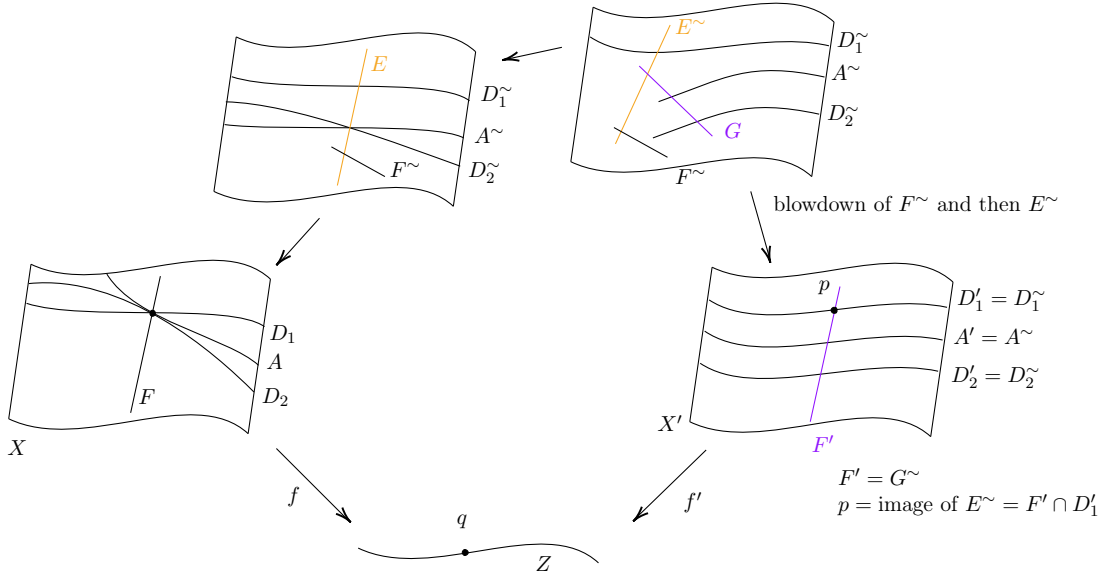
\begin{center}
		\begin{tikzpicture}[x=0.75pt,y=0.75pt,yscale=-0.75,xscale=0.75]
			\draw (396,3) .. controls (446,27) and (520.6,-21) .. (558,13);
			\draw (381,107) .. controls (431,131) and (505.6,83) .. (543,117);
			\draw (396,3) -- (381,107);
			\draw (558,13) -- (543,117);
			\draw (394,23) .. controls (433,49) and (503.6,17) .. (554.6,29);
			\draw (558.63,18.64) node [anchor=north west][inner sep=0.75pt] [xscale=0.75,yscale=0.75] {\(D_{1}^{\sim}\)};
			\draw (440,66.6) .. controls (484.6,50.6) and (500.6,39) .. (551.6,50);
			\draw (554.74,41) node [anchor=north west][inner sep=0.75pt] [xscale=0.75,yscale=0.75] {\(A^{\sim}\)};
			\draw (436.2,91) .. controls (480.6,75) and (497,64) .. (548,75);
			\draw (552,68) node [anchor=north west][inner sep=0.75pt] [xscale=0.75,yscale=0.75] {\(D_{2}^{\sim}\)};
			\draw [color={rgb, 255:red, 245; green, 166; blue, 35 }, draw opacity=1]   (447.8,15) -- (411,95);
			\draw (450,10) node [anchor=north west][inner sep=0.75pt] [color={rgb, 255:red, 245; green, 166; blue, 35 }, opacity=1, xscale=0.75, yscale=0.75] {\(E^{\sim}\)};
			\draw [color={rgb, 255:red, 144; green, 19; blue, 254}, draw opacity=1]   (427,42.61) -- (476.2,89.8);
			\draw (483,80) node [anchor=north west][inner sep=0.75pt] [color={rgb, 255:red, 144; green, 19; blue, 254 }, opacity=1, xscale=0.75, yscale=0.75] {\(G\)};
			\draw (410,84) -- (446,104);
			\draw (450,111) node [anchor=north west][inner sep=0.75pt] [xscale=0.75,yscale=0.75] {\(F^{\sim}\)};

			\draw (1,293) node [anchor=north west][inner sep=0.75pt] [xscale=0.75,yscale=0.75] {\(X\)};
			\draw (18,176) .. controls (68,200) and (142,152) .. (179,186);
			\draw (3,280) .. controls (53,304) and (127,256) .. (164,290);
			\draw (18,176) -- (3,280);
			\draw (179,186) -- (164,290);
			\draw (177,215.5) node [anchor=north west][inner sep=0.75pt] [xscale=0.75,yscale=0.75] {\(D_{1}\)};
			\draw (15,202.3) .. controls (53,214.7) and (152,201.5) .. (175,218);
			\draw (16.6,189) .. controls (80.6,182.6) and (141,235) .. (167.91,262);
			\draw (172,258) node [anchor=north west][inner sep=0.75pt] [xscale=0.75,yscale=0.75] {\(D_{2}\)};
			\draw (69,182.6) .. controls (88.6,214.6) and (150,225) .. (171,244);
			\draw (175,235) node [anchor=north west][inner sep=0.75pt] [xscale=0.75,yscale=0.75] {\(A\)};
			\draw (105,184) -- (85,276);
			\draw (89,259) node [anchor=north west][inner sep=0.75pt] [xscale=0.75,yscale=0.75] {\(F\)};
			\draw  [fill={rgb, 255:red, 0; green, 0; blue, 0}, fill opacity=1] (98,209) .. controls (98,208) and (98.5,207) .. (99.6,207) .. controls (101,207) and (101.6,208) .. (101.6,209) .. controls (101.6,210) and (101,211) .. (99.6,211) .. controls (98.5,211) and (98,210) .. (98,209) -- cycle;

			\draw (438,276) node [anchor=north west][inner sep=0.75pt] [xscale=0.75,yscale=0.75] {\(X'\)};
			\draw (476,178) .. controls (526,202) and (600.6,154) .. (638,188);
			\draw (461,282) .. controls (511,306) and (585.6,258) .. (623,292);
			\draw (476,178) -- (461,282);
			\draw (638,188) -- (623,292);
			\draw (474,200) .. controls (512,225) and (583,194) .. (634,205);
			\draw (639.6,191) node [anchor=north west][inner sep=0.75pt] [xscale=0.75,yscale=0.75] {\(D_{1}'=D_{1}^{\sim}\)};
			\draw (471,221) .. controls (509,247) and (580,216) .. (631,227);
			\draw (637,217) node [anchor=north west][inner sep=0.75pt] [xscale=0.75,yscale=0.75] {\(A'=A^{\sim}\)};
			\draw (468,242.64) .. controls (506,268) and (577,236.89) .. (628,248);
			\draw (635.5,242.84) node [anchor=north west][inner sep=0.75pt] [xscale=0.75,yscale=0.75] {\(D_{2}'=D_{2}^{\sim}\)};
			\draw [color={rgb, 255:red, 144; green, 19; blue, 254}, draw opacity=1]   (562,191) -- (542,282);
			\draw (541.87,289.84) node [anchor=north west][inner sep=0.75pt] [color={rgb, 255:red, 144; green, 19; blue, 254 }, opacity=1, xscale=0.75, yscale=0.75] {\(F'\)};
			\draw  [fill={rgb, 255:red, 0; green, 0; blue, 0}, fill opacity=1] (556,207) .. controls (556,206) and (557,205) .. (558,205) .. controls (560,205) and (560,206) .. (560,207) .. controls (560,208) and (560,209) .. (558,209) .. controls (557,209) and (556,208) .. (556,207) -- cycle;
			\draw (546.6,186) node [anchor=north west][inner sep=0.75pt] [xscale=0.75,yscale=0.75] {\(p\)};

			\draw (238,367) .. controls (288,391) and (363,343) .. (399.6,377);
			\draw (356,371) node [anchor=north west][inner sep=0.75pt] [xscale=0.75,yscale=0.75] {\(Z\)};
			\draw  [fill={rgb, 255:red, 0; green, 0; blue, 0}, fill opacity=1] (308,370) .. controls (308,369) and (309,368) .. (309.69,368) .. controls (310.8,368) and (312,369) .. (312,370) .. controls (312,371) and (310.8,372) .. (309.69,372) .. controls (309,372) and (308,371) .. (308,370) -- cycle;
			\draw (303,343.54) node [anchor=north west][inner sep=0.75pt] [xscale=0.75,yscale=0.75] {\(q\)};

			\draw (156,23) .. controls (206,47) and (280,-1) .. (317,33);
			\draw (141,127) .. controls (191,151) and (265,103) .. (302,137);
			\draw (156,23) -- (141,127);
			\draw (317,33) -- (302,137);
			\draw (153,49.8) .. controls (191,62.2) and (290.6,49) .. (312.87,66);
			\draw (319,55) node [anchor=north west][inner sep=0.75pt] [xscale=0.75,yscale=0.75] {\(D_{1}^{\sim}\)};
			\draw (149,66.7) .. controls (195,67.8) and (267,95.8) .. (306.64,110);
			\draw (310.8,106) node [anchor=north west][inner sep=0.75pt] [xscale=0.75,yscale=0.75] {\(D_{2}^{\sim}\)};
			\draw (148.2,82) .. controls (185,86.85) and (287.8,79.8) .. (309,91);
			\draw (314.8,83) node [anchor=north west][inner sep=0.75pt] [xscale=0.75,yscale=0.75] {\(A^{\sim}\)};
			\draw [color={rgb, 255:red, 245; green, 166; blue, 35}, draw opacity=1]   (244,31) -- (224,123);
			\draw (245,35) node [anchor=north west][inner sep=0.75pt] [color={rgb, 255:red, 245; green, 166; blue, 35 }, opacity=1, xscale=0.75, yscale=0.75] {\(E\)};
			\draw (220,97) -- (256,117);
			\draw (258,106) node [anchor=north west][inner sep=0.75pt] [xscale=0.75,yscale=0.75] {\(F^{\sim}\)};

			\draw (183,300) -- (231,348);
			\draw [shift={(233,349.64)}, rotate = 225.1] [color={rgb, 255:red, 0; green, 0; blue, 0}][line width=0.75] (11,-3) .. controls (7,-1) and (3,0) .. (0,0) .. controls (3,0) and (7,1) .. (11,3);
			\draw (188.93,324.54) node [anchor=north west][inner sep=0.75pt] [xscale=0.75,yscale=0.75] {\(f\)};
			\draw (502,126.97) -- (515,173);
			\draw [shift={(515.2,175)}, rotate = 254] [color={rgb, 255:red, 0; green, 0; blue, 0}][line width=0.75] (11,-3) .. controls (7,-1) and (3,0) .. (0,0) .. controls (3,0) and (7,1) .. (11,3);
			\draw (516,129) node [anchor=north west][inner sep=0.75pt] [xscale=0.75,yscale=0.75] {blowdown of \(F^{\sim}\) and then \(E^{\sim}\)};
			\draw (212,145) -- (185,175);
			\draw [shift={(183,176.64)}, rotate = 312] [color={rgb, 255:red, 0; green, 0; blue, 0}][line width=0.75] (11,-3) .. controls (7,-1) and (3,0) .. (0,0) .. controls (3,0) and (7,1) .. (11,3);
			\draw (472,301) -- (424,347.57);
			\draw [shift={(422.93,349)}, rotate = 316] [color={rgb, 255:red, 0; green, 0; blue, 0}][line width=0.75] (11,-3) .. controls (7,-1) and (3,0) .. (0,0) .. controls (3,0) and (7,1) .. (11,3);
			\draw (455.6,327.87) node [anchor=north west][inner sep=0.75pt] [xscale=0.75,yscale=0.75] {\(f'\)};
			\draw (375,30) -- (337,38.2);
			\draw [shift={(335,38.64)}, rotate = 347] [color={rgb, 255:red, 0; green, 0; blue, 0}][line width=0.75] (11,-3) .. controls (7,-1) and (3,0) .. (0,0) .. controls (3,0) and (7,1) .. (11,3);

			\draw (553,311) node [anchor=north west][inner sep=0.75pt] [xscale=0.75,yscale=0.75] {$\begin{array}{l}
						F'=G^{\sim} \\
						p=\text{image of } E^{\sim} =F'\cap D_{1}'
					\end{array}$};
		\end{tikzpicture}
	\end{center}
	If \(F\) is a component of \(D_2\) (resp.\ \(A\)),
	then the picture is similar with \(p=D_2'\cap F'\) (resp.\ \(p'=A'\cap F'\)).

	\medskip

	(3)
	We consider the fibre of \(M_Z(2,2)\to \Hilb_Z^2\).
	The fibres over points corresponding to reduced divisors, are \(\mathbb{P}^1\times\mathbb{P}^1\)
	by \cref{exa:reduced_cokernel_divisor}.
	The fibre over a point corresponding to \(2q\) is
	\[
		M_{2q}=\{[\mathcal{E},s]\in M_Z(2,2) \mid \coker(s) \div = 2q\}
	\]
	which has a stratification as in \cref{thm:stable_rank_2_over_curve_irreducible_cokernel},
	\[
		M_{2q}=G_0\cup G_2.
	\]
	Here \(G_0\) is one point corresponding to the case
	when the fibre \(F\) over \(q\) is a component of both \(D_1\) and \(D_2\).
	On the other hand \(G_2\) parametrising the case when \(D_1,D_2\) have no common component.
	By the proof of \cref{thm:stable_rank_2_over_curve_irreducible_cokernel},
	\(G_2\) admits a morphism
	\[
		G_2\longrightarrow \mathbb{P}^1
	\]
	whose fibres are \(\mathbb{A}^1\).
	This is seen from the process of going from \(X\) to \(X'\).
\end{example}

\begin{example}[Degree two on \(\mathbb{P}^1\), algebraic treatment]
	\hfill

	\medskip
	(1)
	Then \(M_Z(2,2)\simeq\Quot(\mathcal{O}_{Z}^2,2)\) and the latter parametrising embeddings
	\[
		\mathcal{K} \subseteq \mathcal{O}_{Z}\oplus \mathcal{O}_{Z}
	\]
	where either
	\[
		\mathcal{K}=\mathcal{O}_{Z}(-1)\oplus \mathcal{O}_{Z}(-1) \quad \text{or} \quad
		\mathcal{K}=\mathcal{O}_{Z}\oplus \mathcal{O}_{Z}(-2).
	\]
	Tensoring with \(\mathcal{O}_{Z}(1)\),
	\[
		\mathcal{K}(1) \subseteq \mathcal{O}_{Z}(1)\oplus \mathcal{O}_{Z}(1).
	\]
	In the first case, \(\mathcal{K}\) can be recovered from
	\begin{equation}
		H^{0}(Z,\mathcal{K}(1)) \subseteq H^{0}(Z,\mathcal{O}_{Z}(1)\oplus \mathcal{O}_{Z}(1))
	\end{equation}
	but not in the second case.
	In any case,
	the later inclusion gives a birational morphism
	\[
		g\colon\Quot(\mathcal{O}_{Z}^2,2)\longrightarrow \Gr(2,4)
	\]
	and \(\Gr(2,4)\subseteq \mathbb{P}^5\) is a quadric hypersurface.
	So we get a diagram
	\begin{equation}\label{eq:resolution}
		\xymatrix{
		& M_Z(2,2) \ar[ld]_g \ar[rd]^{\pi} \\
		\Gr(2,4) \ar@{-->}[rr]_{\text{rational map}} & & \Hilb_Z^2 \simeq \mathbb{P}^2.
		}
	\end{equation}

	Think of points in \(\Gr(2,4)\) as (projective) lines in \(\mathbb{P}(V\otimes S^1)\)
	where \(V=H^0(\mathcal{O}^{\oplus 2})\) and \(S^{\ell}=H^0(\mathbb{P}^1,\mathcal{O}(\ell))\).
	There are three types of lines in \(\mathbb{P}(V\otimes S^1)\).
	\begin{enumerate}[label=(\roman*),leftmargin=*]
		\item \(\mathbb{P}(v\otimes S^1)\) that generates \(v\otimes S^{\ell}\subseteq V\otimes S^{\ell}\) of codimension \(\ell+1\).
		\item \(\mathbb{P}(V\otimes f(x,y))\) for some linear function \(f\neq 0\).
		      By change of coordinates, we may assume \(f(x,y)=x\).
		      In \(V\otimes S^{\ell}\), it generates (a subspace isomorphic to) \(V\otimes S^{\ell-1}\)
		      which has codimension \(2\),
		      where its complement space is generated by \(V\otimes y^{\ell}\).
		\item line intersects Segre quadric in \(2\) points.
		      We may assume it is represented by the subspace spanned by \(v_1\otimes x\) and \(v_2\otimes y\) (as \(2\times 2\) matrices).
		      In \(V\otimes S^{\ell}\), it generates a subspace spanned by \(v_1\otimes x g(x,y)\), \(v_2\otimes x g(x,y)\),
		      which also has codimension \(2\),
		      where its complement space is spanned by \(v_1\otimes y^{\ell}\) and \(v_2\otimes x^{\ell}\).
	\end{enumerate}

	Now the non-injective case is (i),
	i.e., when \(H^0(\mathcal{K}(1))=\langle v\rangle\times S^1\hookrightarrow\Gamma(V\otimes S^1)\simeq V\otimes S^1\).
	This happens when \(H^0(\mathcal{K}|_{k\langle v\rangle})\subseteq H^0(V\otimes\mathcal{O}_Z)=V\).
	That is, \(\mathcal{K}=\mathcal{O}_Z\oplus\mathcal{O}_Z(-2)\).
	The point in \(\Gr(2,4)\) has no information about the embeddings \(\mathcal{O}_Z(-2)\hookrightarrow \mathcal{O}_Z^2\)
	and that can be added in after tensoring \(\mathcal{K}\hookrightarrow\mathcal{O}_Z^2\)
	by \(\mathcal{O}_Z(2)\) and taking global sections.
	That will be a subspace
	\[
		\Big\langle\langle v\rangle\otimes S^2+\langle w\rangle\Big\rangle\subseteq V\otimes S^2,
	\]
	where \(\langle w\rangle\) is the image of \(H^0(\mathcal{O}_Z(-2)\otimes\mathcal{O}_Z(2))\).
	Note that \(\langle w\rangle\) is a well-defined line in \((V/\langle v\rangle)\otimes S^2\),
	i.e., a point of \(\mathbb{P}^2\simeq \mathbb{P}(S^2)\).
	So the special points of \(\Quot(\mathcal{O}_{Z}^2,2)\) form a \(\mathbb{P}(S^2)\)-bundle over \(\mathbb{P}(V)\).
	Indeed, it is isomorphic to \(\mathbb{P}^2\times\mathbb{P}^1\);
	see the calculation below and also \cref{rem:the_edge_case}.
	So the birational morphism \(g\) is blowup of a copy of \(\mathbb{P}^1=\mathbb{P}(V)\)
	\cite[Corollary~4.11]{andreatta1993note}.
	Thus the diagram \eqref{eq:resolution} is the resolution of indeterminacies of the lower rational map.
	It is possible to write down this map explicitly in coordinate.

	To compute the normal bundle of
	\(\mathbb{P}(V)\simeq \mathbb{P}^1\hookrightarrow G\coloneqq\Gr(2,V\otimes S^1)\) explicitly,
	use the Pl\"{u}ker map
	\[
		G\hooklongrightarrow \mathbb{P}^5 = \mathbb{P}(\wedge^2(V\otimes S^1))
	\]
	and \(\wedge^2(V\otimes S^1)=(\wedge^2V\otimes S^2)\oplus(\Sym^2V\otimes\wedge^2 S^1)\).
	On one hand, the Pl\"{u}ker coordinates of \(\mathbb{P}(V)\subseteq G\) is \([0:x_1^2:x_1x_2:x_1x_2:x_2^2:0]\),
	which is not contained in any plane in \(G\),
	where \(x_0,x_1\) are the homogeneous coordinates of \(\mathbb{P}(V)\).
	On the other hand, a subspace \(W\subseteq V\otimes S^1\) has form \(v\otimes S^1\) if and only if the projection
	from \(\wedge^2 W\) to the first factor \(\wedge^2 V\otimes S^2\) is zero.
	The space \(\wedge^2 w\) for \([w]\in G\subseteq \mathbb{P}^5\) is the fibre of determinant of
	the universal subbundle
	which is \(\mathcal{O}_{\mathbb{P}^5}(-1)|_G\eqqcolon \mathcal{O}_{G}(-1)\).
	So \(\mathbb{P}(V)\subseteq G\) is the zero locus of a canonical bundle homomorphism
	\[
		\mathcal{O}_{G}(-1)\longrightarrow\mathcal{O}_{G}\otimes\wedge^2V\otimes S^2.
	\]
	Alternatively it is the zero locus for a section of \(\mathcal{O}_G(1)\otimes\wedge^2V\otimes S^2\)
	and the normal bundle is isomorphic to
	\[
		(\mathcal{O}_G(1)\otimes\wedge^2V\otimes S^2)|_{\mathbb{P}(V)}
	\]
	which is \(\mathcal{O}_{\mathbb{P}(V)}(2)^{\oplus 3}\),
	noting that \(\mathbb{P}(V)\) is a conic in \(\mathbb{P}^5\).

	Since \(g\) is a blowup, the canonical divisor of \(M_Z(2,2)\) is linearly equivalent to
	\(g^*K_G-2E\sim 4g^*H|_G-2E\), where \(E\) is the \(g\)-exceptional divisor,
	\(H\subseteq\mathbb{P}^5\) is a hyperplane.
	Let \(\ell\subseteq G\) be a line that intersects \(\mathbb{P}(V)\) at a point \(q\) (with multiplicity one).
	By the blowup formula \cite[Theorem~6.7]{fulton1998intersection},
	\(g^*[\ell]=[\widetilde{\ell}]+[L]\) in the Chow group \(A_1(M_Z(2,2))\),
	where \(\widetilde{\ell}\) is the proper transform of \(\ell\) and
	\(L\subseteq \mathbb{P}^2\subseteq M_Z(2,2)\) is a line in the fibre of \(g\) over \(q\).
	Since \(\cNE(M_Z(2,2))\) is generated by the classes of \(\widetilde{\ell}\) and \(L\),
	Kleiman's criterion implies that \(4g^{*}H|_G-2E\) is ample and hence \(M_Z(2,2)\) is Fano;
	see also \cite[Theorem~7]{mukai1989biregular}.

	\medskip
	(2)
	We calculate the fibres of \(\pi\).
	We need to parametrise quotients
	\[
		\mathcal{O}_{Z}^2\longrightarrow \mathcal{L}
	\]
	with \(\mathcal{L}\) of rank zero whose divisor if of degree \(2\).

	First consider the case when \(\mathcal{L}\) is supported at distinct points \(q_1,q_2\).
	By \cref{thm:moduli_of_product,rem:reduced_cokernel_divisor_algebraic},
	such quotients are parametrised by \(\mathbb{P}^1\times\mathbb{P}^1\).
	This agrees with the geometric picture above.

	Now consider fibre over \(2q\),
	that is, when \(\mathcal{L}\) is supported at one point \(q\).
	By \cref{rem:reduced_cokernel_divisor_algebraic},
	it is enough to parametrise quotients
	\[
		k[t]/\langle t^2 \rangle\oplus k[t]/\langle t^2 \rangle \longrightarrow L
	\]
	where \(L\) is a \(k[t]/\langle t^2 \rangle\)-module of length \(2\).
	Such \(L\) is either \(k\oplus k\) or \(k[t]/\langle t^2\rangle\).

	First consider \(L=k[t]/\langle t^2\rangle\).
	A quotient is determined by the choice of
	\[
		e=a_1+b_1t, \quad  h=a_2+b_2t \in k[t]/\langle t^2\rangle
	\]
	such that either \(a_1\neq 0\) or \(a_2\neq 0\).
	Then kernel of the quotient is
	\[
		\{(f,g)\in k[t]/\langle t^2\rangle\oplus k[t]/\langle t^2\rangle \mid fe+gh=0\}.
	\]
	If \(a_1\neq 0\), then \(e\) is invertible in \(k[t]/\langle t^2\rangle\),
	so if \((f,g)\) is in the kernel,
	then
	\[
		f=-e^{-1}gh,
	\]
	so
	\[
		(f,g)=(-e^{-1}gh,g)=e^{-1}g(-h,e).
	\]
	If \(a_2\neq 0\), we can see
	\[
		(f,g)=h^{-1}f(h,-e).
	\]
	So in any case the kernel is \(\langle (-h,e)\rangle\),
	i.e., the submodule generated by \((-h,e)\).


	When \(e\) is invertible
	\[
		\langle (-h,e)\rangle = \langle (-e^{-1}h,1)\rangle
	\]
	where
	\[
		-e^{-1}h\eqqcolon m_1+m_2t
	\]
	is uniquely determined.
	Thus such kernels are parametrised by the points of \(\mathbb{A}^1\times\mathbb{A}^1\) with coordinates \(m_1,m_2\).
	Similarly, when \(h\) is invertible,
	the corresponding kernel is \(\langle (1,-h^{-1}e)\rangle\)
	where
	\[
		-h^{-1}e\eqqcolon l_1+l_2t.
	\]
	Such kernels are parametrised by \(\mathbb{A}^1\times\mathbb{A}^1\) with coordinates \(l_1,l_2\).

	When both \(e,h\) are invertible,
	\[
		\langle (m_1+m_2t,1)\rangle = \langle (1,l_1+l_2t)\rangle
	\]
	so
	\[
		(m_1+m_2t)(l_1+l_2t)=1 \text{ in } k[t]/\langle t^2\rangle.
	\]
	This means the two \(\mathbb{A}^1\times\mathbb{A}^1\) are glued
	along \(\mathbb{A}^1\setminus \{0\}\times\mathbb{A}^1\) by the isomorphism
	\begin{align*}
		\mathbb{A}^1\setminus\{0\}\times\mathbb{A}^1 & \longrightarrow \mathbb{A}^1\setminus\{0\}\times\mathbb{A}^1 \\
		(m_1,m_2)                                    & \longmapsto \Big(\frac{1}{m_1},-\frac{m_2}{m_1^2}\Big).
	\end{align*}
	The union of the two copies \(\mathbb{A}^1\times\mathbb{A}^1\) under the above isomorphism
	is a variety mapping to \(\mathbb{P}^1\) and with fibres \(\mathbb{A}^1\).
	This corresponds to the morphism \(G_2\to \mathbb{P}^1\) determined in the geometric discussion above.

	Consider
	\[
		V(xz+y^2) \subseteq \mathbb{P}^3
	\]
	where we consider the coordinates \(x,y,z,u\) on \(\mathbb{P}^3\).
	The points with \(x\neq 0\) are
	\[
		(1:y:-y^2,u)
	\]
	which are parametrised \(\mathbb{A}^1\times\mathbb{A}^1\) with coordinates \(y,u\).
	And the points with \(z\neq 0\) are
	\[
		(-y^2:y:1:u)
	\]
	which are parametrised \(\mathbb{A}^1\times\mathbb{A}^1\) with coordinates \(-y,u\).
	The intersection of the two sets of points is the set of points with \(x\neq 0\), \(z\neq 0\), \(y\neq 0\)
	which consists of points
	\[
		(1:y:-y^2:u) = \Big(-\frac{1}{y^2}:-\frac{1}{y}:1:-\frac{u}{y^2}\Big)
	\]
	corresponding to an isomorphism from \(\mathbb{A}^1\setminus \{0\}\times\mathbb{A}^1\) given by
	\[
		(y,u) \longmapsto \Big(\frac{1}{y}:-\frac{u}{y^2}\Big)
	\]
	with respect to the coordinates above.
	So \(V(xz+y^2)\) consists of the union of two copies of \(\mathbb{A}^1\times\mathbb{A}^1\)
	glued by the above isomorphism together with the singular point \((0:0:0:1)\).

	Now consider all the quotients
	\[
		k[t]/\langle t^2\rangle \oplus k[t]/\langle t^2\rangle \longrightarrow L
	\]
	with \(L=k\oplus k\).
	Such a quotient is given by sending \((1,0),(0,1)\) to \((a_1,b_1),(a_2,b_2)\in k^2\) such that
	\[
		\det
		\begin{bmatrix}
			a_1 & a_2 \\
			b_1 & b_2
		\end{bmatrix}
		\neq 0.
	\]
	The quotient send \(t\) to zero,
	so it factors through a surjection \(k\oplus k\to L\) which should be an isomorphism.
	So the kernel is \(\langle t\rangle\oplus \langle t\rangle\)
	and this is unique.

	Finally by \cref{lem:fibre_normal},
	all the quotients for \(L=k[t]/\langle t^2\rangle\) and \(L=k\oplus k\) are parametrised by a normal surface.
	And the above discussion shows that this normal surface is isomorphic to \(V(xz+y^2)\)
	when we remove one point from each side.
	But since both are normal surfaces, they are isomorphic.

	This completes the proof of \cref{thm:main_m22}.
\end{example}

\section{Degree three stable pairs on curves}

\subsection{Algebraic treatment}\label{sub:alg_treat_degree_3}

Assume \(Z=\mathbb{P}^1\).
We consider sheaf stable pairs \(\mathcal{O}_Z^2\to\mathcal{E}\) of rank \(2\) and degree \(3\).
Since \(\mathcal{E}\) is nef,
either \(\mathcal{E}\simeq\mathcal{O}_Z(1)\oplus\mathcal{O}_Z(2)\),
or \(\mathcal{E}\simeq\mathcal{O}_Z\oplus\mathcal{O}_Z(3)\).

We first focus on the former, which is the generic case.
Let \(\mathcal{M}\) be the moduli space of stable pairs with \(\mathcal{E}=\mathcal{O}_Z(1)\oplus\mathcal{O}_Z(2)\).

Denote \(V\coloneqq H^0(\mathcal{O}_Z(1)\oplus\mathcal{O}_Z(1))\),
\(H\coloneqq H^0(\det\mathcal{E})\simeq H^0(\mathcal{O}_{\mathbb{P}^1}(3))\),
which are both vector spaces of dimension \(4\).
We claim that there is an injective morphism \(\mathcal{M}\to\mathbb{P}(H)\times\mathbb{P}(V)\).

As in \cref{rem:the_edge_case},
a stable map \(s\colon\mathcal{O}_Z^2\to\mathcal{E}=\mathcal{O}_Z(1)\oplus\mathcal{O}_Z(2)\) is given by
\[
	M=
	\begin{pmatrix}
		\alpha_1 & \alpha_2 \\
		\beta_1  & \beta_2
	\end{pmatrix}
	\in\overline{\mathcal{M}}
\]
where \(\alpha_{i}\in H^0(\mathcal{O}_{\mathbb{P}^1}(2))\) and \(\beta_{i}\in H^0(\mathcal{O}_{\mathbb{P}^1}(1))\)
such that its determinant is a non-zero vector in \(H=H^0(\det\mathcal{E})\).
Note that
\[
	\Aut(\mathcal{E})=(\Aut(\mathcal{G}_1)\times\Aut(\mathcal{G}_2))\ltimes\Hom(\mathcal{G}_1,\mathcal{G}_2)
	=(k^{*}\times k^*)\ltimes\Hom(\mathcal{O}_Z(1),\mathcal{O}_Z(2))
\]
acts on \(\overline{\mathcal{M}}\) by matrix multiplication

%
%
%
%


\[
	\begin{pmatrix}
		\gamma_1 & \varphi  \\
		0        & \gamma_2
	\end{pmatrix}
	\begin{pmatrix}
		\alpha_1 & \alpha_2 \\
		\beta_1  & \beta_2
	\end{pmatrix}
	=
	\begin{pmatrix}
		\gamma_1 \alpha_1 + \varphi \beta_1 & \gamma_2 \alpha_2 + \varphi \beta_2 \\
		\gamma_2 \beta_1                    & \gamma_2 \beta_2
	\end{pmatrix}
\]
where \(\gamma_i\in k^*\) and \(\varphi\in\Hom(\mathcal{O}_Z(1),\mathcal{O}_Z(2))\).
We still see that the determinant \(\det M=\alpha_1\beta_2-\alpha_2\beta_1\)
is invariant up to scaling of \(\gamma_1\gamma_2\).
View \([\beta_1:\beta_2]\) as an element of \(V\simeq\mathbb{A}^4\),
which is invariant up to scaling of \(\gamma_2\).
So there is an induced morphism
\[
	\mathcal{M}\longrightarrow \mathbb{P}(H)\times\mathbb{P}(V), \quad M\longmapsto (\det M, [\beta_1:\beta_2]).
\]

Assume that \(\pi(M)=\pi(M')\) for some \(M,M'\in \overline{\mathcal{M}}\) with
\[
	M=
	\begin{pmatrix}
		\alpha_1 & \alpha_2 \\
		\beta_1  & \beta_2
	\end{pmatrix}
	\quad \text{and}\quad
	M'=
	\begin{pmatrix}
		\alpha_1' & \alpha_2' \\
		\beta_1'  & \beta_2'
	\end{pmatrix}
	.
\]
After multiplying by matrices from \(\Aut(\mathcal{E})\)
we may assume that \([\beta_1:\beta_2]=[\beta_1':\beta_2']\)
and \(\det M=\det M'\).
Then we may take \(\gamma_1=\gamma_2=1\).
In the case when \(\beta_1\) and \(\beta_2\) are proportional,
let \(\varphi\in\Hom(\mathcal{O}_Z(1),\mathcal{O}_Z(2))\) be any homomorphism satisfying \(\varphi\beta_1=\alpha_1'-\alpha_1\);
in the case when \(\beta_1\) and \(\beta_2\) are not proportional,
let \(\varphi=0\) be given by \(\varphi\beta_j=\alpha_j'-\alpha_j\) for \(j=1,2\).
Then
\[
	\begin{pmatrix}
		1 & \varphi \\
		0 & 1
	\end{pmatrix}
	M=M'
\]
and hence \(\pi\) is injective.

\medskip

Now we turn to the general setting.
Note that \(M_Z(2,3)\simeq\Quot(\mathcal{O}_{Z}^2,3)\) and the latter parametrising embeddings
\[
	\mathcal{K}\hooklongrightarrow \mathcal{O}_Z\oplus\mathcal{O}_Z
\]
where either \(\mathcal{K}=\mathcal{O}_Z(-1)\oplus\mathcal{O}_Z(-2)\) or \(\mathcal{K}=\mathcal{O}_Z\oplus \mathcal{O}_Z(-3)\).
In both cases,
\[
	H^0(\mathcal{K}(2))\hooklongrightarrow H^0(\mathcal{O}_Z(2)\oplus\mathcal{O}_Z(2)) = V\otimes S^2
\]
is a \(3\)-dimensional subspace in \(6\)-dimensional vector space,
where \(V=H^0(\mathcal{O}_Z\oplus\mathcal{O}_Z)\) and \(S^{\ell}=H^0(\mathcal{O}_Z(\ell))\).
But \(\mathcal{K}\) is globally generated only in the first (generic) case.

Assume that \(\oplus_{\ell\geq 0}S^{\ell}\simeq k[x,y]\) as graded modules.
For a \(3\)-dimensional subspace \(W\hookrightarrow V\otimes S^2\) there is a linear map
\[
	W\oplus W\longrightarrow V\otimes S^3, \quad (w_1,w_2)\longmapsto w_1x+w_2y.
\]
The image is expected to be \(H^0(\mathcal{K}(3))\) and hence of codimension \(3\)
in the \(8\)-dimensional vector space \(V\otimes S^3\).
On the Grassmannian \(G=\Gr(3,6)\) we have a universal subbundle \(\mathcal{S}\hookrightarrow\mathcal{O}_{G}\otimes V\otimes S^2\) of rank \(3\),
with a morphism \(\sigma\colon\mathcal{S}\oplus\mathcal{S}\to \mathcal{O}_{G}\otimes V\otimes S^3\)
for which we like the image of \(\sigma\) to have rank \(5\).
By general theory in \cite[Chapter~14]{fulton1998intersection},
the \emph{expected codimension} of the \emph{degeneracy locus}
\[
	D_5(\sigma)=\{x\in G\mid \rk\sigma(x)\leq 5\}
\]
in \(G\) is \(3\).
By \cite[\S~4.1]{banica1991smooth}, the degeneracy locus \(D_5(\sigma)\) has singularities along \(D_4(\sigma)\).
This corresponds to the case when \(\mathcal{K}=\mathcal{O}_Z\oplus\mathcal{O}_Z(-3)\).
In this case \(H^0(\mathcal{K}(2))=v\otimes S^2\) for some vector \(v\in V\),
which spans \(H^0(\mathcal{K})\hookrightarrow H^0(\mathcal{O}_Z\otimes V)\).
The image of \(W\oplus W\to V\otimes S^3\) for \(W=v\otimes S^2\) is \(v\otimes S^3\)
which has dimension \(4\).
To get all of \(5\)-dimensional \(H^0(\mathcal{K}(3))\hookrightarrow H^0(V\otimes\mathcal{O}(3))=V\otimes S^3\)
we need to specify an additional line in
\[
	V\otimes S^3 / v\otimes S^3 \subseteq (V/\langle v\rangle)\otimes S^3.
\]
This line is \(H^0(\mathcal{O}_Z(-3)\otimes\mathcal{O}_Z(3))\hookrightarrow H^0(\mathcal{K}(3))\),
so if \(\Quot(\mathcal{O}_{Z}^2,3)=\Quot_{1,2}\cup\Quot_{0,3}\)
then the map \(\pi\colon\Quot(\mathcal{O}_{Z}^2,3)\to G\) is injective on \(\Quot_{1,2}\) and
collapses \(\Quot_{0,3}\simeq\mathbb{P}^2\times\mathbb{P}^1\) onto \(\mathbb{P}^1=\mathbb{P}(V)\)
which is the singular locus of \(\pi(\Quot(\mathcal{O}_{Z}^2,3))\subseteq\Gr(3,6)\).

If we consider \(W\subseteq V\otimes S^2\) and want \(\dim xW+yW=5\) for \(xW+yW\subseteq V\otimes S^3\),
then this means that there exist \(w_1,w_2\in W\) such that \(w_1x=w_2y\).
As \(W\subseteq V\otimes S^2\),
there exists \(v\in V\otimes S^1\) such that \(w_1=vy\) and \(w_2=vx\).
This \(v\in V\otimes S^1\) is unique up to rescaling
since \(\dim xW \cap yW=1\) (as subspaces in \(V\otimes S^3\)).
Hence, on an open subset of \(\Quot(\mathcal{O}_{Z}^2,3)\),
we take \(v\in V\otimes S^1\) then construct \(\langle vx,vy\rangle\) in \(V\otimes S^2\),
which are always independent vectors,
and complete this to a \(3\)-dimensional subspace \(W\) by choosing a line
\[
	\langle w\rangle\subseteq V\otimes S^2 / \langle vx,vy\rangle.
\]
This is an open subset in a \(\mathbb{P}^3\)-bundle over \(\mathbb{P}^3\):
although each \(\langle v\rangle\), \(\langle w\rangle\) corresponds a point of \(\mathbb{P}^3\),
occasionally we will end up with \(\dim xW\cap yW\geq 2\) in \(V\otimes S^3\).

Suppose \(\dim xW\cap yW\geq 2\).
Then there exist \(v_1,v_2\in V\otimes S^1\) such that \(v_1x,v_1y,v_2x,v_2y\in W\).
But \(W\) is \(3\)-dimensional,
so \(\Span \langle v_1x,v_1y\rangle \cap \Span \langle v_2x,v_2y\rangle\neq 0\).
This means that there exists \(u\in V\) such that \(uS^2 \subseteq W\)
and thus \(uS^2=W\).
Such subspaces \(W\) are parametrised by locus \(\langle u\rangle\in\mathbb{P}(V)=\mathbb{P}^1\) and
their \(\pi\)-preimage in \(\Quot(\mathcal{O}_{Z}^2,3)\) is isomorphic to \(\mathbb{P}^2\times\mathbb{P}^1\).
Since \(\pi^{-1}(\mathbb{P}(V))\) is not a divisor on \(\Quot(\mathcal{O}_{Z}^2,3)\),
this \(\mathbb{P}(V)=\mathbb{P}^1\subseteq G=\Gr(3,6)\) is singular in the \(\pi\)-image of \(\Quot(\mathcal{O}_{Z}^2,3)\).

\subsection{Description the fibre of Hilbert-Chow morphism via equations}






Let \(Z=\mathbb{P}^1\).
Denote by \(F\) be the fibre of \(\pi\colon M_Z(2,3)\to\Hilb_Z^3\) over \(3q\) for some closed point \(q\in Z\).
We may assume that \(q\) is the origin.


Let \(t\) be the local coordinate of the origin in \(\mathbb{A}^1\subseteq\mathbb{P}^1\).
Then all possible embeddings \(\mathcal{K}\hookrightarrow \mathcal{O}_Z^2\) with \(\det\mathcal{K}^{\vee}=\mathcal{O}_Z(3q)\) are in bijective correspondence with \(3\)-dimensional subspaces in
\[
	(k[t]/\langle t^3\rangle)^{\oplus 2}\coloneqq M
\]
which are \(t\)-invariant;
see Section 6.4.
There are two types of such submodules \(N\subseteq M\) with \(\dim N=3\):

\begin{enumerate}[leftmargin=*]
	\item \(N=k[t]/\langle t^3\rangle\) as a \(k[t]\) module.
	      This is an open subset of \(F\) (of dimension \(3\)) corresponds to the case when \(N\) projects onto a line in \(M/tM\).
	\item \(N\simeq k[t]/\langle t^2\rangle\oplus k[t]/\langle t\rangle\) corresponds to the case when
	      \(t^2M\subsetneq N\subsetneq M\).
	      This is the singular set,
	      isomorphic to \(\mathbb{P}^1\), which is the set of lines in \(tM/t^2M\).
\end{enumerate}

Local equations can be written in local coordinates on the Grassmannian \(\Gr(3,M)\).
If \(W\subseteq M\) is a \(3\)-dimensional subspace and one chooses a splitting \(M\simeq W\oplus M/W\),
then these subspaces in the neighbourhood of \(W\) correspond to graphs of linear maps \(W\to M/W\).
For example, \(W\) itself corresponds to zero map.

Write \(M=\Span(u,ut,ut^2,v,vt,vt^2)\) and choose \(W=\Span(ut,ut^2,vt^2)\).
This corresponds to the Case (2) above and they are all isomorphic.
Then \(M/W\simeq \Span(u,v,vt)\).
If we order the basis as \((ut,ut^2,vt^2,u,v,vt)\) then subspaces near \(W\) are spanned by rows of
\[
	\begin{pmatrix}
		1 & 0 & 0 & a & b & c \\
		0 & 1 & 0 & d & e & f \\
		0 & 0 & 1 & g & h & i
	\end{pmatrix}
\]
where the right half of the matrix, i.e.,
\[
	\begin{pmatrix}
		a & b & c \\
		d & e & f \\
		g & h & i
	\end{pmatrix}
\]
is the linear map \(W\to M/W\).
By assumption, this subspace is \(t\)-invariant.
Denote the \(j\)-th row by \(R_j\), \(1\leq j\leq 4\).
Since \(t(ut+au+bv+cvt)=ut^2+aut+bvt+cvt^2\),
we have \(tR_1=(a,1,c,0,0,b)\) and hence \(tR_1=aR_1+R_2+cR_3\).
This gives equations
\[
	\begin{cases}
		-a^2-d-gc=0, \\
		-ab-e-hc=0,  \\
		b-ac-f-ci=0.
	\end{cases}
\]
%
Repeat this process, we have the full set of equations on \(9\) variables:
\begin{empheq}[left=\empheqlbrace]{alignat=2}
	& -a^2-d-gc=0, \label{eq:9-1} \\
	& -ab-e-ch=0, \nonumber \\
	& b-ac-f-ci=0, \label{eq:9-3} \\
	& -ad-fg=0, \nonumber \\
	& -bd-fh=0, \nonumber \\
	& e-cd-fi=0, \nonumber \\
	& -ag-gi=0, \label{eq:9-7} \\
	& -bg-hi=0, \label{eq:9-8} \\
	& h-cg-i^2=0. \label{eq:9-9}
\end{empheq}

In \cref{eq:9-7} we have \(g(a+i)=0\).
If \(g=0\) then \cref{eq:9-8} gives \(hi=0\) and from \cref{eq:9-9} we obtain \(h=i^2\).
So both \(h,i\) are also zero.
Remaining equations then give:
\begin{empheq}[left=\empheqlbrace]{alignat=2}
	& -a^2-d=0, \label{eq:6-1} \\
	& -ab-e=0, \nonumber \\
	& b-ac-f=0, \nonumber \\
	& -ad=0, \label{eq:6-4} \\
	& -bd=0, \nonumber \\
	& e-cd=0. \nonumber
\end{empheq}
From \cref{eq:6-1,eq:6-4} we have \(a=d=0\).
Then \(e=0\), \(b=f\) and \(c\) is a free variable.
This gives
\begin{equation}\label{eq:two_dim_matrix_space}
	\begin{pmatrix}
		1 & 0 & 0 & 0 & b & c \\
		0 & 1 & 0 & 0 & 0 & b \\
		0 & 0 & 1 & 0 & 0 & 0 \\
	\end{pmatrix}
\end{equation}
which is \(2\)-dimensional.

If \(g\neq 0\) in \cref{eq:9-7},
then \(i=-a\) from \cref{eq:9-7},
\(b=f\) from \cref{eq:9-3} and
\(h=-d\) from \cref{eq:9-1,eq:9-9}.
So we are left with variables \(a,b,c,d,e,g\) and equations
\[
	\begin{cases}
		-a^2-d-gc = 0, \\
		-ab-e+dc = 0,  \\
		-ad - bg = 0.
	\end{cases}
\]
Note that \(d,e\) can be eliminated from the first two equations
while the last one gives
\[
	-a(-a^2-gc)-bg=0 \quad \Rightarrow \quad a^3=g(b-ac)
\]
in \(\mathbb{A}^4\) with coordinates \(a,b,c,g\).
Finally collect together
\[
	\begin{cases}
		b=f, h=-d, i=-a, \\
		d = -a^2 - gc,   \\
		e = dc - ab,
	\end{cases}
\]
which gives matrix
\[
	\begin{pmatrix}
		1 & 0 & 0 & a       & b     & c  \\
		0 & 1 & 0 & -a^2-gc & dc-ab & b  \\
		0 & 0 & 1 & g       & -d    & -a
	\end{pmatrix}
\]
with relation \(a^3+agc=bg\).
In particular, we can take \(a=g=d=0\) and that recovers \cref{eq:two_dim_matrix_space} which we considered above.

Therefore, locally around a single point, the equation of \(F\) can be written as \(a^3=g(b-ac)\) in \(a,b,c,g\) coordinates
which is a \(cA_2\)-singularity.
In particular, \(F\) has canonical singularities.

\medskip
The calculation above can be generalised to higher dimension.

\begin{proposition}
	Let \(Q_n\) be the variety of \(n\)-dimensional submodules in \(M_n=(k[t]/\langle t^n\rangle)^{\oplus 2}\).
	\begin{enumerate}[leftmargin=*]
		\item \(\dim Q_n=n\);
		\item the singular locus \(Q_n^{\Sing}\simeq Q_{n-2}\) when \(n\geq 2\);
		\item the smooth locus \(Q_n^{\circ}\coloneqq Q_n\setminus Q_n^{\Sing}\) is an \(\mathbb{A}^{n-1}\)-bundle over \(\mathbb{P}^1\);
		\item if \(x\in Q_{n-2}^{\circ}\subseteq Q_{n-2}\subseteq Q_n\),
		      then in some local coordinates around \(x\),
		      \(Q_n\) is isomorphic to \(\mathbb{A}^{n-2}\times (\text{singular surface } xy=z^n)\),
		      which is a higher dimensional \(cA_{n-1}\)-singularity.
	\end{enumerate}
\end{proposition}

\medskip

Now we can give a proof of \cref{thm:main_m23}.

\begin{proof}[Proof of \cref{thm:main_m23}]
	By \cref{thm:moduli_of_product} we can assume \(Z=\mathbb{P}^1\).
	And by \cref{rem:reduced_cokernel_divisor_algebraic}
	it is enough to consider the fibre \(F\) of \(\pi\) over \(3q\).

	As in the previous section,
	it is enough to parametrise all quotients
	\[
		k[t]/\langle t^3\rangle \oplus k[t]/\langle t^3\rangle \longrightarrow L
	\]
	where \(L\) is a \(k[t]/\langle t^3\rangle\)-module with length \(3\).
	The only possibilities for \(L\) are
	\[
		\begin{cases}
			k[t]/\langle t^3\rangle \\
			k[t]/\langle t^2\rangle \oplus k.
		\end{cases}
	\]
	It is enough to consider the case
	\[
		L=k[t]/\langle t^3\rangle
	\]
	as this is the generic case by \cref{rem:mq_smooth_cod_one}.
	This corresponds to the case
	when the divisors of \(s_1,s_2\) have no common component,
	that is, \(3=n=m\) as in \cref{thm:stable_rank_2_over_curve_irreducible_cokernel}.
	Each quotient is determined by
	\[
		\begin{array}{l}
			(1,0) \longmapsto e \\
			(0,1) \longmapsto h
		\end{array}
		\quad (e \text{ or } h \text{ is invertible})
	\]
	and then the kernel of the quotient is the submodule \(\langle (-h,e)\rangle\).

	The quotients with \(e\) invertible are parametrised by \(\mathbb{A}^3\) via
	\begin{align*}
		\mathbb{A}^3  & \longrightarrow U \subseteq F                   \\
		(m_1,m_2,m_3) & \longmapsto \langle (m_1+m_2t+m_3t^2,1)\rangle.
	\end{align*}
	Similarly, the quotients with \(h\) invertible are parametrised by \(\mathbb{A}^3\) via
	\begin{align*}
		\mathbb{A}^3  & \longrightarrow V \subseteq F                   \\
		(l_1,l_2,l_3) & \longmapsto \langle (1,l_1+l_2t+l_3t^2)\rangle.
	\end{align*}
	When both \(e,h\) are invertible we have
	\begin{align*}
		 & \phantom{\;=\;} \langle (m_1+m_2t+m_3t^2,1)\rangle = \langle (1,l_1+l_2t+l_3t^2)\rangle             \\
		 & = \Big\langle\Big(1,\frac{1}{m_1}-\frac{m_2}{m_1^2}t+\frac{m_2^2-m_1m_3}{m_1^3}t^2\Big) \Big\rangle
	\end{align*}
	inducing a birational map
	\begin{align*}
		U             & \dashrightarrow V                                                                \\
		(m_1,m_2,m_3) & \longmapsto \Big(\frac{1}{m_1},\frac{m_2}{m_1^2},\frac{m_2^2-m_1m_3}{m_1^3}\Big)
	\end{align*}
	which is an isomorphism on points with non-zero first coordinates.

	The above birational map induces a birational map
	\begin{align*}
		\varphi\colon X=\mathbb{P}^3 & \dashrightarrow \mathbb{P}^3=Y                                                 \\
		(m_1:m_2:m_3:m_4)            & \longmapsto (m_1^2m_4:-m_1m_2m_4:m_2^2m_4-m_1m_3m_4:m_1^3) = (l_1:l_2:l_3:l_4)
	\end{align*}
	where we identify \(U\) with points with \(m_4\neq 0\) and \(V\) with \(l_4\neq 0\).

	Define the hyperplanes
	\[
		A_{i}: \quad m_{i}=0 \qquad \text{and} \qquad H_{i}: \quad l_{i}=0.
	\]
	By construction, \(\varphi\) gives an isomorphism
	\[
		X\setminus (A_1\cup A_4) \longrightarrow Y\setminus (H_1 \cup H_4)
	\]
	and one has \(\varphi^2=\id\).
	Also we can see that
	\[
		\varphi \text{ contracts } A_1,A_4 \quad \text{and} \quad
		\varphi^{-1} \text{ contracts } H_1,H_4
	\]
	and no other divisors are contracted.
	Pick the canonical divisor \(K_Y\) such that
	\[
		K_Y+H_1+H_2+H_3+H_4=0 \quad (=0 \text{ not just} \sim 0).
	\]
	We want to compute
	\[
		\varphi^{*}(K_Y+\sum_1^4 H_{i})
	\]
	in terms of \(K_X\) and \(A_{i}\).
	First note that
	\[
		\begin{array}{l}
			\varphi^{*}H_1 = 2A_1 + A_4,      \\
			\varphi^{*}H_2 = A_1 + A_2 + A_4, \\
			\varphi^{*}H_3 = A_4 + G,         \\
			\varphi^{*}H_4 = 3A_1,
		\end{array}
	\]
	where \(G\subseteq X\) is given by \(m_2^2-m_1m_3=0\).
	Moreover,
	\[
		\varphi^{*}K_Y = K_X + a A_1 + b A_4
	\]
	where \(K_X\) is determined as a Weil divisor by \(K_Y,\varphi\), and \(a,b\in \mathbb{Z}\).
	So
	\begin{align*}
		0 = \varphi^{*}(K_Y+\sum_1^4 H_{i}) & = K_X+a A_1+b A_4 + 2A_1+A_4 + A_1+A_2+A_4 + A_4+G + 3A_1 \\
		                                    & = K_X + (a+6)A_1 + A_2 + (b+3) A_4 + G.
	\end{align*}
	Since \(\deg G=2\) and \(\deg K_X=-4\),
	\[
		-4 + (a+6) + 1 + (b+3) +2 = 0
	\]
	hence
	\[
		(a+6) + (b+3) = 1.
	\]
	On the other hand,
	\((Y,\sum_1^4 H_{i})\) has lc singularities,
	so
	\[
		a+6\leq 1 \quad \text{and} \quad b+3\leq 1.
	\]
	Thus one of \(a+6\) and \(b+3\) is zero and the other is one.

	We claim that \(a+6=0\).
	Assume not.
	Then
	\[
		(X,A_1+A_2+G)
	\]
	is lc.
	But then
	\[
		(A_1,A_2|_{A_1}+G|_{A_1})
	\]
	is lc \cite[Theorem~5.50]{kollar1998birational} which is not the case,
	so we have \(a+6=0\) and \(b+3=1\) as claimed
	(note \(G|_{A_1}\) is a double line).
	Summarising, we have
	\begin{equation}\label{eq:pullback_K_Y}
		\varphi^{*}(K_Y+\sum_1^4 H_{i}) = K_X + A_2 + A_4 + G.
	\end{equation}
	Since the construction are symmetric for \(X,Y\),
	we also have
	\[
		(\varphi^{-1})^{*}(K_X+\sum_1^4 A_{i}) = K_Y + H_2 + H_4 + P
	\]
	for some appropriate choices of \(K_X,K_Y\) where \(P\) is a hypersurface of degree two.
	Let \(p\colon W'\to X\) and \(q\colon W'\to \) be a common resolution of \(X\) and \(Y\).
	Then we have
	\[
		p^{*}(K_X+\frac{1}{2}A_1+A_2+A_3+\frac{1}{2}A_4+\frac{1}{2}G) = q^{*}(K_Y+\frac{1}{2}H_1+H_2+H_3+\frac{1}{2}H_4+\frac{1}{2}P).
	\]
	Therefore, the log discrepancy \(a(A_1,Y,\frac{1}{2}H_1+H_2+H_3+\frac{1}{2}H_4+\frac{1}{2}P)=\frac{1}{2}\) and
	hence we can extract \(A_1\) from \(Y\) by an extremal birational contraction \(\rho\colon W\to Y\)
	where \(W\) is of Fano type \cite[Lemma~4.6]{birkar2016effectivity}.
	Then by \eqref{eq:pullback_K_Y} we have \(a(A_1,Y,\sum_1^4 H_i)=1\) and thus
	\[
		\rho^{*}(K_Y+\sum_1^4 H_{i}) = K_W + \sum_1^4 H_{i}^{\sim}
	\]
	where \(\sim\) denotes birational transform.
	Moreover,
	\[
		\begin{array}{l}
			\rho^{*}H_1 = H_1^{\sim} + 2A_1^{\sim}, \\
			\rho^{*}H_2 = H_2^{\sim} + A_1^{\sim},  \\
			\rho^{*}H_3 = H_3^{\sim},               \\
			\rho^{*}H_4 = H_4^{\sim} + 3A_1^{\sim}.
		\end{array}
		\qquad (A_2^{\sim}=H_2^{\sim} \text{ and } G^{\sim}=H_3^{\sim})
	\]

	\begin{claim}\label{clm:eff_w}
		The cone of effective divisors of \(W\) is generated by \(H_4^{\sim}\) and \(A_1^{\sim}\).
	\end{claim}

	\begin{proof}[Proof of the \cref{clm:eff_w}]
		Pick a Weil divisor \(D\geq 0\) on \(W\).
		We can write
		\[
			\rho^{*}\rho_{*}D=D+\alpha A_1^{\sim}
		\]
		for some \(\alpha\in \mathbb{Z}\).
		Letting \(\beta=\deg\rho_{*}D\) and noting \(\rho_{*}D\sim \beta H_4\),
		we get
		\[
			\beta H_4^{\sim}+3\beta A_1^{\sim} \sim D + \alpha A_1^{\sim}.
		\]
		Then the left side of
		\begin{align*}
			(\rho^{-1}\varphi)^{*} (\beta H_4^{\sim}+(3\beta-\alpha)A_1^{\sim})
			 & = \varphi^{*}H_4 - (\rho^{-1}\varphi)^{*}\alpha A_1^{\sim} \\
			 & \leq 3A_1 - \alpha A_1
		\end{align*}
		is pseudo-effective on \(X\) which shows
		\[
			\alpha \leq 3.
		\]
		Thus \(3\beta-\alpha\geq 0\) implying the claim.
		Note that if \(\beta=0\),
		then \(D=-\alpha A_1^{\sim}\) with \(-\alpha\geq 0\).
	\end{proof}

	\begin{claim}\label{clm:kodaira_dim}
		The Kodaira dimension \(\kappa(H_4^{\sim})=0\).
	\end{claim}

	\begin{proof}[Proof of the \cref{clm:kodaira_dim}]
		Assume not.
		Then \(H_4^{\sim}\) is movable and there exist \(e\in\mathbb{N}\) and \(M\geq 0\)
		such that \(e H_4^{\sim}\sim M\)
		where \(H_4^{\sim}\not\subseteq \Supp M\).
		Then from
		\[
			\rho^{*}\rho_{*}H_4^{\sim} = \rho^{*} H_4 = eH_4^{\sim}+3A_1^{\sim}
		\]
		we get
		\[
			\rho^{*}\rho_{*}M = M+3eA_1^{\sim}.
		\]
		So
		\[
			\varphi^{*}(\rho_{*}M) = R + 3e A_1 \quad \text{for some } R\geq 0.
		\]
		Moreover, \(e=\deg\rho_{*}M\),
		so \(\rho_{*}M\) is a hypersurface defined by a polynomial \(\Pi\) of degree \(e\).
		Thus \(\varphi^{*}(\rho_{*}M)\) is given by the polynomial
		\[
			\Pi(m_1^2m_4,-m_1m_2m_4,m_2^2m_4-m_1m_3m_4,m_1^3)
		\]
		of degree \(3e\).
		But then \(R=0\) and \(\varphi^{*}(\rho_{*}M)=3e A_1\).

		Since \(H_1,H_4\) are the only exceptional divisors of \(\varphi^{-1}\)
		and since \(H_4\not\subseteq \Supp\rho_{*}M\),
		we have \(\rho_{*}M=e H_1\).
		Thus \(\varphi^{*}\rho_{*}M=2e A_1 + e A_4\), a contradiction.
		So we have proved the claim.
	\end{proof}

	Now run the \(H_4^{\sim}\)-MMP and
	let \(T\) be the resulting model.
	We can run this MMP as \(W\) is of Fano type.
	Since \(W\) is \(\mathbb{Q}\)-factorial of Picard number two \cite[Lemma~4.6]{birkar2016effectivity},
	\(T\) is \(\mathbb{Q}\)-factorial of Picard number one as \(H_4^{\sim}\) is contracted,
	given \(\kappa(H_4^{\sim})=0\).
	In particular, \(T\) is a (klt) Fano variety.

	By construction,
	\(W\) contains \emph{big open subsets} of both \(U\) and \(V\), i.e., open subset with complement of codimension \(\geq 2\):
	\begin{itemize}
		\item for \(U\) we use the fact that \(U\dashrightarrow W\) contracts no divisor;
		      recall that \(U\setminus A_1\xrightarrow{\sim}V\setminus H_1\)
		      and \(A_1^{\sim}\subseteq W\);
		\item for \(V\) it is clear as \(W\to Y\) is a birational morphism.
	\end{itemize}
	In fact the complement in \(W\) of the union the two big open sets
	contains only one prime divisor, \(H_4^{\sim}\).
	Now \(W\dashrightarrow T\) is an isomorphism outside \(H_4^{\sim}\) and
	\(H_4^{\sim}\) does not intersect the mentioned big open sets.
	Thus \(T\) also contains big open subsets of \(U\) and \(V\).
	Moreover, the complement in \(T\) of the union of two big open sets, is of codimension \(\geq 2\).

	Therefore \(T\) is isomorphic to the fibre \(F_3\) of \(\pi\) because both \(T,F_3\) are normal varieties
	with isomorphic open subsets with codimension \(\geq 2\) complements.
	Now since \(F_3\) has klt singularities and \(K_{F_3}\) is Cartier by \cref{thm:main_rank_2_on_curves} (3),
	it has canonical singularities,
	along a copy of \(\mathbb{P}^1\) as in Section~\ref{sub:alg_treat_degree_3}.
\end{proof}


\printbibliography

\end{document}